\documentclass[12pt]{article}
\usepackage[authoryear]{natbib}
\usepackage{amsmath,amssymb,amsfonts,amsthm}
\usepackage{epic,eepic,epsfig,longtable}
\usepackage{multirow,verbatim}
\usepackage{array}
\usepackage{comment}
\usepackage{float}
\usepackage{epstopdf}
\usepackage{epsfig}
\usepackage{subfigure}
\usepackage{setspace}
\usepackage{color}

\makeatletter
\def\singlespace{\def\baselinestretch{1}\@normalsize}

\makeatletter
\def\singlespace{\def\baselinestretch{1}\@normalsize}

\numberwithin{equation}{section}

\newcommand{\bfm}[1]{\ensuremath{\mathbf{#1}}}

   \def\bA{\bfm A}  
\def\bb{\bfm b}   \def\bB{\bfm B}  
   \def\bC{\bfm C}  
   \def\bD{\bfm D}  
\def\be{\bfm e}   \def\bE{\bfm E}  
\def\bff{\bfm f}  \def\bF{\bfm F}  
\def\bg{\bfm g}   \def\bG{\bfm G}  
\def\bh{\bfm h}   \def\bH{\bfm H}  
   \def\bI{\bfm I}  
     
   \def\bK{\bfm K}  
     
\def\bm{\bfm m}   \def\bM{\bfm M}

   \def\bP{\bfm P}

   \def\bS{\bfm S}  
   \def\bT{\bfm T}  
\def\bu{\bfm u}   \def\bU{\bfm U}  
\def\bv{\bfm v}   \def\bV{\bfm V}  
\def\bw{\bfm w}   \def\bW{\bfm W}  
\def\bx{\bfm x}   \def\bX{\bfm X}  
   \def\bY{\bfm Y}  
\def\bz{\bfm z}   \def\bZ{\bfm Z}

\newcommand{\bfsym}[1]{\ensuremath{\boldsymbol{#1}}}

 \def\balpha{\bfsym \alpha}
 \def\bbeta{\bfsym \beta}			 
 \def\bgamma{\bfsym \gamma}             
 \def\bdelta{\bfsym {\delta}}           \def\bDelta {\bfsym {\Delta}}
 \def\bfeta{\bfsym {\eta}}              
                  
 \def\bnu{\bfsym {\nu}}
 \def\btheta{\bfsym {\theta}}           
           \def\bepsilon{\bfsym \varepsilon}
              \def\bSigma{\bfsym \Sigma}
 \def\blambda {\bfsym {\lambda}}        \def\bLambda {\bfsym {\Lambda}}

  \def\bvarepsilon{\bfsym {\varepsilon}}

 \DeclareMathOperator{\Cov}{Cov}

\DeclareMathOperator{\diag}{diag}
\DeclareMathOperator{\E}{\mathbb E}

\DeclareMathOperator{\rank}{rank}

\DeclareMathOperator{\Var}{Var}
\DeclareMathOperator{\var}{var}

\DeclareMathOperator{\tr}{tr}

\def\today{\ifcase\month\or
  January\or February\or March\or April\or May\or June\or
  July\or August\or September\or October\or November\or December\fi
  \space\number\day, \number\year}

\newdimen\biblioindent    \biblioindent=30pt

 at 8truept

\newcommand{\beq}{\begin{equation}}
  \newcommand{\eeq}{\end{equation}}
\newcommand{\beqn}{\begin{eqnarray}}
  \newcommand{\eeqn}{\end{eqnarray}}
\newcommand{\beqnn}{\begin{eqnarray*}}
  \newcommand{\eeqnn}{\end{eqnarray*}}

\def\MEAN{{\rm MEAN}}

\allowdisplaybreaks
\usepackage{verbatim}
\renewcommand{\baselinestretch}{1.66}
\numberwithin{equation}{section}
\theoremstyle{plain}
\newtheorem{thm}{Theorem}[section]

\newtheorem{lem}{Lemma}[section]
\newtheorem{cor}{Corollary}[section]
\newtheorem{prop}{Proposition}[section]
\newtheorem{ass}{Assumption}[section]
\theoremstyle{definition}
\newtheorem{exm}{Example}[section]

\newtheorem{remark}{Remark}[section]
\newtheorem{algo}{Algorithm}[section]

\newcounter{CondCounter}

\usepackage{xr}
\begin{document}

\title{\LARGE Learning Latent Factors from Diversified Projections and its Applications to Over-Estimated and Weak Factors}
  \author{  Jianqing Fan\thanks{Department of Operations Research and Financial Engineering, Princeton University, Princeton, NJ 08544, USA.   \texttt{jqfan@princeton.edu}. His research is supported by NSF grants DMS-1662139 and DMS-1712591. }\and  Yuan Liao\thanks{Department of  Economics, Rutgers University, 75 Hamilton St., New Brunswick, NJ 08901, USA. \texttt{yuan.liao@rutgers.edu}}}
\date{}

\maketitle

 \singlespacing

 \begin{abstract}
Estimations and applications of factor models often rely on the crucial condition that the number of latent factors  is consistently estimated,  which in turn also requires that  factors   be relatively strong, data are stationary and weak serial dependence, and the sample size  be fairly large, although in practical applications, one or several of these conditions may fail. In these cases it is  difficult to analyze the eigenvectors of the  data matrix.  To address this issue, we propose simple estimators of the latent factors  using cross-sectional  projections of the panel data, by  weighted averages with pre-determined  weights. These  weights are chosen to diversify away the idiosyncratic  components,  resulting in ``diversified factors".  Because the projections are conducted cross-sectionally,  they are robust to serial conditions, easy to analyze   and work even for finite length of time series. We  formally prove that  this procedure is  robust to over-estimating the number of factors, and  illustrate it  in  several applications, including post-selection inference, big data forecasts, large covariance estimation and factor specification tests.   We also recommend several choices for the diversified weights.



 \end{abstract}





\doublespacing

\section{Introduction}
\label{sec:introduction}

 Consider the following high-dimensional factor model:
\begin{equation}\label{eq1.1}
\bx_t=\bB\bff_t+\bu_t,\quad t=1,\cdots, T,
\end{equation}
where $\bx_t=(x_{1t},\cdots,x_{NT})'$ is an $N$-dimensional outcome. In addition, the model contains   $\bff_t$ as  $r$-dimensional latent factors,  $\bB=(\bb_1,\cdots,\bb_N)'$ as $N\times r$ matrix of loadings, and   $\bu_t=(u_{1t},\cdots,u_{Nt})'$ as   idiosyncratic terms.
Theoretical studies of the   model have  been crucially depending on the assumption that the number of factors, $r$, should be consistently estimated. This in turn, requires the factors be relatively strong,  data  have weak  serial  dependence, and length of time series $T$ is long.
But in practical applications, one or several of these conditions may fail to hold due to weak  signal-noise ratios and  nonstationary  or noisy data,  making  the  first $r$  eigenvalues of the sample covariance of   $\bX=(\bx_1,\cdots,\bx_T)$  empirically be not so-well separated from the  remaining ones.

A promising remedy  is to over-estimate the number of factors.  But  this approach has been quite challenging. 
 Let $R$ be the ``working number of factors" that are empirically estimated.  When $R> r$,    it is often difficult to analyze the behavior of the $(R-r)$ eigenvalues/eigenvectors. As shown in \cite{johnstone2009consistency},    these   eigenvectors can be inconsistent because their   eigenvalues are not so ``spiked".
This  creates challenges to many factor estimators, such as the popular {principal components} (PC)-estimator  \citep{connor1986performance, SW02},  and therefore brings obstacles to  applications when over estimate the number of factors.
Another difficulty is to handle the serial dependence.  As shown by \cite{bai03}, the PC-estimator is inconsistent under finite-$T$ in the presence of serial correlations and heteroskedasticity, but many forecast  applications using estimated factors favor relatively short time series, due to the concern  of nonstationarity.

This paper proposes a new method to address   issues of  over-estimating the number of factors,  small $T$, and strong serial conditions. We propose  a  simple factor estimator   that does not rely on eigenvectors. Let
$
\bW=(\bw_1,\cdots,\bw_R)
$
be a  given  exogenous (or deterministic) $N\times R$ matrix,    where each of its $R$ columns $\bw_k$   is an $N\times 1$ vector of ``diversified weights", in the sense that its strength   should be approximately equally distributed on most of its components. We propose to estimate $\bff_t$ by simply
$$
\widehat\bff_t=\frac{1}{N}\bW'\bx_t, 
$$
or more precisely, the linear space spanned by $\{\bff_t\}_{t=1}^T$ is estimated by that spanned by $\{\widehat\bff_t\}_{t=1}^T$.  By substituting
\eqref{eq1.1} into the definition, we have  \begin{equation}\label{eq1.2}
\widehat\bff_t=\underbrace{(\frac{1}{N}\bW'\bB)}_{\text{affine transform}}\bff_t+\frac{1}{N}\bW'\bu_t.
\end{equation}
Thus $\widehat\bff_t$ (consistently) estimates $\bff_t$ up to an $R\times r$ affine transform, with $\be_t:=\frac{1}{N}\bW'\bu_t$ as the  estimation error. The assumption that $\bW$ should be diversified ensures that as $N\to\infty$, $\be_t$ is ``diversified away" (converging to zero in probability).

We call the new factor estimator as ``diversified factors", which reduces the dimension of $\bx_t$ through diversified projections. Because of the clean expansion    (\ref{eq1.2}), the mathematics  for theoretical analysis   is  much simpler than most benchmark  estimators. We show that $\widehat\bff_t$ leads to valid inferences   in several factor-augmented models   so long as $R\geq r$. Therefore, we formally justify that the use of factor models is robust to over-estimating the number of factors. In particular,  we admit   $r=0$ but  $R\geq 1$ as a special case. That is, the inference is still valid  even if there are  no common factors present,  but we nevertheless take out estimated factors (for insurance). Furthermore,  the projection  is conducted on cross-sections, so is not sensitive to serial conditions.  
We study several applications in detail, including the  post-selection inference, big data forecasts, high-dimensional covariance estimation, and factor specification tests.

One of the key assumptions imposed is that  while $\bW$   diversifies away $\bu_t$,   we have $$
\rank\left(\frac{1}{N}\bW'\bB\right)= r,
$$
and the $r$ th smallest singular value of $\frac{1}{N}\bW'\bB$ does not decay too fast.
That is,  $\bW$  should not diversify away the factor components in the time series.  This condition \textit{does not} hold if $\bW$  has more than $R-r$ columns that are nearly orthogonal to $\bB$.
 This is another motivation of using over-estimated factors: if random weights are used  the probability that more than $R-r$ columns of  $\bW$  are nearly orthogonal to the space of $\bB$ should be very small.  

To satisfy the above conditions on the weights, we rely on external information on the factor loadings, and recommend  four choices for the  weight matrix. The first choice is the individual-specific characteristics.        As documented in semi-parametric factor models, \cite{CMO,park,fan2016projected}, factor loadings are often driven by observed characteristics. When these variables are available, they can be naturally used as diversified weights.  The second choice is based on rolling window estimations. Consider  time  series forecasts. To pertain the  stationarity assumption,  we  divide the sampling periods into (I)  $t=1,...,T_0$ and (II) $t=T_0+1,...,T_0+T$, and only use the    most recent $T$ observations from period (II) to learn the latent factors for forecasts. Or consider a  time series where a structural break occurs at time $T_0$, so the most recent period (II) is of major interest. 
Assume that the loadings are correlated between the two periods, then the PC-estimated loadings from periods (I) would be a good choice of the diversified weights for period (II).
For the third recommendation, when the time series is independent of the initial observation,  we can  use transformations of $\bx_0$ as the weights.
The fourth recommended choice is to use columns of the  Walsh-Hadamard matrix from the statistical experimental design to form the diversified weights. 

 The idea of approximating    factors by   weighted averages of observations has been  applied previously in the  literature. In the asset pricing literature,   factors are created by weighted averages of  a large number of asset returns. There, the weights are also pre-determined, adapted to the filtration up to the last observation time.  In the  {common correlated effects} (CCE) literature  \citep{pesaran, chudik2011weak},  factors are created using a set of random weights to estimate the effect of observables.
 There, $R$ equals the dimensions of additionally observed regressors and the outcome variable, and   certain rank conditions about the regressors are required.  
  In the same setting,   \cite{westerlund2015cross} and \cite{karabiyik2019cce} compared the cross-sectional average and  the PC estimators, and also showed the validity of using $R>r$ number of cross-sectional averages.  Moreover, \cite{barigozzi2018consistent} proposed a different method to address the issue of over-estimating factors.   One of our  recommended weights is inspired by their approach.  \cite{MW11}  studied the problem in a panel data framework   and showed that the inference about the parameter of interest is robust to over-estimating the number of factors.
      It is not so clear if their approach is generally applicable to other factor-augmented inference problems.   
           Finally, there is a large literature   on estimating the number of factors.   See \cite{BN02, HL,  AH, li2017determining}. 

The rest of the paper is organized as follows. Section \ref{sec:2}  explains the key ideas and intuitions in details.
  Section \ref{sec:3} presents    several  applications of the diversified factors.   Section
\ref{choice} recommends several choices of the weight matrix. Section \ref{sec:sim} conducts extensive simulation studies using various models.
All technical proofs are presented in the appendix.

We  use the following notation.
For a matrix $\bA$, we use $\lambda_{\min}(\bA)$ and $\lambda_{\max}(\bA)$ to denote its smallest and largest eigenvalues. We define  the Frobenius norm $\|\bA\|_F=\sqrt{\tr(\bA'\bA)}$ and the operator norm $\|\bA\|=\sqrt{\lambda_{\max}(\bA'\bA)}$.  In addition, define projection matrices $\bM_\bA=\bI-\bP_\bA$ and  $\bP_\bA=\bA(\bA'\bA)^{-1}\bA$  when $\bA'\bA$ is   invertible.     Finally, for two (random) sequences $a_T$ and $b_T$, we write $a_T\ll b_T$ (or $b_T\gg a_T$) if $a_T=o_P(b_T)$.

\section{Factor Estimation Using Diversified Projections} \label{sec:2}

\subsection{The estimator}
   Let $R\geq r$ be a pre-determined  bounded integer that does not grow with $N$, which we call ``the working number of factors". As in practice we do not know the true  number of factors $r$,  we often take a slightly large $R$ so that $R\geq r$ is likely to hold. Let
$
\bW=(\bw_1,\cdots,\bw_R)
$
be a  user-specified  $N\times R$ matrix,  either deterministic or random but  independent of the $\sigma$-algebra generated by  $\{\bu_t: t=1,2,...\}$. Each of its $R$ columns $\bw_k = (w_{k,1},\cdots,w_{k,N})'$ ($k\leq R$) is an $N\times 1$ vector satisfying the following:
\begin{ass}[Diversified weights]\label{ass2.1}  There are  constants $0<c<C$, so that (almost surely if $\bW$ is random) as $N\to\infty$,\\
	(i) $\max_{i\leq N} |w_{k,i}|<C$.
	\\
	(ii) The $R\times R$ matrix $\frac{1}{N}\bW'\bW$ satisfies $\lambda_{\min}(\frac{1}{N}\bW'\bW)>c.$\\
	(iii) $\bW$ is independent of $\{\bu_t: t\leq T\}$.
\end{ass}
 
Construct a   factor estimator  as an  $R\times 1$ vector at each time $t $:
$$
\widehat\bff_t:= \frac{1}{N}\bW'\bx_t.
$$ 
   In financial economics applications where $\bx_t$ is a vector of asset returns, then  each component of $\widehat\bff_t$ is essentially a diversified portfolio return at time $t$ due to its linear form. The behavior of $\widehat\bff_t$ is strikingly  simple and clean.  Define an $R\times r$ matrix
$$
\bH:= \frac{1}{N}\bW'\bB.
$$
Then, it follows from the definition and \eqref{eq1.1}, we have
\begin{equation} 
\widehat\bff_t=\bH\bff_t+\frac{1}{N}\bW'\bu_t.
\end{equation}
 Therefore, $\widehat\bff_t$ estimates an affine transformation of $\bff_t$, where  $\bH$ is the  $R\times r$ transformation matrix. The estimation error   equals the diversified idiosyncratic noise
 $
 \frac{1}{N}\bw_k'\bu_t= \frac{1}{N}\sum_{i=1}^Nw_{k,i}u_{it} $ for each $k\leq R.$ 
 When $(u_{1t},\cdots,u_{Nt})$ are cross-sectionally weakly dependent,
Assumption \ref{ass2.1} ensures that  $ \frac{1}{N}\bw_k'\bu_t$  admits a cross-sectional central limit theorem. For instance, in the special case of cross-sectional independence, it is straightforward to verify  the  Lindeberg's condition under Assumption \ref{ass2.1}, and therefore as $N\to\infty$, \begin{equation}\label{eq2.3}
 \frac{1}{\sqrt{N}}\bW'\bu_t \overset{d}{\longrightarrow}\mathcal N(0, \bV), 
\end{equation}
where $\bV=\lim_{N\to\infty}\frac{1}{N}\bW'\var(\bu_t)\bW$ which is assumed to exist.

 The convergence (\ref{eq2.3}) shows that $\sqrt{N}(\widehat\bff_t-\bH\bff_t)$ is asymptotically normal for each $t\leq T$. Importantly, it holds regardless of whether $T\to\infty$,   $R=r$, or not.  It requires only that $N\to\infty$ and that the weights should be chosen to satisfy Assumption \ref{ass2.1}.
 This fact is particularly useful for analyzing short time series. 


In addition, the factor components should not be diversified away. This gives rise to the following condition on the transformation matrix.  Let $\nu_{\min}(\bH)$ and $\nu_{\max}(\bH)$ respectively denote the minimum and maximum nonzero singular values of $\bH$.

\begin{ass}\label{ass2.4} Suppose $R\geq r$. Almost surely
(i) $\rank(\bH)=r$. \\
(ii) There is $C>0,$
\begin{equation*}
	\nu^2_{\min}(\bH)\gg \frac{1}{N},  
	\quad \nu_{\max}(\bH)\leq C\nu_{\min}(\bH).  \label{fan1}
	\end{equation*}
\end{ass}
Assumption \ref{ass2.4}    requires that  $\bW$ have  at least $r$ columns  that are not orthogonal to $\bB$ so  that $\bB$ is not diversified away. This is the key assumption, but is not stringent in the context of over-estimating factors. 
 In the current setting the factor strength is measured by $\nu_{\min}(\bH)$, which is required not to decay very fast by condition (ii).
 This quantity determines the rate of convergence in  recovering the space spanned by the factors.

   Given $\widehat\bff_t$, it is straightforward to estimate the loading matrix by using the least squares:
 $$
 \widehat\bB=(\widehat\bb_1,\cdots,\widehat\bb_N)'= \sum_{t=1}^T\bx_t\widehat\bff_t'(\sum_{t=1}^T\widehat\bff_t\widehat\bff_t')^{-1}.
 $$
 We show that the $R\times R$ matrix $\frac{1}{T}\sum_{t=1}^T\widehat\bff_t\widehat\bff_t'$  is  nonsingular with probability approaching one  even when  $R> r$.  So $\widehat\bB$ is well defined.
 Finally,  $\bu_t$ can be estimated by 
 \begin{equation}\label{eq2.4}
 \widehat\bu_t=(\widehat u_{1t},\cdots,\widehat u_{Nt})= \bx_t-\widehat\bB\widehat\bff_t.
 \end{equation}

Just like the PC-estimator, the diversified projection can   estimate dynamic factor models by treating dynamic factors as static factors. In addition, it is straightforward to extend the model to allowing time-varying factor loadings, by   time-domain local smoothing before applying the diversified projection. While these extensions are straightforward, here we focus on static   and time invariant models.

   \subsection{Over-estimating the number of factors}\label{poet}

The consistent estimation for the  number of factors $r$  often requires strong conditions that may  be  violated  in finite sample. An advantage of the diversified factors is  being robust to over-estimating the number of factors in many inference problems.


We start with a heuristic discussion of the main issue in this subsection.
Recall that $\bH=\frac{1}{N}\bW'\bB$ is the $R\times r$ matrix, which is no longer a square matrix when $R>r$. In this case $\widehat\bB$ is essentially  estimating  $\bB \bH^+$, with the $r\times R$ transformation matrix  $\bH^+$ being the Moore-Penrose generalized inverse of $\bH$, defined as follows. Suppose   $\bH'$ has the following singular value decomposition:
 $$
 \bH'= \bU_H(\bD_H,0)\bE_H',\quad r\times R
 $$
 where $0$ in the above singular value matrix is present whenever $R>r$, and $\bD_H$ is an $r\times r$ diagonal matrix of the nonzero singular values. Then $\bH^{+}$ is an $r\times R$ matrix:
 $$
 \bH^+= \bU_H(\bD_H^{-1},0)\bE_H'.
 $$
It is straightforward to verify that $\bH^+\bH=\bI_r $ holds and that for estimating the common component $\bB\bff_t$ using over-estimated number of factors, we have
 \begin{equation}\label{eq3.1add}
 \widehat\bB\widehat\bff_t
=  \bB\bH^+\bH \bff_t +o_P(1)=  \bB \bff_t +o_P(1).
 \end{equation}
 where $o_P(1)$ in the above approximation can be made uniformly across elements.

However, a  key challenge of formalizing the  intuition behind (\ref{eq3.1add})  is to analyze the  invertibility of  the gram matrix $ \frac{1}{T}\sum_{t=1}^T\widehat\bff_t\widehat\bff_t'$,
 which appears in the definition of $\widehat\bB$. It is also a key ingredient in  most applications of factor-augmented models  wherever the estimated factors  are   used as  regressors.
Define
\begin{eqnarray*}
	\widehat\bS_f&=& \frac{1}{T}\sum_{t=1}^T\widehat\bff_t\widehat\bff_t',\quad  \bS_f=\bH\frac{1}{T}\sum_{t=1}^T\bff_t\bff_t'\bH',
	\end{eqnarray*}
where $\bS_f$ is the population analogue of $\widehat\bS_f$.
The following three  bounds when $R>r$,  proved in Proposition \ref{la.2},   play a  fundamental role in the asymptotic analysis throughout the paper:

(i) With probability approaching one, $\widehat\bS_f$ is invertible,  but its eigenvalues may decay quickly so that
\begin{equation}\label{e2.6}
\| \widehat\bS_f^{-1}\|=O_P(N).
\end{equation}
On the other hand,  $\bS_f$  is   degenerate  when $R>r$, whose rank equals $r$.  Also note that  we still have $\| \widehat\bS_f^{-1}\|=O_P(1)$ when $R=r$ holds.

(ii) Even if $R>r$,  $\|\bH'  \widehat\bS_f^{-1}\|$ is  much smaller:
$$
\|\bH'  \widehat\bS_f^{-1}\|=O_P\left(\sqrt{\frac{\max\{N,T\}}{T}}\right).
$$

(iii) When $R>r$, $\|\widehat\bS_f^{-1}-\bS_f^+\|\neq o_P(1)$ but  we have
$$
\|\bH'(\widehat\bS_f^{-1}-\bS_f^+)\bH\|=O_P\left(\frac{1}{T}+\frac{1}{N}\right).
$$


Therefore, $\widehat\bS_f$ is  invertible, and when weighted by the transformation matrix $\bH'$, its inverse is well behaved and fast converges to the generalized inverse of $\bS_f$, even though $\bS_f$ is singular when $R>r$.  It is  sufficient to  consider $\bH'\widehat\bS_f^{-1}$   in most factor-augmented inference problems, because in regression models $\widehat\bS_f^{-1}$  often appears in the projection matrix $\bP_{\widehat\bF}=\widehat\bF(\widehat\bF'\widehat\bF)^{-1}\widehat\bF'$ through $\bH'\widehat\bS_f^{-1}$ asymptotically, where  $\widehat\bF:=(\widehat\bff_1,\cdots,\widehat\bff_T)'$ and $\bF:=(\bff_1,\cdots,\bff_T)'$ denote the estimated and true factor matrices. 


\begin{remark}
In the CCE literature, (e.g., \cite{pesaran, chudik2011weak}), it has also been claimed that estimating the factors using   cross-sectional averages    does not require consistently estimating the number of factors. While
the claim is true, its  proof is not straightforward as $
\|\widehat\bS_f^{-1}-\bS_f^+\|\neq  O_P(1)\|\widehat\bS_f-\bS_f\|$ when $R>r.$  Also see \cite{karabiyik2017role,karabiyik2019cce} for more   discussions on the related issue. Our method  therefore also potentially  contributes to this literature  as an alternative rigorous approach.

\end{remark}

\subsection{Estimating the factor space}

Throughout the paper, the loading matrix $\bB$ can be either deterministic or random. When they are random, it is assumed that  it is independent of $\bu_t$, and all the expectations throughout the paper is taken   conditionally on $\bB$.

We make the following conditions.

\begin{ass}\label{ass2.2}
	(i) $\{(\bff_t, \bu_t): t\leq T\}$ is a stationary process, satisfying
	$\E(\bu_t|\bff_t)=0$.
	
	(ii) There are constants $c, C>0$, so that $\max_{i\leq N}\|\bb_i\|<C$,   and almost surely
	$$c<\lambda_{\min}(\frac{1}{T}\sum_{t=1}^T\bff_t\bff_t')\leq \lambda_{\max}(\frac{1}{T}\sum_{t=1}^T\bff_t\bff_t')<C.
	$$
\end{ass}
\begin{ass}[Weak  dependence]\label{ass2.3} There is a constant $C>0$, 
	\\
	(i) $  \max_{j,i\leq N}\frac{1}{NT}  \sum_{ q,v\leq N}\sum_{t,s\leq T}|\Cov(u_{it}u_{qt},u_{js}u_{vs} |\bF )|<C$  almost surely in $ \bF $,
  \\
	(ii) 
	 $ \frac{1}{T } \sum_{s=1}^T \sum_{t=1}^T\E\| \bff_{t}  \| \|\bff_{s} \|    \|\E(\bu_t \bu_s' |\bF)\|<C $ and $  \E \| \E(\bu_t\bu_t'|\bF)\|<C$.
\end{ass}

\begin{thm}\label{t2.1}
	Suppose  Assumptions \ref{ass2.1} - \ref{ass2.3}  hold.  Also $N\to\infty$ and $T$ is either finite or grows. Then for all bounded $R\geq r$,
\begin{eqnarray}
	\|\bP_{\widehat\bF}\bP_\bF -\bP_\bF\|& = &O_P\left(  \frac{1 }{\sqrt{N}} \nu_{\min}^{-1}(\bH)   \right),\label{eq3.3add}
	\\
	\|\bP_{\widehat\bF\bM} -\bP_\bF\| &=& O_P\left(  \frac{1 }{\sqrt{N}} \nu_{\min}^{-1}(\bH)   \right), \label{eq3.4addd}
	\end{eqnarray}
	where $\bM= ( \bH\bH')^{+}\bH$ is an $R\times r$ matrix \footnote{We show in the proof that $(\bM'\widehat\bF'\widehat\bF\bM)$ and $\widehat\bF'\widehat\bF$ are both invertible with probability approaching one. So $\bP_{\widehat\bF\bM}$ and $\bP_{\widehat\bF}$ are well defined asymptotically.}.
	
\end{thm}
	Equation (\ref{eq3.3add}) shows that  when $R\geq r$, the linear space spanned by $\widehat\bF$ asymptotically covers the linear space spanned by $\bF$.  To understand the intuition, note that  (\ref{eq3.3add}) implies $\bP_{\widehat\bF}\bP_\bF\bY \approx\bP_\bF\bY$ for an arbitrary random matrix $\bX$. Meanwhile, if we heuristically regard $\bP_{\bF}$ and $\bP_{\widehat\bF}$ as conditional (linear) expectations given $\bF$ and $\widehat\bF$, then approximately,
\begin{equation}\label{eq3.5add}
\mathbb E\left(\mathbb E(\bY|\bF)\bigg{|}\widehat\bF\right)\approx \mathbb E(\bY|\bF).
\end{equation}
Let $\text{span}(\bA)$     denote the linear space  spanned  by the columns of $\bA$.
The approximation  (\ref{eq3.5add}) is well known  to be the   ``tower property",  which heuristically means
$
\text{span}(\bF)\subseteq\text{span}(\widehat\bF)$.


	Equation (\ref{eq3.4addd})  shows that a particular subspace of $\text{span}(\widehat\bF)$ is consistent for $\text{span}(\bF)$. In the special case $R=r$, we have $\bP_{\widehat\bF\bM}=\bP_{\widehat\bF}$ since $\bM$ in  (\ref{eq3.4addd}) is invertible. It then reduces to the usual space consistency. Importantly, we  allow $T$ to be finite.

	To gain more insights of  these results, let us compare with the usual methods based on  estimating the number of factors, e.g., the eigenvalue-ratio method of \cite{AH}.  There are two key quantities in this comparison:  the strength of the spiked eigenvalues  of $\bS_x:=\frac{1}{T}\sum_{t=1}^T\bx_t\bx_t'$,   and the largest eigenvalue of  $\bS_u:=\frac{1}{T}\sum_{t=1}^T\bu_t\bu_t'$.

 We consider a setting where we can easily quantify the signal-noise ratio, as given in the following example.
 
\begin{exm}\label{ex2.1} This example presents a \textit{pervasive factor model} that satisfies  Assumption \ref{ass2.4}. 
 Suppose each individual loading satisfies $\bb_i=\nu_N\blambda_i$ for some sequence $ \nu_N\asymp N^{-(1-\alpha)/2}$  and  $\alpha\in(0,1]$, where $\{\blambda_i: i\leq N\}$ is a sequence of $r\times 1$ vectors  such that:

	(i)  $\frac{1}{N}\sum_{i=1}^N\blambda_i\blambda_i'\to \bC$ (or converges in probability if $\blambda_i$ is random) for   some positive definite matrix $\bC$;
	
	(ii)   $\nu_{\min}(\frac{1}{N} \bW'\bLambda)$
	is bounded away from zero, where $\bLambda=(\blambda_1,...,\blambda_N)'$.\\
Then Assumption \ref{ass2.4} holds for   $
	\nu_{\min}(\bH)\asymp \nu_N
	$
	and any $\alpha\in(0,1]$.
It is straightforward to verify that the $r$ th spiked eigenvalue satisfies: 
	$$
	\lambda_{r} \left( \bS_x\right)\asymp N^{\alpha},\quad \alpha\in (0,1].
	$$
	Theorem \ref{t2.1} then shows that $\|\bP_{\widehat\bF\bM} -\bP_\bF\|  =o_P(1)$  for any $\alpha>0$.

\end{exm}


	

	The key implication of Example \ref{ex2.1} is that the  
strength of the spiked eigenvalues   can grow at an arbitrarily slow polynomial rate in $N$, and $T$ is allowed to be finite. In   applications where $T\to \infty$ is required, the growth requirement of $T$ can be very mild. For instance, as we shall show in the high-dimensional factor-augmented regression (Section \ref{sec:4}), it is only required that $\log ^2N=o(T)$ if the number of ``important" control variables (corresponding to nonzero coefficients) is finite.  The relative flexibility on the growth of $T$ is achieved  thanks to the fact that the  diversified projection  does not  demand strong eigenvalues of the population covariance matrix.

	Now let us revisit the conditions required by the eigenvalue-ratio method  by \cite{AH}.  If $\bu_t$ is sub-Gaussian, under weak dependence conditions, 
	$$
	\lambda_{\max}(\bS_u)= O_P\left(\frac{\max\{T, N\}}{T}\right).
	$$
The selection consistency requires $\lambda_{r} \left( \bS_x\right)\gg\lambda_{\max}(\bS_u) $, which in this context,  becomes $T\gg N^{1-\alpha}$. In the case that the spiked eigenvalues are not so strong $(\alpha<0.5)$, 	 it requires a considerably longer time series to override the effect of the idiosyncratic  noise.

 \subsection{Summary of  advantages}

  Below we summarize key advantages of the use of diversified projection.

   \begin{enumerate}
 \item It is computationally and mathematically simple. 

 \item  When the true number of factors is  over estimated ($R\geq r$),    inferences about  transformation invariant parameters are still asymptotically valid.   
  This leads to  important  implications on   factor-augmented inferences and out-of sample forecasts.

 \item It admits an interesting special case, where $r=0$ and $R\geq 1$. That is,  $\bx_t$ is in fact weakly dependent, but we nevertheless estimate ``factors".  The resulting inference is still asymptotically valid in this case. We shall formally prove this in the high-dimensional factor-augmented inference in the next section. This shows that  extracting estimated factors  is   a robust inference procedure.

 \item As the diversified projections are applied cross-sectionally, they require very weak serial conditions. For instance,  the space spanned by the latent factors can be consistently estimated even if $T$ is finite. 
 It is also a good choice under weak signal-noise ratios where the consistent selection of the number of factors is hard to achieve.

 \item After applying the diversified projection to $\bx_t$ to reduce to a lower dimensional space, one can continue to employ the PCA  on $\widehat\bff_t$ to estimate the factor space and  the number of factors.  This becomes a low-dimensional PCA problem, and potentially much  easier than benchmark methods dealing with large dimensional  datasets.
  \end{enumerate}

  \section{Applications}\label{sec:3}


  \subsection{Forecasts using augmented factor regression}\label{sec:3.1}

 Consider forecasting time series using a large panel of  augmented factor regression:
\begin{eqnarray*}
y_{t+h} &=& \balpha'\bff_t+ \bbeta' \bg_t+\varepsilon_{t+h},\quad t=1,\cdots., T\cr
\bx_t&=& \bB\bff_t+\bu_t
\end{eqnarray*}
with  observed data $\{(y_t,\bx_t): t\leq T\}$.
 Here $h\geq 0$ is the lead time and $\bg_t$  is a vector of observed predictors including lagged outcome variables. The goal is  the mean forecast:
 $$
  y_{T+h|T}:=  \balpha'\bff_T+ \bbeta' \bg_T:=\bdelta'\bz_T,
 $$
 where $\bz_t=(\bff_t'\bH', \bg_t')'$ and $\bdelta'=(\balpha'\bH^+,\bbeta')$.  The prediction   also depends on unobservable factors $\bff_t$ whose information is contained in a high-dimensional panel of data.    This model has been studied extensively in the literature, see e.g., \cite{SW02, bai2006confidence, ludvigson2007empirical}, where $\bff_T$ is replaced by a consistent estimator. 
 Once   estimated factors $\widehat\bff_t$ is obtained, the forecast of $ y_{T+h|T}$ is straightforward:
 $$
 \widehat  y_{T+h|T} =\widehat\bdelta'\widehat\bz_T,\quad \widehat\bdelta=(\sum_{t=1}^{T-h}\widehat\bz_t\widehat\bz_t')^{-1}\sum_{t=1}^{T-h}\widehat\bz_ty_{t+h}
 $$
 where $\widehat\bz_t=(\widehat\bff_t', \bg_t')'$ denotes the estimated regressors. Note that $(\sum_{t=1}^{T-h}\widehat\bz_t\widehat\bz_t')^{-1}$ is well defined  even if $R>r$ with high probability. This   follows    from  the invertibility  of $\widehat\bF'\bM_{\bG}\widehat\bF $, a claim to be proved (the definition of $\bG$ is clear below, and the notation $\bM_{\bG}$ is defined in Introduction).

Our study is motivated by two important yet unsolved issues. First,  the study of prediction rates has been crucially relying on the assumption that the number of latent factors is correctly estimated. Secondly, the time series that are being studied are often relatively short,  to preserve the stationarity.   As we explained in Section 2, this leads to strong conditions on the strength of factors of using the PC estimator.

We show below that by allowing  $R>  r$, the diversified projection does not require a consistent estimator of the number of factors.  In addition to the assumptions in Section 2, we impose the following conditions on the forecast equation for $y_{t+h}$. Let $\bG$ be the   matrix of   $\{\bg_t: t\leq T-h\}$.   

 \begin{ass}\label{ass3.1}

 (i) $\{\varepsilon_t, \bff_t, \bg_t, \bu_t: t=1,\cdots, T+h\}$ is stationary with  $\E(\bu_t|\bff_t,\bg_t)=0$ and  $\E(\varepsilon_t|\bff_t, \bg_t, \bu_t,\bW)=0.$ \\
(ii) Weak dependence:  there is $C>0$,
$   \max_{s\leq T}  \sum_{t\leq T} |\E(\varepsilon_t\varepsilon_s |\bF,\bG,\bW)|< C$ almost surely.\\
(iii) Moment bounds:  there are  $c, C>0$,     $\lambda_{\min}(\frac{1}{T}\bF'\bM_\bG\bF) >c$, $\lambda_{\min}(\frac{1}{T}\bG'\bM_{ \bF\bH'}\bG) >c$, \\
and $c<\lambda_{\min}(\frac{1}{T}\bG'\bG)\leq \lambda_{\max}(\frac{1}{T}\bG'\bG)<C$.

 \end{ass}

 Our  theory   \textit{does not} follow from the standard theory of linear models of \cite{bai2006confidence}. A new  technical phenomenon  arises when $R>r$ due to the degeneracy of the gram matrices.
 Define $\widehat \bZ=(\widehat\bz_1',...,\widehat\bz_{T-h}')'$, $\bZ=(\bz_1',...,\bz'_{T-h})'
 $ and
 consider two gram matrices
 $$
  \widehat \bZ'  \widehat\bZ
= \begin{pmatrix}
   \widehat\bF'  \widehat\bF  &  \widehat\bF' \bG\\
\bG'\widehat \bF&\bG'\bG
\end{pmatrix}, \quad
 \bZ'\bZ
=  \begin{pmatrix}
\bH\bF'   \bF\bH'  &  \bH\bF' \bG\\
\bG'\bF\bH'&\bG'\bG
\end{pmatrix}.
$$
 The linear regression theory crucially depends on the inverse of $ \widehat \bZ'  \widehat\bZ$, whose population version
$\bZ'\bZ$, in this context, becomes degenerate when $R>r$.   The full rank matrix  $\frac{1}{T}\widehat\bF'\bM_{\bG}\widehat\bF$ converges to a degenerate matrix $\bH\frac{1}{T}\bF'\bM_{\bG}\bF\bH'$,  and therefore in general
   $$
\left\|  \Big (\frac{1}{T} \widehat \bZ'\widehat \bZ \Big) ^{-1}
-\Big (\frac{1}{T}  \bZ' \bZ \Big ) ^{+}
\right\| \neq o_P(1).
   $$
  We develop a new theory that takes advantage of   $  \bH$, which allows   to establish the three claims in Section \ref{poet}. They  imply that the convergence holds when weighted by $\widetilde\bH$:
  $$
\left\| \widetilde\bH' \left  ((\frac{1}{T} \widehat \bZ'\widehat \bZ) ^{-1}
-(\frac{1}{T}  \bZ' \bZ) ^{+} \right)\widetilde\bH
\right\|  =O_P\left(\frac{1}{{T}}+\frac{1}{{N}}\right), \text{ where}
\quad \widetilde\bH=\begin{pmatrix}
\bH&\\
& \bI
\end{pmatrix}.
   $$
The weighted convergence is sufficient to derive the prediction rate of  $\widehat  y_{T+h|T}$.


  \begin{thm}\label{t3.1}
Suppose Assumptions  \ref{ass2.1} - \ref{ass2.3},   \ref{ass3.1} hold. As $T, N\to\infty$, $h$ is bounded, and for all bounded $R\geq r$,
$$\widehat  y_{T+h|T}- y_{T+h|T}=O_P(\frac{1}{\sqrt{T}}+\frac{1}{\nu_{\min}\sqrt{N}}).$$

 \end{thm}



    \subsection{High-dimensional inference in factor augmented models}\label{sec:4}

  \subsubsection{Factor-augmented post-selection inference}
  Consider a high-dimensional regression model
  \begin{eqnarray}\label{eq4.1}
  y_{t} &=& \bbeta  \bg_t+ \bnu'\bx_t +\eta_t,\cr
  \bg_t&=& \btheta'\bx_t +\bvarepsilon_{g,t}
  \end{eqnarray}
  where  $\bg_t$ is a  treatment variable whose effect $\bbeta$ is the main interest. The model contains   high-dimensional control variables $\bx_t=(x_{1t},\cdots,x_{Nt})$ that determine both the outcome and treatment variables. Having  many control variables creates challenges for statistical inferences, as such,  we assume that  $(\bnu, \btheta)$ are  sparse vectors.    \cite{belloni2014inference}  proposed to make inference using \cite{robinson1988root}'s residual-regression, by first selecting among the high-dimensional controls in both the $y_t$ and $\bg_t$ equations.

  Often, the   control variables are strongly correlated due to the presence of  confounding factors
  \begin{equation}\label{eq4.2}
  \bx_t=\bB\bff_t+\bu_t.
  \end{equation}
This  invalidates  the conditions of using penalized regressions to directly select among $\bx_t$.  Instead, if we substitute (\ref{eq4.2}) to (\ref{eq4.1}), we reach  factor-augmented regression model:
  \begin{eqnarray}\label{eq4.3}
   y_{t} &=&\balpha_y'\bff_t+  \bgamma'\bu_t+   \bvarepsilon_{y,t},\cr
  \bg_t&=& \balpha_g'\bff_t +\btheta'\bu_t+\bvarepsilon_{g,t},\cr
  \bvarepsilon_{y,t}&=&\bbeta'  \bvarepsilon_{g,t}    +\eta_{t}
  \end{eqnarray}
  where  $\balpha_g'=\btheta'\bB$,  $\balpha_y'=\bbeta  \balpha_g'+\bnu'\bB$,  and $\bgamma'=\bbeta  \btheta'+ \bnu'$. The model contains  high-dimensional latent controls $\bu_t$.  Here $(\balpha_y, \balpha_g,\bbeta)$ are low -dimensional coefficient vectors    while $(\bgamma, \btheta)$ are high-dimensional sparse vectors. \cite{fan2020factor} and \cite{hansen2018fac} showed that the penalized regression can be successfully applied to (\ref{eq4.3}) to  select  components in $\bu_t$, which are    cross-sectionally weakly correlated. They   require strong conditions so that we can consistently estimate the number of factors $r=\dim(\bff_t)$ first.

  The main  result of this section is to  show that the  factor-augmented post-selection inference is valid for any $R\geq r$. 
Therefore,  we have  addressed   an important question in empirical applications, where the evidence of the number of factors is not so strong and one may  use a slightly larger number of ``working factors". 
The theoretical intuition, again, is that the model depends on $\bff_t$ only through   transformation invariant  terms,  so that $ 
\widehat\balpha_y'\widehat\bff_t=   \balpha_y'\bH^+\bH\bff_t+o_P(1)=
\balpha_y'\bff_t+o_P(1).$ 
  In addition, $\bu_t$ can also be well estimated with over-identified number of factors.

Importantly, we admit the special case $r=0$, and $R\geq 1$, leading to $\balpha_y$ and $\balpha_g$  both being zero in (\ref{eq4.3}). That is, there are no factors, so $\bx_t=\bu_t$ itself is cross-sectionally weakly dependent, but nevertheless we estimate $R\geq 1$ number of factors to run post-selection inference. This setting   is   empirically relevant as it allows  to avoid pre-testing the presence of common factors for inference.
The simulation in Section \ref{sec:sim} shows that with $R\geq r$,  this procedure  works well even if $r=0$; but when   $r$ $(r \geq 1)$ factors are present, directly selecting   $\bx_t$  leads to severely biased estimations. Therefore as a practical guidance, we recommend that one  should always run factor-augmented post-selection inference, with  $R\geq 1$, to guard against confounding factors  among the control variables.

Below we first present the factor-augmented algorithm as in  \cite{hansen2018fac} for estimating (\ref{eq4.1}). 
For notational simplicity, we focus on the univariate case $\dim(\bbeta)=1$.

\begin{algo}\label{ag4.1}
Estimate $\bbeta$ as follows.
\begin{description}
\item[Step 1] Fix the working number of factors $R$. Estimate $\{(\bff_t, \bu_t): t\leq T\}$ as in Section 2.

\item[Step 2] (1) Estimate coefficients:
$
\widehat\balpha_y=(\sum_{t=1}^T\widehat\bff_t\widehat\bff_t')^{-1}\sum_{t=1}^T\widehat\bff_ty_t,$ and $\widehat\balpha_g=(\sum_{t=1}^T\widehat\bff_t\widehat\bff_t')^{-1}\sum_{t=1}^T\widehat\bff_t\bg_t.
$
(2) Run penalized regression:
\begin{eqnarray*}
 \widetilde\bgamma&=&\arg\min_{\bgamma} \frac{1}{T}\sum_{t=1}^T(y_t-\widehat\balpha_y'\widehat\bff_t-\bgamma'\widehat\bu_t)^2+ P_{\tau}(\bgamma),\cr
  \widetilde\btheta&=&\arg\min_{\btheta} \frac{1}{T}\sum_{t=1}^T(\bg_t-\widehat\balpha_g'\widehat\bff_t-\btheta'\widehat\bu_t)^2+ P_{\tau}(\btheta).
\end{eqnarray*}
(3) Run post-selection refitting: let $\widehat J=\{j\leq p: \widetilde\gamma_j\neq 0\}\cup \{j\leq p: \widetilde\theta_j\neq 0\} $.
\begin{eqnarray*}
\widehat\bgamma&=&
\arg\min_{\bgamma} \frac{1}{T}\sum_{t=1}^T(y_t-\widehat\balpha_y'\widehat\bff_t-\bgamma'\widehat\bu_t)^2, \quad \text{ such that } \widehat\gamma_j=0\text{ if } j\notin \widehat J.\cr
\widehat\btheta&=&\arg\min_{\btheta} \frac{1}{T}\sum_{t=1}^T(\bg_t-\widehat\balpha_g'\widehat\bff_t-\btheta'\widehat\bu_t)^2, \quad \text{ such that } \widehat\theta_j=0\text{ if } j\notin \widehat J.
\end{eqnarray*}

\item[Step 3] Estimate residuals:
$\widehat \bvarepsilon_{y,t}=y_{t} -( \widehat\balpha_y' \widehat\bff_t+   \widehat\bgamma' \widehat\bu_t),$ and $ \widehat \bvarepsilon_{g,t}=\bg_{t} -( \widehat\balpha_g' \widehat\bff_t+   \widehat\btheta' \widehat\bu_t).$

\item[Step 4] Estimate  $\bbeta$ by residual-regression:
$$
\widehat\bbeta=(\sum_{t=1}^T\widehat\bvarepsilon_{g,t} ^2)^{-1}\sum_{t=1}^T\widehat\bvarepsilon_{g,t} \widehat\bvarepsilon_{y,t}.
$$
\end{description}
\end{algo}
One can also simplify step 2   following \cite{fan2020factor}:  finding $(\widehat\balpha_y, \widehat \bgamma)$ by minimizing
$ \frac{1}{T}\sum_{t=1}^T(y_t - \balpha_y'\widehat\bff_t-\bgamma'\widehat\bu_t)^2+ P_{\tau}(\bgamma)$ and defining $(\widehat\balpha_g, \widehat \btheta)$ similarly.

  Note that $\bgamma:\to P_\tau(\bgamma)$ is a sparse-induced penalty function with a tuning parameter $\tau$. In the main theorem below, we prove for the lasso $P_\tau(\bgamma)=\tau\|\bgamma\|_1$, where $\|\bgamma\|_1=\sum_{j=1}^N|\gamma_j|$.  As in \cite{Bickeletal}, we set $$
  \tau=C\sqrt{\frac{\sigma^2\log N}{T}}
  $$ for some constant $C>4,$
  where $\sigma^2=\var(\bvarepsilon_{y,t})$ for estimating $\bgamma$, and $\sigma^2=\var(\bvarepsilon_{g,t})$ for estimating $\btheta$. Refer to \cite{belloni2014inference} for feasible tunings so that  $\sigma^2$ is estimated iteratively. 

    \subsubsection{The main result}

    We impose the following assumptions.

      \begin{ass}\label{a4.1}
    (i)    $\E(\bvarepsilon_{g,t}|\bu_t,\bff_t , \bW)=0 $ and $\E(\bvarepsilon_{y,t}|\bu_t,\bff_t, \bW)=0 $,\\
    (ii)  Coefficients: there is $C>0$, so that $\|\balpha_y\|$, $\|\balpha_g\|$, $\|\bbeta\|$ are all bounded by $C$. \\
(iii)    Weak dependence:  There is $C>0$, almost surely,\\
$   \max_{s\leq T}  \sum_{t\leq T} |\E(\bvarepsilon_{y,t}\bvarepsilon_{y,s} |\bF,\bU, \bW)|+ \max_{s\leq T}  \sum_{t\leq T} |\E(\bvarepsilon_{g,t}\bvarepsilon_{g,s} |\bF,\bU, \bW)|< C$.\\
 (iv) Uniform bounds: \\
$
 \max_{i\leq N} \|\frac{1}{T}\sum_{t=1}^Tu_{it}\bv_{ t}\|= O_P(\sqrt{\frac{\log N}{T}}) $
for all $\bv_t\in\{\bvarepsilon_{g,t}, \bvarepsilon_{y,t},\bff_t\}$.
  In addition,
 \\
$\max_{i\leq N}|\frac{1}{T}\sum_{t=1}^T(u_{it}u_{jt}-\E u_{it}u_{jt})|=O_P(\sqrt{\frac{\log N}{T}}),
$
 and \\
$\max_{i\leq N}|\frac{1}{TN}\sum_{t=1}^T\sum_{j=1}^N(u_{it}u_{jt}-\E u_{it}u_{jt})w_{k,j}|=O_P(\sqrt{\frac{\log N}{TN}})
$   for  all $k\leq R$.

\end{ass}

Assumption \ref{a4.1} (iv) holds generally  under weak time-series dependent conditions  for $\{(\bv_t, \bu_t): t\leq N\}$ with sub-Gaussian tails.

Suppose the high-dimensional coefficients $\btheta$ and $\bgamma$ are strictly sparse.   Let $J$ denote the nonzero index set:
    $$
    J=\{j\leq N: \theta_j\neq 0\}\cup \{j\leq N: \gamma_j\neq 0\}.
    $$

    The following \textit{sparse eigenvalue condition}   is standard for the post-selection inference. Note that it is imposed on the covariance of $\bu_t$ rather than $\bx_t$, because $\bu_t$ is weakly dependent.
     \begin{ass}[Sparse eigenvalue condition]
 For any $\bv\in\mathbb R^N\backslash\{0\}$,
define:
\begin{eqnarray*}
\phi_{\min}(m)=\inf_{\bv\in\mathbb R^N: 1\leq \|\bv\|_0\leq m}\mathcal R(\bv),\quad \text{ and }\phi_{\max}(m)=\sup_{\bv\in\mathbb R^N: 1\leq \|\bv\|_0\leq m}\mathcal R(\bv),
\end{eqnarray*}
where
$
\mathcal R(\bv):=\|\bv\|^{-2}\bv'\frac{1}{T}\sum_{t=1}^T\bu_t\bu_t'\bv.
$ Then there is a sequence $l_T\to\infty$ and $c_1, c_2>0$ so that with probability approaching one,
$$
c_1<\phi_{\min}(l_T|J|_0)\leq \phi_{\max}(l_T|J|_0)<c_2.
$$
    \end{ass}

     \begin{ass}\label{ass3.5}
 (i) $\frac{1}{T} \sum_{t=1}^T\bvarepsilon_{g,t}^2  \overset{P}{\longrightarrow} \sigma_g^2  $ for some $\sigma_g^2>0$.\\
 (ii) $\frac{1}{\sqrt{T}}\sum_{t=1}^T\eta_t\bvarepsilon_{g,t}\overset{d}{\longrightarrow}\mathcal N(0,\sigma_{\eta g}^2)$  for some $\sigma_{\eta g}^2>0$. In addition, there is a consistent variance estimator $\widehat\sigma_{\eta g}^2\overset{P}{\longrightarrow}\sigma_{\eta g}^2.$
\\
(iii)  The rates $(N, T, |J|_0)$ satisfy:
$$|J|_0^4 \log^2N=o(T),\quad \text{ and }  T|J|_0^4=o(N^2\min\{1, |J|_0^4  \nu^{4}_{\min}(\bH)\} ).$$

    \end{ass}

Condition~\ref{ass3.5} (iii) requires the ``effective dimension" $N\nu^{2}_{\min}(\bH)$ be relatively large in order to accurately estimate the latent factors.


     \begin{thm}\label{t3.2}
  Suppose  $\widehat\bff_t $  contains    $R\geq r\geq 0$ number of diversified weighted averages of $\bx_t$.   If $r\geq 1$   (there are factors in $\bx_t$), Assumptions \ref{ass2.1} - \ref{ass2.3}, \ref{a4.1}-\ref{ass3.5} hold. If $r=0$ (there are no factors in $\bx_t$), Assumption \ref{ass2.4} is relaxed, and all $\bff_t$ involved  in the above assumptions  can be removed.
 Then as $T, N\to\infty$,   for all bounded $R\geq r\geq 0$,
  $$
 \sigma_{\eta, g}^{-1} \sigma_g^{2}  {\sqrt{T}(\widehat\bbeta-\bbeta)}\overset{d}{\longrightarrow}\mathcal N(0,1).
$$

 \end{thm}

 Fix a significant level $\tau$,  let $\zeta_\tau$ be the $(1-\tau/2)$ quantile of standard normal distribution.  In addition,    let $\widehat\sigma_g^{2}=\frac{1}{T}\sum_{t=1}^T\widehat\bvarepsilon_{g,t}^2$.
 Immediately, we have the following uniform coverage.
 \begin{cor}\label{cor3.1}
Suppose the assumptions of Theorem \ref{t3.2} hold.  Let $\bar R>0$ be a fixed upper bound for $R$. Then uniformly for all $0\leq r\leq R\leq \bar R$,
$$
\mathbb P\left( \bbeta\in [\widehat\bbeta\pm \frac{1}{\sqrt{T}}\widehat\sigma_{\eta, g}\widehat\sigma_g^{-2} \zeta_\tau  ]\right)\to 1-\tau.
$$
 \end{cor}
The novelty of the above uniformity    is that  the coverage is valid uniformly   for all bounded $r$ as the true number of factors, and all   over-estimated $ R$  as the working number of factors. In particular, it also admits the weak-dependence     $r=0$ while  $R\geq 1$  as a special case.

 \begin{remark}[Case $r=0, R\geq 1$]
  We now explain the intuition of the case  $\bx_t= \bu_t$ (no presence of confounding factors), but we nevertheless extract $R\geq 1$ ``factors".  In this case $\balpha_y=\balpha_g=0$ in the system (\ref{eq4.3}). Then $\widehat\bff_t= \frac{1}{N}\sum_{i=1}^N\bw_iu_{it}:=\be_t$ degenerates to zero. Both $  \bu_t$ and  $ \balpha_y' \bff_t$ (which is zero)  are  still estimated well  in the following sense: 
  \begin{eqnarray*}
 \max_{i\leq N}\frac{1}{T} \sum_{t=1}^T(\widehat u_{it}-u_{it})^2&=& O_P\left(\frac{1}{N}+\frac{\log N}{T}\right)
\cr
    \frac{1}{T}\sum_{t=1}^T(\widehat\balpha_y'\widehat\bff_t)^2&=&O_P\left(\frac{ |J|_0^2}{N}+\frac{ |J|_0^2}{ { T}}\right).
 \end{eqnarray*}
  \end{remark}

 \begin{remark}[Case $R=0$]
  For completeness of the theorem, we  define the estimator for the case $R=0$. In this case we do not extract any factor estimators, and simply set $\widehat\balpha_y=\widehat\balpha_g=0$, and $\widehat\bu_t=\bx_t$ in Algorithm \ref{ag4.1}. This  is then the same setting as in \cite{belloni2014inference}.
  \end{remark}

\subsection{Estimating the idiosyncratic covariance}
\label{estimSigu}
The estimation of the $N\times N$ idiosyncratic covariance matrix $\bSigma_u:=\E\bu_t\bu_t'$  is of general interest in many   applications. Examples include the efficient estimation    of factor models \citep{bai2012statistical},   high-dimensional testing \citep{powerenhancement}, and bootstrapping latent factors  \citep{gonccalves2018bootstrapping}, among many others.
 While this problem has been studied by  \cite{POET},  they  require that  the true number of factors $r$ has to be either known or consistently estimated, and the factors are estimated through PCA.   Here we show that using the diversified factors, their conclusion holds for all fixed $R\geq r$.

A key assumption
is that $\bSigma_u = (\sigma_{u, ij})$ is sparse: As in \cite{Bickel08a} the sparsity of $\bSigma_u$ is measured by the following quantity:
$$
m_N=\max_{i\leq N}\sum_{j\leq N} |\sigma_{u, ij}|^q,\quad \text{ for some } q\in[0,1].
$$
Given the estimated residual  $\widehat u_{it}$ that is obtained using a working number of factors $R$, we estimate  $\E u_{it}u_{jt}$ by applying a generalized thresholding: define  $s_{u,ij}:=\frac{1}{T}\sum_{t=1}^T\widehat u_{it}\widehat u_{jt}$,
$$
\widehat\sigma_{u,ij}=\begin{cases}
s_{u,ij},& \text{ if } i=j\\
h(s_{u,ij},\tau_{ij}),   & \text{ if } i\neq j
\end{cases}
$$
where $h(s,\tau)$ is a  thresholding function with threshold value $\tau$.  Then the sparse idiosyncratic covariance estimator is defined as
$\widehat\bSigma_u=(\widehat\sigma_{u,ij})_{N\times N}.$
The threshold value $\tau_{ij}$ is  chosen as
$$
\tau_{ij}= C\sqrt{s_{u,ii}s_{u,jj}}\omega_{NT},\quad
\omega_{NT}:=\sqrt{\frac{\log N}{T}}+\frac{1}{\sqrt{N}}
$$
for some large constant $C>0$, which applies a constant thresholding to correlations.

In general, the thresholding function should satisfy:
\\
(i) $h(s,\tau)=0$ if $|s|<\tau$,\\
(ii) $|h(s,\tau)-s|\leq \tau$.\\
(iii) there are constants $a>0$ and $b>1$ such that $|h(s,\tau)-s|\leq a\tau^2$ if $|s|>b\tau$.

Note that condition (iii) requires that the thresholding bias should be of higher order.  It  is not necessary for consistent estimations, but we recommend using   nearly unbiased thresholding  \citep{AF} for  inference applications. One such example is known as SCAD.
As noted  in  \cite{powerenhancement},   the unbiased thresholding is required  to avoid   size distortions in a large class of high-dimensional testing problems involving a ``plug-in" estimator of $\bSigma_u$.
In particular, this rules out the popular \textit{soft-thresholding} function,  which does not satisfy (iii) due to its first-order shrinkage bias.

  \begin{thm}\label{t2.2}
  Let    $\widehat\bu_t$ be constructed using   $R\geq r$ number of diversified weighted averages of $\bx_t$.  Suppose  that Assumptions \ref{ass2.1} - \ref{ass2.3} hold and that $\log N=o(T)$. In addition,
  either $\nu_{\min}^2(\bH)\gg \frac{1}{\sqrt{N}}$ or $\nu_{\min}^2(\bH)\gg\frac{1}{N}\sqrt{\frac{T}{\log N}}.$
 Then as $N, T\to\infty$,  for any $R\geq r\geq 0$,

  (i)
  $$
\max_{i\leq N}\frac{1}{T}\sum_{t=1}^T(\widehat \bb_{i}'\widehat\bff_t-\bb_i'\bff_t)^2 =O_P(\omega_{NT}).
  $$

  (ii) For a sufficiently large constant $C>0$ in the threshold $\tau_{ij}$,
  $$
  \|\widehat\bSigma_u-\bSigma_u\|=O_P(\omega_{NT}^{1-q} m_N).
  $$

 (iii) If in addition, $\lambda_{\min}(\bSigma_u)>c_0$ for some $c_0>0$ and $\omega_{NT}^{1-q} m_N=o(1)$, then
   $$
  \|\widehat\bSigma_u^{-1}-\bSigma_u^{-1}\|=O_P(\omega_{NT}^{1-q} m_N).
  $$
  \end{thm}


      \subsection{Testing Specification  of  Factors }\label{sec:5}
     In practical applications, many ``observed factors" $\bg_t$ have been proposed to approximate the true latent factors. For example, in asset pricing, popular choices of $\bg_t$ are proposed and discussed in  seminal works by \cite{FF,  carhart1997persistence}, which are known as   the  Fama-French factors and Carhart four factor models.

     We test the (linear) specification of a given set of empirical factors $\bg_t$. That is,  we test:      $$
      H_0: \text{there is a $r\times r$ invertible matrix $\btheta$ so that } \bg_t=\btheta\bff_t,\quad \forall t\leq T.
      $$
      Under the null hypothesis, $\bg_t$ and $\bff_t$ are linear functions of each other.  We propose a simple statistic:
    $$
    \|\bP_{\bG}-\bP_{\widehat\bF}\|_F^2
    $$
    where $\bG=(\bg_1,\cdots,\bg_T)'$ and recall that  $\bP_{(\cdot)}$ denotes the projection matrix.  Here we still use the diversified factor estimator $\widehat\bF$. The test statistic measures the distance between (linear) spaces respectively spanned by $\bg_t$ and $\widehat\bff_t$. To derive the asymptotic null distribution, we naturally set   the working number of factors $R=\dim(\bg_t)$, which is known and equals $\dim(\bff_t)=r$  under the null. Then $\|\bP_{\widehat\bF}-\bP_\bF\|_F=o_P(1)$, followed from Theorem \ref{t2.1}.

      \subsubsection{Asymptotic null distribution}
      With the diversified factor estimators, the null distribution of the statistic is very easy to derive, and satisfies:
      $$
      \frac{N\sqrt{T}( \|\bP_{\bG}-\bP_{\widehat\bF}\|_F^2-\MEAN)}{\sigma}\overset{d}{\longrightarrow}\mathcal N(0,1),
      $$
where  for $\bA =2\bH^{'-1}(\frac{1}{T} \bF' \bF)^{-1}\bH^{-1}$,
\begin{eqnarray*}
 \MEAN&=&  \frac{1}{N^2} \tr\bA \bW' \E(\bu_t\bu_t'|\bF)\bW , \qquad
  \sigma^2 = \Var(\frac{1}{N}\tr\bA\bW'\bu_t\bu_t'\bW|\bF,\bW)>0.
 \end{eqnarray*}
Here we assume $\sigma^2>0$ to be bounded away from zero.  To avoid nonparametrically estimating high-dimensional  covariances,  we shall assume the conditional covariances in both bias and variance are independent of $\bF$ almost surely.
Nevertheless,  the bias   depends on  a high-dimensional  matrix   $ \bSigma_u= \E(\bu_t\bu_t').
 $
We employ the sparse covariance $\widehat\bSigma_u$ as defined in Section \ref{estimSigu} and replace the bias  by
$$  \widehat{\MEAN}:= \frac{1}{N^2} \tr\widehat\bA\bW'\widehat \bSigma_u\bW 
\quad \mbox{with} \quad \widehat \bA:=2(\frac{1}{T}\widehat\bF'\widehat\bF)^{-1}.
$$
Further suppose $\sigma$ can be consistently estimated by some $\widehat\sigma$,  then together, we have the feasible standardized statistic:
\begin{equation}\label{eq3.4}
\frac{N\sqrt{T}(\|\bP_{\widehat\bF}-\bP_\bG\|_F^2-\widehat\MEAN)}{ \widehat\sigma}.
\end{equation}

The problem, however, is not as straightforward as it looks by far. The use of $\widehat{\MEAN}$ and $\widehat\sigma$ both come with   issues, as we now explain.

 \textbf{The issue of $ \widehat{\MEAN}$. }

 When deriving the asymptotic null distribution,  we need to address the effect of $\widehat\bSigma_u-\bSigma_u$, which is   to show
\begin{equation}\label{eq3.5}
\frac{N\sqrt{T}(\widehat\MEAN-\MEAN)}{\sigma}\approx\frac{N\sqrt{T}}{\sigma} \frac{1}{N^2} \tr\bA \bW' (\widehat\bSigma_u-\bSigma_u)\bW\overset{P}{\longrightarrow} 0.
\end{equation}
But   simply applying the  rate of convergence
of $\|\widehat\bSigma_u-\bSigma_u\|$ in Theorem \ref{t2.2}  fails to show the above convergence, even though the rate is minimax optimal \footnote{A simple calculation would only yield $\frac{N\sqrt{T}}{\sigma} \frac{1}{N^2} \|\bA \bW' \|\| \widehat\bSigma_u-\bSigma_u\|\|\bW\|\leq O_P(1)$ but not necessarily $o_P(1)$.}.
   Similar phenomena also arise in  \cite{powerenhancement,bai2017inferences}, where a plug-in estimator for $\bSigma_u$ is used for inferences. Proving (\ref{eq3.5})  requires a new technical argument to address the accumulation of high-dimensional estimation errors. It requires a strengthened condition on the weak cross-sectional dependence, in Assumption  \ref{as3.8} below.

 \textbf{The issue of $ \widehat\sigma$. }

It is difficult to estimate $\sigma$ through residuals $\widehat\bu_t$  since   $\bW'\widehat\bu_t=0$ almost surely. In fact, estimated $\bu_t$   constructed based on any  factor estimator would lead to \textit{inconsistent} estimator for $\sigma^2$.  Therefore, we propose to estimate $\sigma^2$ by  parametric bootstrap.  Observe that $\frac{1}{\sqrt{N}}\bW'\bu_t$ is asymptotically normal, whose variance is given by
$
\bV=\frac{1}{N}\bW'\bSigma_u\bW.
$
Hence $\sigma^2$ should be approximately equal to
\begin{equation}\label{eq3.4add}
f(\bA, \bV):=\Var(\frac{1}{N}\tr\bA\bW'\bZ_t\bZ_t'\bW),
\end{equation}
where $\bZ_t$ is distributed as $\mathcal N(0, \bV). $ Therefore we estimate $\sigma^2$ by
$$
\widehat\sigma^2=f(\widehat\bA, \widehat\bV),\quad \text{with }\widehat\bV=\frac{1}{N}\bW'\widehat\bSigma_u\bW,
$$
which can be calculated  by simulating from $\mathcal N(0, \widehat\bV). $

Above all,  despite of the simple construction of $\widehat\bF$, the technical problem is still challenging. Therefore, this subsection calls for relatively stronger conditions, as we now impose.

 \begin{ass}
 \label{as3.6} (i) $\{\bu_t: t\leq T\}$ are stationary and conditionally serially independent, given $\bF $ and $\bG$. \\
(ii)         There is $C>0$,
        $\E[\|\frac{1}{\sqrt{N}}\bW'\bu_t\|^4|\bW]<C$.\\
        (iii)   $\nu_{\min}(\bH)>c$ for some $c>0.$
         \end{ass}

     The next assumption ensures that $\sigma^2$ can be estimated by simulating from the Gaussian distribution.

    \begin{ass}
     (i) There is $c>0$ so that $\sigma^2>c$.\\
   	(ii)  As $N\to\infty$,     $|\sigma^2-f(\bA,\bV)|\to 0$ almost surely in $\bF$, where $f(\bA,\bV)$ is given in (\ref{eq3.4add}).
   \end{ass}

 Next, we shall require $\bSigma_u$ be strictly sparse, in the sense that the ``small" off-diagonal entries are  exactly zero. In this case, we use the following measurement  for the total sparsity:
        $$
  D_N:=\sum_{i,j\leq N} 1\{\E u_{it}u_{jt}\neq 0 \}.
$$
   Recall that
$        \omega_{NT}:=\sqrt{\frac{\log N}{T}}+\frac{1}{\sqrt{N}}$. We assume:
\begin{ass}[Strict sparsity]
               (i)  $ (\frac{\omega_{NT}^2\sqrt{T}}{N}) D_N\to0$.\\
               (ii)  $\min\{|\E u_{it}u_{jt}|: \E u_{it}u_{jt}\neq 0 \}\gg \omega_{NT}$.
               \end{ass}   For block-diagonal matrices with finite  block sizes,  $D_N=O(N)$; for banded matrices with band size $l_N$, $D_N=O(l_NN)$.  In general,  suppose $D_N=l_NN$ with some slowly growing $l_N\to\infty$. Then condition (i) reduces to requiring
               $
               l_N^2\log N\ll l_N\sqrt{T}\ll N.
               $
               This requires an upper bound for $l_N$; in addition, the lower bound for $N$ arises from the requirement of estimating factors.  Condition (ii) requires that the nonzero entries are well-separated from the statistical errors.

      \begin{ass}\label{as3.8}
           Write $\sigma_{u,ij}:=\E u_{it}u_{jt}$. There is $C>0$ so that
           $$
           \frac{1}{N}\sum_{(m,n):\sigma_{u,mn}\neq 0,} \sum_{(i,j):\sigma_{u,ij}\neq 0}|\Cov( u_{it}u_{jt} ,    u_{mt}u_{nt} )|<C.
           $$
               \end{ass} The above  assumption   is the key condition to argue for (\ref{eq3.5}). It requires further conditions on the weak cross-sectional dependence, in addition to the sparsity.
 \cite{powerenhancement} proved that if $u_{it}$ is Gaussian, then a sufficient condition for Assumption \ref{as3.8} is as follows:
 $$
D_N=O(N),\text{ and }  \max_{i\leq N}\sum_{j\leq N}  1\{\E u_{it}u_{jt}\neq 0 \}=O(1),
 $$
which is the case for block diagonal matrices with finite members in each block and banded matrices with $l_N=O(1)$. 

       \begin{thm}\label{t3.3} Suppose $R=\dim(\bg_t)$, and   Assumptions \ref{ass2.1} - \ref{ass2.3}, \ref{as3.6}- \ref{as3.8} hold. As $N, T\to\infty$, under $H_0$,
       $$
  \frac{N\sqrt{T}(\|\bP_{\widehat\bF}-\bP_\bG\|_F^2-\widehat\MEAN)}{ \widehat\sigma}\overset{d}{\longrightarrow}\mathcal N(0,1).
       $$

       \end{thm}

\subsection{Factor-adjusted  false discovery control for multiple testing.}

Controlling the false discovery rate  (FDR) in large-scale hypothesis testing based on strongly correlated  testing series   has been an important   problem.  Suppose the data are generated from:
$$
\bx_{t}= \balpha + \bB\bff_t+ \bu_{t},
$$
where $\balpha=(\alpha_1,...,\alpha_N)'$ is the mean vector. This model allows strong cross-sectional dependences among $\bx_{t}$.
We are interested in testing $N$ number of hypotheses:
$$
H_{0}^i: \alpha_i=0,\quad i=1,..., N.
$$
  The FDR control  aims to develop test statistics $Z_i$ and threshold values   so that the overall false discovery rate is controlled at certain value. A crucial requirement is that   these test statistics should be   weakly dependent.
However, for $\bar\bff=\frac{1}{T}\sum_t\bar \bff_t$ and $\bar\bu=\frac{1}{T}\sum_t\bar \bu_t$,  we have
$
 \bar\bx=\frac{1}{T}\sum_t\bar \bx_t=\balpha +\bB\bar\bff+\bar\bu,
$
so  the presence of $\bB\bff_t$ makes the   mean  vector
    be cross-sectionally strongly dependent, failing  usual FDR procedures based on the simple sample average. This is the well known   confounding factor problem. While several methods have been proposed to remove the effect of confounding factors \citep{wang2017confounder, fan2019farmtest}, again, it has been assumed that the number of factors should be consistently estimable.

The diversified projection can be applied directly as a simple  implementation for the FDR control, valid for all $R\geq r$. Let  the diversified projection be
$\widehat\bff_t = \frac{1}{N}\bW'\bx_t $, and let $\widehat\bb_i$
be the OLS estimator for the slope vector by regressing $ x_{it}$ on $\widehat\bff_t$ with intercept.
 Then we can define the factor-adjusted regularized multiple test  \citep{fan2019farmtest} statistics
$
Z_i=   \widehat\alpha_i /se(\widehat\alpha_i)
$
where
$$
\widehat\alpha_i=  \bar x_i- \widehat\bb_i'\widehat\bff,\quad \widehat\bff=\frac{1}{T}\sum_{t=1}^T\widehat\bff_t,
$$
and $se(\widehat\alpha_i)$ is the associated standard error.
Our theories imply  the following expansion, uniformly for $i=1,..., N$ and all $R\geq r$,
\begin{eqnarray*}
\widehat\alpha_i-\alpha_i&=&\frac{1}{T}\sum_{t=1}^T\bg_tu_{it}
+o_P(T^{-1/2}),  \end{eqnarray*}
where $\bg_t=  1- \bar \bff' \bS_f^{-1} (\bff_t-\bar \bff), $ and $ \bS_f=\frac{1}{T}\sum_t(\bff_t-\bar \bff)(\bff_t-\bar \bff)'. $ This gives rise to the desired expansion
 so that $Z_i$ are weakly dependent. Therefore, we can apply standard procedures to $Z_i$ for the false discovery control.

       \section{Choices of Diversified Weights}\label{choice}

We discuss  some specific examples to choose  $\bW=(\bw_1,\cdots,\bw_R)=(w_{k,i}: k\leq R, i\leq N)$, the  weight matrix.
      \subsection{Loading characteristics}

      Factor loadings are often driven by observed characteristics. For example, in genetic studies,  single-nucleotide polymorphism (SNP) data  are often collected with the gene expression data  on the same group of subjects.  The SNPs
      drive underlying structure in the gene expressions, clinical and demographics data, through affecting their   loadings on the biological factors.
In asset pricing studies, it has been well documented that factor loadings are driven by firm specific characteristics, which are  independent of the model noise, but have strong explanatory powers on the loadings.

Motivated by the presence of characteristics, ``characteristic based" factor models have been extensively studied in the literature, e.g.,  \cite{gagliardini2016time, li2016supervised,CMO}. The general form of this model assumes the loadings  have the following decomposition \citep{fan2016projected}:
      $$
      \bb_i= \bg(\bz_i) + \bgamma_i,\quad \mathbb E(\bgamma_i|\bz_i)=0,\quad i\leq N,
      $$
      where $\bz_i$ is a vector of   characteristics that are observed on each subject and $\bg(\cdot)$ is a  nonparametric mean function. It is assumed that $\{\bz_i:i\leq N\}$ is independent of $\bu_t$ and that $\bg(\bz_i)$ is not degenerate so that $\bz_i$ has explanatory power. In addition,
   $\bgamma_i$ is the remaining loading components, after conditioning on $\bz_i.$
      When $\bz_i$ is available,   we can employ them as a  natural choice of the weights for the diversified factors.
      Fix an $R$-component of sieve basis functions: $(\phi_1(\cdot),...,\phi_R(\cdot))$ such as the Fourier basis or B splines. Then define
      $$\bW:=  (w_{i,k})_{N\times R},\quad \text{ where } w_{i,k}=\phi_k(\bz_i).$$

 The diversified projection using the so-constructed $\bW$ is related to the ``projected PCA" of \cite{fan2016projected}, but the latter is more complicated and requires stronger conditions than the diversified projection, because it is still PCA based. 

    \subsection{Moving window estimations}\label{trimpca}

This method is useful  when $\bu_t$ is serially independent, and related ideas have    been used recently  by \cite{barigozzi2018consistent}. Consider    out-of-sample forecasts  using  moving windows.
   Suppose   $\bx_t$ is observed for $T+T_0$ periods in total, but to pertain the  stationarity assumption, we only use the most recent $T$ observations to learn the latent factors, where $T$ may be potentially small.  Divide the sample into two periods:
   \begin{eqnarray*}
&&\text{periods (I) of learning weights:} \quad \bx_t=\bB_1\bff_t+\bu_t,\quad t=1,...,T_0\cr
&&\text{periods (II) of interest: } \quad\bx_t=\bB\bff_t+\bu_t,\quad t=T_0+1,...,T_0+T.
   \end{eqnarray*}
  While $\bB_1$ and $\bB$ can be different (e.g., presence of structural  breaks), they are assumed to be closely related between two sampling periods. As such, we can learn about the diversified weights from periods (I) to estimate the latent factors for the periods of estimation interests (II). Specifically, apply PCA on periods (I) to extract $R$ number of factor loadings:
  $
  \widehat \bB_1=(\widehat b_{i,k})_{N\times R}.
  $
  Now for a pre-determined  constant $\epsilon>0$, define $\bW=(w_{i,k})_{N\times R}$  where
  $$
  w_{i,k}=\frac{\widehat b_{i,k}}{\max\{1, \epsilon\max_{i\leq N}|\widehat b_{i,k}|\}},\quad k\leq R,\quad i\leq N.
  $$
 As discussed by \cite{barigozzi2018consistent},  the trimming constant $\epsilon$  ensures that  these weights are well diversified and correlated with the loadings. In addition, if $\bu_t$ is serially independent, then $\bW$ is also  independent  of $\bu_t$ for $t=m+1,...,m+T. $ As such, the conditions on the diversified weights are satisfied. It is straightforward to extending this idea to multi-periods rolling window forecasts, where weights are sequentially updated for rolling windows.
  
The aforementioned method   uses the idea that sample splitting creates serial independences. 
In the presence of mixing-type serial dependences, \cite{barigozzi2018consistent}
proposed to split   the data into blocks     and estimate factor loadings using subsamples  
omitting the current block as well as its immediate neighbors. Their method can be also applied in the current context to create the weighting matrix.

      \subsection{Initial Transformation}

 A related idea is to use transformations of the initial observation $\bx_{t}$ for $t=0$.  Suppose $(\bff_0, \bu_0)$ is independent of $\{\bu_t: t\geq 1\}$, and let $\{\phi_k: k=1,..., R\}$ be a set of sieve transformations. Then we can apply
 $
 w_{i,k} = \phi_k(x_{i,0})$ .
  These weights are correlated with $\bB$ through $\bx_0=\bB\bff_0+\bu_0$ so that the rank condition is satisfied.  The initial transformation method only requires $\{\bu_t\}$ be independent of its initial value. The similar idea has been used recently by \cite{juodis2020linear}.

       \subsection{Hadamard projection}

       We can set deterministic weights as in the statistical experimental designs:      $$
 \bW =\begin{pmatrix}
 1& 1 & 1& 1&\hdots\\
 1& -1&1 & 1\\
 1& 1& -1& 1 \\
 1& -1& -1&-1&\hdots\\
 1& 1& 1&-1\\
 1& -1& 1&-1\\
 \vdots&  \vdots& \vdots& \vdots
 \end{pmatrix}.
       $$
     So for each $2\leq k\leq R$, the $k$ th column of $\bW$  equals $(  1_{k-1}', -1_{k-1}',  1_{k-1}', -1_{k-1}',\hdots )$, where $1_m$ denotes the $m$-dimensional vector of ones.  Closely related types of matrices are known as the   Walsh-Hadamard matrices, formed by rearranging the columns so that the number of sign changes in a column is in an increasing order, and the columns are orthogonal.   Therefore,   we  can also set $\bW$ as the $N\times R$ upper-left corner submatrix of a Hadamard matrix of dimension  $2^K$ with $ K=\lceil\log_2 N\rceil$, where $\lceil.\rceil$ denotes the ceiling function.

 \section{Monte Carlo Experiments}\label{sec:sim}

 In this section we illustrate the finite sample properties of the  forecasting and inference methods based on diversified factors, and  use   four types of  weight matrices:
\begin{enumerate}


 \item[(i)]  Hadamard weight: $\bw_1 = {\bf 1}$ and $\bw_k=(  {\bf 1}_{k-1}', -{\bf 1}_{k-1}',  {\bf 1}_{k-1}', -{\bf 1}_{k-1}',\hdots )$ for $2\leq k \leq R$, where ${\bf 1}_{k-1}$ is a vector of one's of length $k-1$.

 \item[(ii)] Loading characteristics:  loadings depend on some characteristics $\bz_i$, and we apply the polynomial transformations  so that the $i$ th row of $\bW$ is   $ (g_1(\bz_i), g_2(\bz_i),...,g_R(\bz_i))$ for $  i\leq N$.  In our numerical work, we take one characteristic   and set $g_j(\bz_i) = \bz_i^j$.

 \item[(iii)] Rolling windows:   when conducting simulations for  out-of-sample forecasts, we use the trimmed PCA as described in Section \ref{trimpca}.

 \item[(iv)] Initial transformations:   we use the initial transformation
 so that the $i$ th row of $\bW$ is   $ (x_{i,0},  x_{i,0}^2,...,x_{i,0}^R)$ for $  i\leq N.$
\end{enumerate}

 We generate the data from
 $$
 \bx_{t}=\bB\bff_t+\bu_t,\quad \bB=(b_{i,k})*N^{-(1-\alpha)/2},\quad \text{with } b_{i,k}= (\bz_i^k+0.5\gamma_{i,k}) .
 $$ We   set $\bz_i=\sin(h_i)$ where both   $h_i$ and $\gamma_{i,k}$ are  independent  scalar standard normal variables. The DGP of $b_{i,k}$ is motivated from the asset pricing literature, where factor ``betas" are known to be partially explained by individual-specific observables $\bz_i$, which represent a set of time-invariant  characteristics such as individual stocks' size, momentum, values. Here we use the polynomial transformation $\bz_i^k$ to represent the effect of characteristics.  In addition, the $\gamma_{i,k}$-component captures the unobservable beta components that are not explainable by the characteristics. With the identification condition $\mathbb E(\gamma_{i,k}|\bz_i)=0$, both components in $b_{i,k}$ can be consistently estimated.  See more motivations of this model in  \cite{fan2016projected} and \cite{kim2018arbitrage}. The multiplier $N^{-(1-\alpha)/2}$ measures the strength of the factors, whereas    the spiked eigenvalue of the sample covariance grows at rate $N^{\alpha}$.   Hence larger $\alpha$ indicates stronger factors. 

The factors  are   multivariate standard normal. 
 To generate the idioscyncratic term, we set  the $N\times T$ matrix $\bU=\bSigma_N^{1/2}\bar \bU\bSigma_T^{1/2} ;$ here $\bar\bU$ is an $N\times T$ matrix, whose entries  independent   standard normal.   The $N\times N$ matrix $\bSigma_N$ and the $T\times T$ matrix $\bSigma_T$ respectively govern the cross-sectional and serial correlations of $u_{it}$. We set
 $
 \bSigma_T=(\rho_T^{|t-s|})_{st}
 $
 , and use a sparse cross-sectional covariance:
 \begin{equation}\label{eq6.1add}
 \bSigma_N= \diag\{\underbrace{\bA,\cdots,\bA}_{\text{ $n$ of them}}, \bI\},\quad \bA= (\rho_N^{|i-j|})
 \end{equation}
 where $\bA$ is a small four-dimensional block matrix and $\bI$ is $(N-4n)\times (N-4n)$ identity matrix so that $\bSigma_N$ has a block-diagonal structure. We fix $n=3$  and $\rho_N=0.7$. The numerical performances are studied in the following subsection with various choice of $\rho_T$ to test about the sensitivity against serial correlations.

\subsection{Covariance estimation}

We first study the performance of estimating $\bSigma_u$.   To do so, we set $r=1$ and respectively calculate $\widehat\bSigma_u$ using $R=r,\cdots,r+3$. As estimating $\bSigma_u$ is particularly important in asset pricing models,  we use the    loading characteristic weights $w_{i,k}=\bz_i^k$, $k=1,..., R$, as the characteristic $\bz_i$ is often directly observable along with the return data.

For comparison purposes, we also estimate $\bSigma_u$ using two benchmark estimators: 

(i) The PC-estimator for factors with $R=r$ (the POET method by \cite{POET}). So the PC-estimator in this simulation assumes   the true number of factors $r=1$ to be known; 

  (ii) The known-factor method.  We use the true factors, and  estimate loadings and $u_{it}$ by OLS, followed by SCAD-thresholding.   
  
  We  set two serial dependence scenarios: $\rho_T =0.1$ (weak serial dependence) and $\rho_T=0.7$  (strong serial dependence), as well as two factor-strength scenarios: 
  $\alpha=1$   and   $\alpha=0.5$.

  Figure \ref{fig1}  plots $\|\widehat\bSigma_u-\bSigma_u\|$ and  $\|\widehat\bSigma_u^{-1}-\bSigma_u^{-1}\|$, averaged over 100 replications, as $N=T$ grows. 
  While all  estimators perform similarly,
the POET-estimator is not always better than the diversifying projection (DP).
For estimating $\bSigma_u$, both the DP with $R=r$ and the known factor method are overall better than the POET estimator, followed by DP with other choices of $R$. This comparison is reasonable, reflecting the robustness of DP to the serial conditions and strength of factors.  Perhaps what is surprising is the comparison for estimating the inverse covariance. In all four scenarios of the factor strength and serial correlations, the DP with $R=r$ performs the worst among the six estimators, and DP with over estimated $R$ is in general better than 
both the known factor method and the POET. Our interpretation of this is that we set    relatively strong cross-sectional correlations in the data generating process, making $\bSigma_u^{-1}$   more unstable. The use of  more diversified weights provides extra information to help stabilizing the  inverse covariance estimator.


\begin{figure}[htbp!]
	\begin{center}
		\includegraphics[width=12cm]{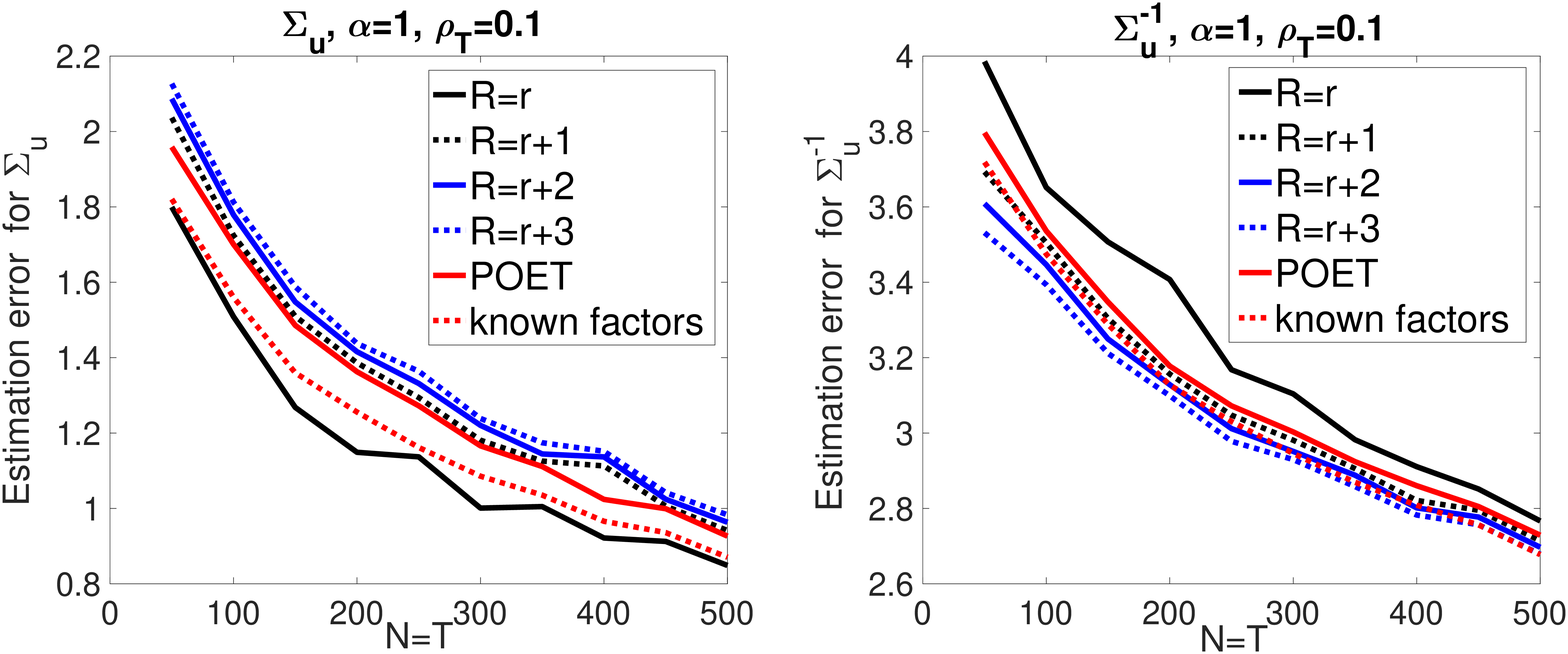} 
		\includegraphics[width=12cm]{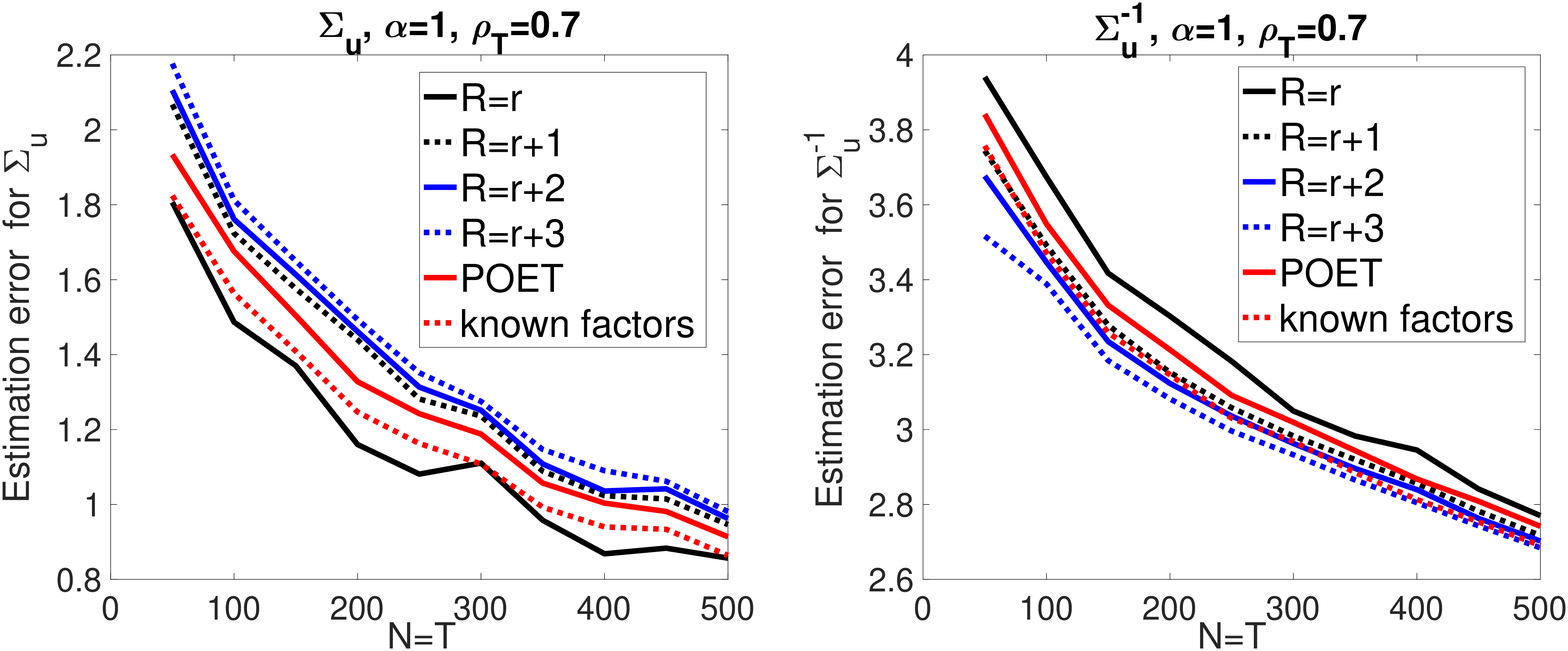} \\
		\includegraphics[width=12cm]{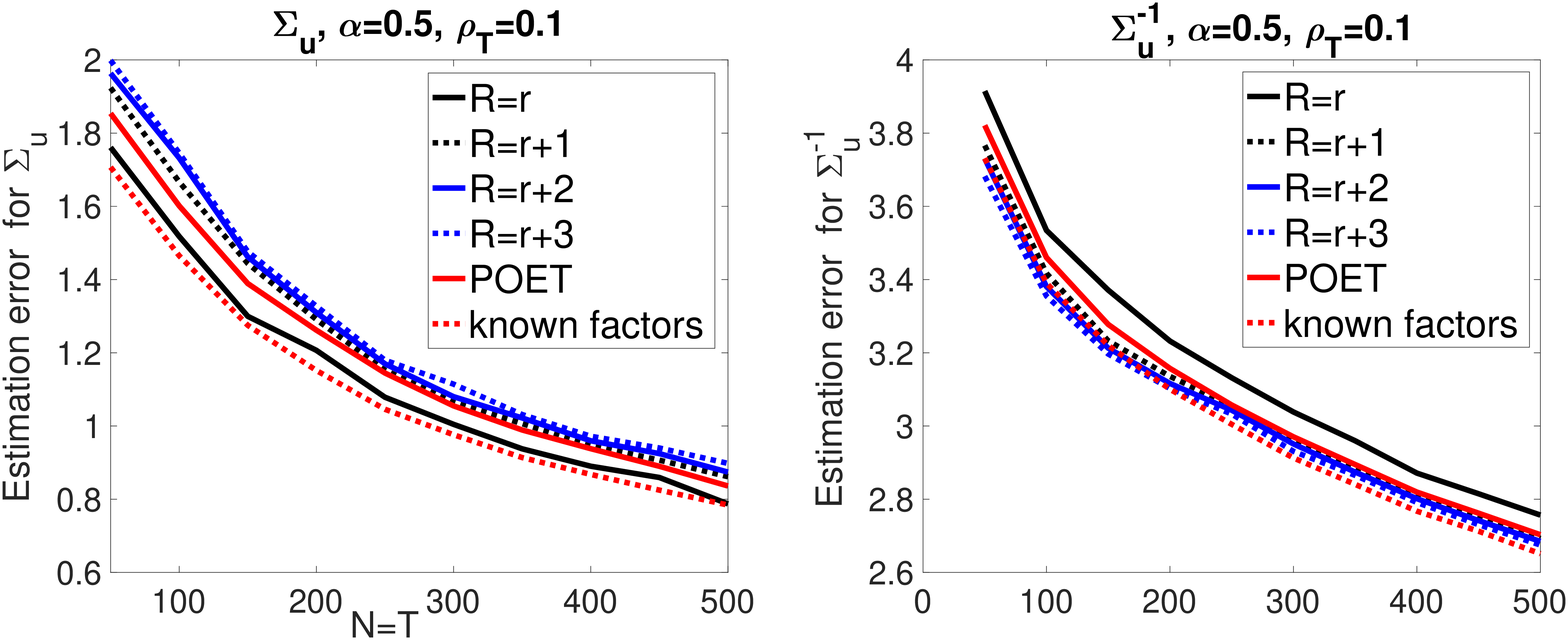} \\
		\includegraphics[width=12cm]{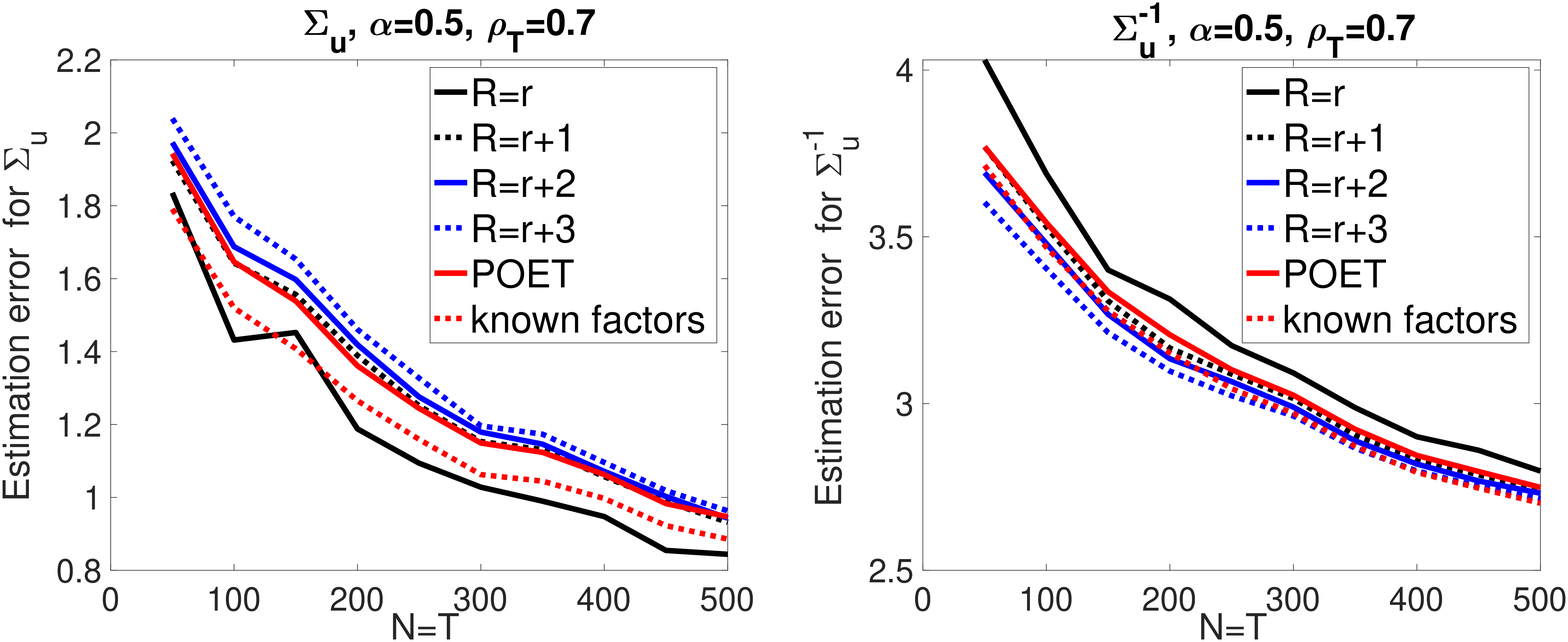} 

		\caption{\footnotesize  The estimation errors in operator-norm $\|\widehat\bSigma_u-\bSigma_u\|$ (left) and $\|\widehat\bSigma_u^{-1}-\bSigma_u^{-1}\|$ (right) as the dimension increases, averaged over 100 replications. We set $N=T$. Here $R=r,\cdots,r+3$ correspond to the diversified factor estimators using $R$ number of working factors. Characteristic weights are used. Here $\alpha$ measures  the factor strength and $\rho_T$ is the serial correlation. 
		\label{fig1} }
	\end{center}
	
\end{figure}

\subsection{Out-of-sample forecast}\label{forecastsimulation}

We assess the performance of the proposed factor estimators on out-of-sample forecasts. Consider the following forecast model
$$y_{t+1} =\beta_0+\beta y_t+ \balpha'\bff_t+\varepsilon_{t+1}$$
where we set $r=\dim(\bff_t)=2$, $(\beta_0,\beta)=(1.5, 0.5)$,  and $\balpha=(1,1)'$.  In addition, $ \varepsilon_t $  are independent   standard normal. The data generating process for   $ \bx_{t}= \bB\bff_{t}+\bu_{t}$ is the same as before, in the presence of both serial and cross-sectional correlations.
We conduct  one-step-ahead  out-of-sample forecast  $m$ times using a moving window of   size $T$. Here $T$ is also the sample size for estimations.  We simulate $m+T$ observations in total.  For each $t=0,\cdots,m-1$, we use the data $\{(\bx_{t+1}, y_{t+1}),\cdots,( \bx_{t+T}, y_{t+T})\}$ to conduct  one-step-ahead forecast of  $y_{t+T+1}$. Specifically,    we estimate the factors using $\{\bx_{t+1},\cdots,\bx_{t+T}\}$, and obtain $\{\widehat \bff_{t+1},\cdots,\widehat \bff_{t+T}\}$. The coefficients    in the forecasting regression is then estimated by the OLS, denoted by $(\widehat\beta_{0, t+T}, \widehat\beta_{t+T}, \widehat\balpha_{t+T})$.  We then forecast $y_{t+T+1}$  by $$\widehat y_{t+T+1|t+T}= \widehat \beta_{0, t+T} +\widehat\beta_{t+T} y_{t+T}+\widehat\balpha_{t+T}'\widehat\bff_{t+T}.$$  
 Such a procedure continues for $t=0,\cdots, m-1$.

\begin{table}[htp]
\small
\caption{Out-of-Sample  $ \text{MSE}(\text{M})/\text{MSE}(\text{PC})$  for three types of estimators.}
\begin{center}
\begin{tabular}{ccc|ccccccccc}
\hline
\hline
  &  &  & \multicolumn{3}{c}{Characteristic weights} & \multicolumn{3}{c}{Rolling window weights}  &  \multicolumn{2}{c}{GDF} & KF     \\
     &  &  & \multicolumn{9}{c}{$R $}    \\
$\rho_T $& $N$ & $T$ & $r$ & $r+1$ & $r+3$ & $r$ & $r+1$ & $r+3$ & $3$ & $4$ & $r$ \\
\hline
&&&&&&\\
     &  &  & \multicolumn{9}{c}{$\alpha=1 $}   \\
0  & 100 & 50 & 1.141 & 1.090 &  1.109 &  0.968 &  1.001 &  1.010 &    0.991&     1.016 & 1.007  \\
&    & 100 &  0.998 & 0.980 & 1.035 &  0.979 &1.039 & 1.046 &  1.008 &1.009 &    1.002 \\
&&&&&&\\
0.5&   & 50 &  0.996 &  1.008 &  0.965 &  0.993 & 1.018 & 1.055&   1.000 &0.996 &0.986  \\
 & & 100&  0.885 & 0.886 & 0.917 & 0.937& 0.922&  0.939 &  0.995&  0.997& 1.005 \\

&&&&&&\\
0.9 & & 50 & 0.602 & 0.621 & 0.637 &  0.608&0.620 & 0.680 &  0.763 &  0.772 &   1.023  \\
 &   & 100 &0.434 & 0.458 &  0.482 & 0.422 & 0.419 & 0.450 &  0.863 &      0.578&    0.985\\
&&&&&&\\
 &  &  & \multicolumn{9}{c}{$\alpha=0.2 $}   \\
0  &  & 50 & 0.876 &0.913  &0.987  & 1.072  &1.059  & 1.071  &  0.991  &  0.985  & 1.003 \\
   &  & 100 &  0.931 & 0.906 &    0.966 &    1.065&  1.114 & 1.156 &0.996&  1.012 &  0.992 \\
&&&&&&\\
0.5  &  & 50 &   0.891 &0.897 &    1.044 &1.059 &  1.082& 1.149& 1.002  &   0.981   & 0.958\\
   &   & 100 &  0.972 & 0.963&   0.970 &0.868 &    0.793&   0.817 & 0.968 &    0.981  &  1.007  \\

&&&&&&\\
0.9  &    & 50 & 0.478 & 0.513 &  0.647 &0.713& 0.731 & 0.688 &  0.953
 & 0.745&  0.966 \\
   &   & 100&  0.762 & 0.765 & 0.767 & 0.788 &0.806
    & 0.849 & 0.927 &  0.851 &  0.951\\

\hline

\hline
\end{tabular}
\end{center}
\label{tab2}
{\footnotesize  Reported are the out-of-sample relative MSEs. The benckmark PC-estimator uses the true number of factors. The dimension $N=100$ is fixed.  The diversified projection uses  $R$ estimated factors with two types of weights: characteristic weights and rolling window weights. In addition, the columns of GDF estimates factors from the   generalized dynamic factor model of \cite{forni2005generalized}, with $R$ number of dynamic factors. The Matlab codes for implementing \cite{forni2005generalized} and \cite{HL} are downloaded from Matteo Barigozzi's website \texttt{www.barigozzi.eu/codes.html.} The column of KF refers to the Kalman filtering developed by \cite{doz2011two}, which uses the true $r$ number of factors. Both GDF and KF specifically estimate dynamic factors.}
\end{table}%

 We compute the diversified factor estimators using the two types of weights, with $R=r, r+1,r+3$ as the working number of factors.  As for the  moving  windows weight, we assume there is a historical time series $\bx_t=\bB_1\bff_t+\bu_t$, for $t=-T,...,0$, and the loadings $\bB_1$ is correlated with $\bB$ in the sense that $ \bB_1=0.8\bB+0.5 \bZ$, where $\bZ$ is multivariate standard normal.
 We then apply the moving window method to create $\bW$ as outlined in Section \ref{trimpca}.
 Though the theory for the moving window  weights requires serial correlation $\rho_T=0$, we nevertheless set $\rho_T=0, 0.5$ and 0.9 to examine the performance under serially correlated series.

The benchmark method is the PC-estimator, which uses the true number of factors. 
In addition,  we also consider two well known methods that specifically estimate factor dynamics:

(i)  GDF: the  generalized dynamic factor model of \cite{forni2005generalized}. The selection criterion of \cite{HL} recommended using, on average, three dynamic factors, so we    use $R= 3, 4$ numbers of dynamic factors.

(ii) KF: the two-step Kalman filtering of \cite{doz2011two}. In the first step factors are preliminarily estimated and fit  a VAR model; in the second step,  Kalman smoother  is applied to calculate the projection onto the  observations.  For this approach, we  use $R=2$, the true number of factors. 

 For each method M, we calculate the   mean squared out-of-sample forecasting error:
$$
\text{MSE}(\text{M})=\frac{1}{m}\sum_{t=0}^{m-1} (y_{t+T+1}-\widehat y_{t+T+1|t+T})^2,
$$
and report the relative MSE to the PC method:
$
 \text{MSE}(\text{M})/\text{MSE}(\text{PC}).
$
It is worthwhile  to emphasize that this study does  not aim to beat the PC-method. In fact, the PC-estimator yields the optimal rank $r$-estimation of the low-rank structure, in the sense that the  estimated  low-rank component $\bB\bF'$ satisfies:
$
\widehat \bB_{pc}\widehat\bF'_{pc}=\arg\min_{\text{rank}(\bA)=r} \|\bX-\bA\|_F^2.
$
So when the number of factors $r$ is correctly specified and the time series dependence is not strong,  the PC-estimator enjoys some optimal property. Nevertheless we use PC as the benchmark as it is the most commonly used   in this literature. We aim to   see how well  the proposed DP method performs  relative to the benchmark.

The results are reported in Table \ref{tab2}   for $m=50$, and is computed based on one set of simulation replications. We see that   the DP with various $R$ and Generalized DF are in most scenarios similar to the PC-estimator, and DP  outperforms under the strong serial correlations.  In all cases, Kalman filtering is comparable with PC, including the case of strong serial correlations.

\subsection{Post-selection inference}
We now study the inference  for the  effect  of $\bg_t$ in the following factor-augmented model
 \begin{eqnarray*}
  y_{t} &=& \bbeta\bg_t+ \bnu'\bx_t +\eta_t,\cr
  \bg_t&=& \btheta'\bx_t +\bvarepsilon_{g,t}\cr
  \bx_t&=&\bB\bff_t+\bu_t,
  \end{eqnarray*}
where both $\bnu$ and $\btheta$ are set to high-dimensional sparse vectors.  The goal is to make inference about $\bbeta$, using the factor-augmented post-selection inference. We generate $\bu_t\sim \mathcal N(0,\bSigma_u)$, $(\eta_t, \bvarepsilon_{g,t})\sim \mathcal N(0,\bI)$. We set $(\bu_t, \bvarepsilon_{g,t}, \eta_t)$ be serially independent, but still allow the same cross-sectional dependence among $\bu_t$. This allows us to  focus on the effect of over-estimating factors.
The $r$-dimensional $\bff_t$ are  independent standard normal.  We set the true $\bbeta=1$,  $\btheta=\bnu=(1,-1.5, 0.5,0,...,0)$ and $T=N=200.$

We employ the diversified factor estimator described in Section \ref{sec:4} with various working number of factors $R$, and compare with the benchmark ``double-selection" method of \cite{belloni2014inference}. In particular, we consider two settings:

(i) $r=0$: there are no factors so $\bx_t$ itself is weakly dependent.

(ii) $r=2$: there are two factors driving  $\bx_t$. Set $\alpha=1$ so both factors are strong. \\
We calculate the standardized estimates:
$
z:=\widehat\sigma_{\eta, g}^{-1}\widehat \sigma_g^2\sqrt{T}(\widehat\beta-\beta),
$
where the standard error is the estimated feasible one.
  Our theory shows that the sampling distribution of $z$ should be approximately standard normal.

Figures  5.2  and \ref{fig3} plot the  histograms of the standardized estimates over 200 replications, superimposed with the standard normal density. The histogram is scaled to be a density function.
We present the results when  the initial transformation are used as weights   for the diversified factors. The results from  characteristics and Hadamard weights are very similar.     When $r=0$, while it is expected that the double selection performs very well, as is shown in Figure \ref{fig3},  using $R\geq 1$ factors  also produces  $z$-statistics whose distribution is also close to the standard normality. This shows that the factor-augmented method is robust to the absence of factor structures. On the other hand, when $r=2$,  the factor-augmented method continues to perform well. In contrast,   the double selection is severely biased, and the distribution of its $z$-statistic is far off from the standard normality.

      \begin{minipage}{0.45\linewidth}
          \begin{figure}[H]\label{fig2}
                 \hspace{-0.15\linewidth}
        \includegraphics[width=9cm]{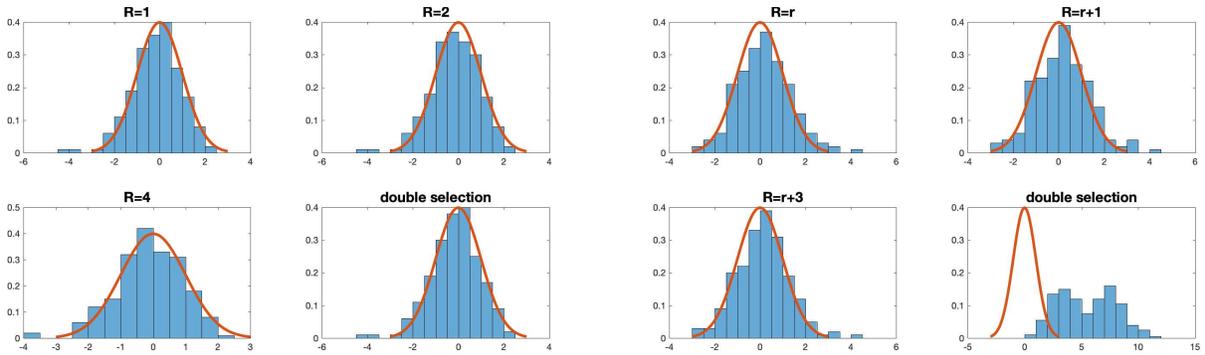}
              \caption{true $r=0$}
          \end{figure}
      \end{minipage}
      \begin{minipage}{0.45\linewidth}
          \begin{figure}[H]\label{fig3}
            \includegraphics[width=9cm]{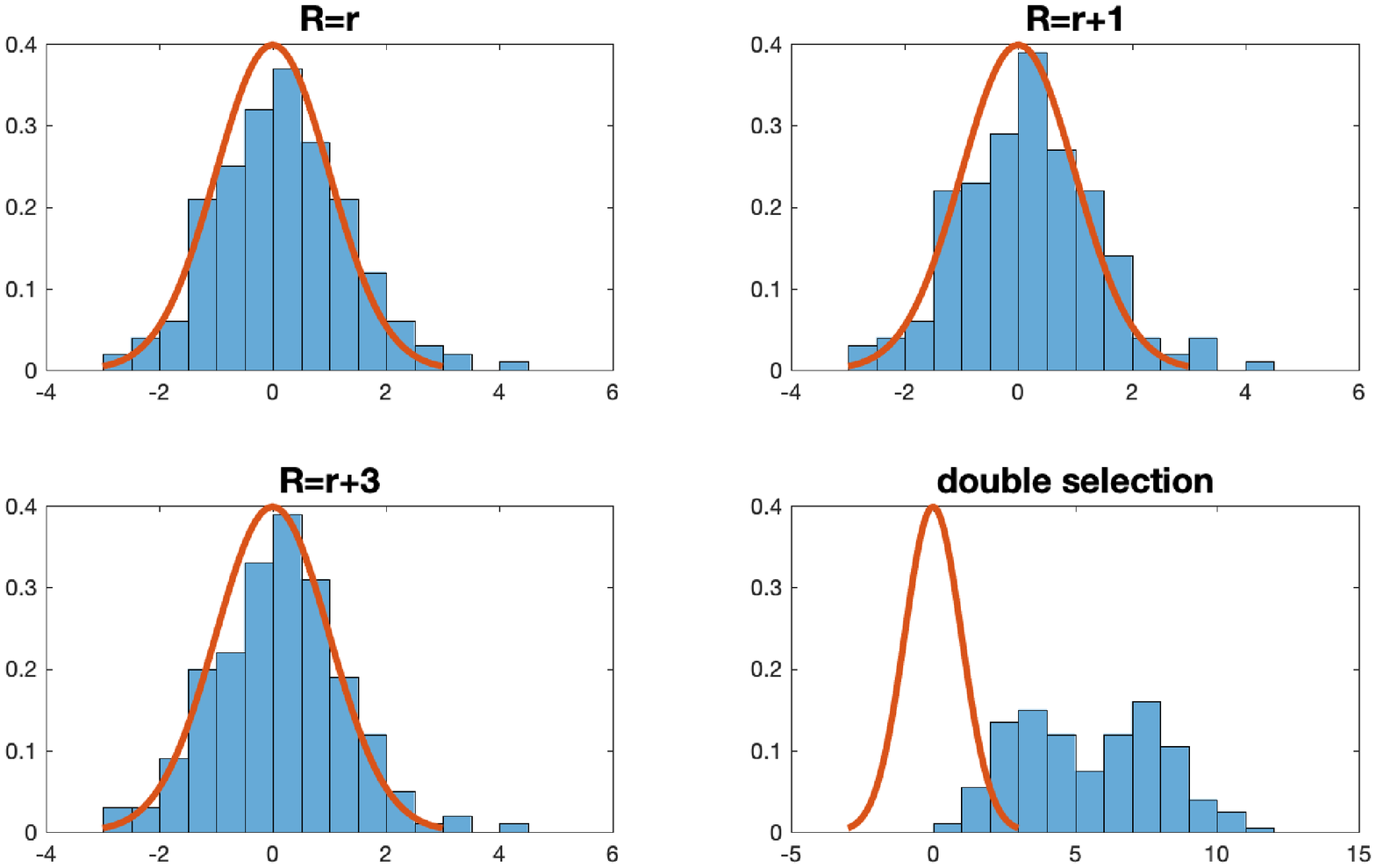}
              \caption{true $r=2$}
          \end{figure}
      \end{minipage}

 \vspace{1em}
 { \footnotesize
	 The first three panels employ the diversified factor estimator with $R$ number of working factors. The last panel uses the double selection, which directly selects among $\bx_t$. The weights used are the initial transformations ($t=0$) so   that the $i$ th row of $\bW$ is $ (x_{i,0}, x_{i,0}^2,...,x_{i,0}^R)$ for $  i\leq N$. }

\subsection{Testing the specification of   empirical factors}\label{sec:5.4}

In the last simulation study, we study the size and power of the test statistic for
$H_0: \bg_t=\btheta\bff_t$ for some $r\times r$ invertible matrix $\btheta$. Here $\bg_t$ is  a vector of known ``empirical factors" that applied researchers propose to   approximate the true factors.  We generate
$$
\bg_t= \btheta\bff_t+ \gamma \bh_t,  \quad t\leq T,
$$
where $\btheta$ is an $r$-dimensional identity matrix, and  $(\bff_t, \bh_t)\sim \mathcal N(0, \bI)$.  Here $\gamma$ governs the strength of the alternatives. We assume that $\bu_t$ is serially independent  normal generated  from $\mathcal N(0,\bSigma_N)$, with $\bSigma_N$ as in (\ref{eq6.1add}), pertaining the same cross-sectional dependence. We set $R=r=2$ and fix $N=200$.
In each of the simulations, we calculate  the test statistic as defined in Section \ref{sec:5}, and set the significance level to 0.05.
We use the SCAD-thresholding to estimate $\bSigma_u$ for both $\widehat\MEAN$ and $\widehat\sigma$.

 Table \ref{tab3} presents the rejection probability over 1000 replications, with $\gamma=0$ representing the size of the test. Above all, the results look satisfactory with controlled size and reasonable powers, while  weights using initial transformations have some size distortions.

\begin{table}[htp]
\small
\caption{  Probability of rejection at  level 0.05. $\gamma$ represents the strength of alternatives. }
\begin{center}
\begin{tabular}{cc|ccc}
\hline
\hline

 $\gamma$ & $T$ & Characteristic weights  & Hadamard  weights  & Initial transformation\\
 \hline
 &&&&\\
0 & 100 & 0.054 & 0.046 & 0.065 \\
 & 200 & 0.052 & 0.047  & 0.074\\
  &&&&\\
0.2 & 100 & 1.000 & 0.998 & 1.000\\
 & 200 & 0.975 & 1.000 &1.000\\

\hline

\hline
\end{tabular}
\end{center}
\label{tab3}
\end{table}%

\section{Conclusion}
We propose simple estimators of the latent factors  using cross-sectional  projections of the panel data, by  weighted averages. These  weights are chosen to diversify away the idiosyncratic  components,  resulting in ``diversified factors".  Because the projections are conducted cross-sectionally,  they are robust to serial conditions, easy to analyze due to data-independent weights, and work even for finite length of time series. We  formally prove that  this procedure is  robust to over-estimating the number of factors, and  illustrate it  in several  applications.   We also recommend several choices for the diversified weights.

\appendix

 \section{Technical Proofs}

 Throughout the proofs, we use $C$ to denote a generic positive constant. Recall that $\nu_{\min}(\bH)$ and $\nu_{\max}(\bH) $ respectively denote the minimum and maximum nonzero singular values of $\bH . $  In addition, $\bP_\bA=\bA(\bA'\bA)^{-1}\bA'$ and $\bM_\bA=\bI-\bP_\bA$ denote the projection matrices of a matrix $\bA$. If $\bA'\bA$ is singular, $(\bA'\bA)^{-1}$ is replaced with its Moore-Penrose generalized inverse $(\bA'\bA)^{+}$.
 Let $\bU$ be the $N\times T$ matrix of $u_{it}$.  Recall that   $R=\dim(\widehat\bff_t)$ and $r=\dim(\bff_t)$.
 
We use $\|\bA\|$ and $\|\bA\|_F$ to respectively denote  the operator norm and Frobinus norm. 
Finally, we define $\|\bA\|_\infty$ as follows:  if $\bA$ is an $N\times K$ matrix with $K= R$ or $r$, then $\|\bA\|_\infty=\max_{i\leq N}\|\bA_i\|$ where $\bA_i$ denotes the $i$ th row of $\bA$;  if $\bA$ is a $K \times N$ matrix with $K= R$ or $r$, then $\|\bA\|_\infty=\max_{i\leq N}\|\bA_i\|$ where $\bA_i$ denotes the $i$ th column of $\bA$; if $\bA$ is an $N\times N$ matrix, then $\|\bA\|_\infty=\max_{i,j\leq N}|A_{ij}|$ where $A_{ij}$ denotes the $(i,j)$ th element of $\bA.$

Throughout the proof, all $\E(.)$, $\E(.|.)$ and $\Var(.)$ are calculated conditionally on $\bW$. 

  \subsection{A key Proposition for asymptotic analysis when $R\geq r$}


 \begin{prop}
 	\label{la.2}  Suppose $R\geq r$ and $T,N\to\infty$. Also suppose $\bG$ is a $T\times d$ matrix so that  $\E(\bU|\bG)=0$,  $\frac{1}{T}\|\bG\|^2=O_P(1)$, for some fixed dimension $d$, and  Assumption \ref{ass2.1} - \ref{ass2.3} hold. 
	In addition, for each $\bK\in\{\bI_T,\bM_\bG\},$ suppose $\lambda_{\min}(\frac{1}{T}\bF'\bK\bF)>c>0$. Then \\
 	(i)   $\lambda_{\min}(\frac{1}{T}\widehat\bF'\bK\widehat\bF)\geq cN^{-1}$
 	with probability approaching one for some  $c >0$,\\
 	(ii) $\|\bH'(\frac{1}{T}\widehat\bF'\bK\widehat\bF)^{-1} \|=O_P(\nu_{\min}^{-1}+\sqrt{\frac{N}{T}})$, and  $\|\bH'(\frac{1}{T}\widehat\bF'\bK\widehat\bF)^{-1} \bH\|=O_P(1 )$.\\
 	(iii) $\|\bH'(\frac{1}{T}\widehat\bF'\bK\widehat\bF)^{-1}\bH- \bH'(\bH\frac{1}{T} \bF' \bK\bF\bH')^{+}\bH\|=O_P(  \frac{1}{N\nu^2_{\min}} + \frac{1}{T})$, and
 	$\frac{1}{T}\bG'(\bP_{\widehat\bF}-\bP_{\bF\bH'})\bG= O_P(  \frac{1}{N\nu^2_{\min}} + \frac{1}{T}).$
 \end{prop}

 \begin{proof} 
 	The proof applies for both $\bK=\bI_T$ and $\bK=\bM_\bG$. In addition, the proof depends on results in the later Lemma \ref{la.1}; the   latter is proved independently which does not depend on this proposition.
	Write 
	$
	\nu_{\min}:=\nu_{\min}(\bH), $ and $  \nu_{\max}:=\nu_{\max}(\bH).
	$

 	First, it is easy to see
 	$$
 	\widehat\bF= \bF\bH'+\bE.
 	$$
 	where $\bE= (\be_1,\cdots,\be_T)'=\frac{1}{N}\bU'\bW$, which is $T\times R.$
 	Write
 	$$\bDelta:= \frac{1}{T} \E\bE' \bE+ \frac{1}{T}\bH\bF'\bK\bE+\frac{1}{T}\bE'\bK\bF\bH' +\frac{1}{T} (\bE' \bE-\E\bE'\bE)+\bDelta_1$$
 	where $\bDelta_1=0$ if $\bK=\bI_T$ and $\bDelta_1=-\frac{1}{T} \bE'\bP_\bG\bE$ if $\bK=\bM_\bG$.
 	
 	(i) We have
 	$$
 	\frac{1}{T}\widehat\bF'\bK\widehat\bF=\frac{1}{T}\bH\bF'\bK\bF\bH'  +\bDelta.
 	$$ By assumption $\lambda_{\min}(\frac{1}{T }\E \bU\bU' ) \geq c_0,$
 	so
 	$
 	\lambda_{\min}(\frac{1}{T}\E\bE'\bE)   \geq \lambda_{\min}(\frac{1}{T }\E \bU\bU' )
 	\lambda_{\min}(\frac{1}{N^2} \bW' \bW) \geq c_0N^{-1}
 	$ for some $c_0>0.$
 	In addition,  Lemma  \ref{la.1}  shows $\frac{1}{T}(\bE'\bE-\E\bE'\bE)+\bDelta_1 =O_P(\frac{1}{N\sqrt{T }})$. Hence $\| \frac{1}{T}(\bE'\bE-\E\bE'\bE)+\bDelta_1 \|\leq \frac{1}{2}\lambda_{\min}(\frac{1}{T}\E\bE'\bE) $ with large probability.  We now continue the argument conditioning on this event. 

	 Now let $\bv$ be the unit vector so that $\bv'\frac{1}{T}\widehat\bF'\bK\widehat\bF\bv=\lambda_{\min}(\frac{1}{T}\widehat\bF'\bK\widehat\bF)$ and let 
	  $$\eta_v^2:= \frac{1}{T} \bv'\bH\bF' \bK\bF\bH' \bv.$$ 
	  Because $\bv'\frac{1}{T}\widehat\bF'\bK\widehat\bF\bv=\eta_v^2+\bv'\bDelta\bv,$
 we have 	$$
 	\lambda_{\min}(\frac{1}{T}\widehat\bF' \bK\widehat\bF)
\geq\eta_v^2
 	+2\bv'\frac{1}{T}\bH\bF' \bK\bE\bv
	+ \frac{c_0}{2N} . 	$$
	 	If $\bv'\bH=0$ then $\lambda_{\min}(\frac{1}{T}\widehat\bF' \bK\widehat\bF)\geq \frac{c_0}{2N} .  $ If $\bv'\bH\neq 0$ then $\eta_v^2\neq 0$ with large probability because $\frac{1}{T}  \bF' \bK\bF$ is positive definite.  Now let
 	$$
 	X:= (\frac{\eta_v^2}{TN})^{-1/2}2\bv'\frac{1}{T}\bH\bF' \bK\bE\bv  ,\quad 2\bv'\frac{1}{T}\bH\bF' \bK\bE\bv= X\sqrt{\frac{\eta_v^2}{TN}}.
 	$$ Then 
 	$$
 	\lambda_{\min}(\frac{1}{T}\widehat\bF' \bK\widehat\bF)
 	\geq \eta_v^2
 	+X\sqrt{\frac{\eta_v^2}{TN}}+ \frac{c_0}{2N}.
 	$$

	Suppose for now  $X=O_P(1)$, a claim to be proved later.  Then consider two cases.
	
 	 In case 1,  $\eta_v^2\leq 4|X|\sqrt{\frac{\eta_v^2}{TN}}$. Then $|\eta_v|\leq 4|X| \frac{1 }{\sqrt{TN}}$ and 
 	$$
 	\lambda_{\min}(\frac{1}{T}\widehat\bF' \bK\widehat\bF)
 	\geq   \frac{c_0}{2N}
 	-|X||\eta_v| \frac{1}{\sqrt{TN}}\geq
 	\frac{c_0}{2N}-4|X|  ^2 {\frac{1 }{TN}} \geq  \frac{c_0}{4N}
 	$$
 	where the last inequality holds for $X=O_P(1)$ and as $T\to\infty$, with probability approaching one.

	 In case 2, $\eta_v^2> 4|X|\sqrt{\frac{ \eta_v^2}{TN}}$, then
 	$$
 	\lambda_{\min}(\frac{1}{T}\widehat\bF' \bK\widehat\bF)
 	\geq \eta_v^2
 	-|X|\sqrt{\frac{ \eta_v^2}{TN}}+ \frac{c_0}{2N}
 	\geq \frac{3}{4}\eta_v^2+ \frac{c_0}{2N}\geq \frac{c_0}{2N}.
 	$$
 	In both cases, $\lambda_{\min}(\frac{1}{T}\widehat\bF' \bK\widehat\bF)>c_0/N$ for some $c_0>0$ with overwhelming probability.
	
	It remains to argue $X=O_P(1)$.  By the assumption $\lambda_{\min}(\frac{1}{T}\bF'\bK\bF)>c>0$, we have 
	$$\eta_v^2\geq \lambda_{\min}(\frac{1}{T}\bF'\bK\bF)\bv'\bH\bH'\bv>c\|\bv'\bH\|^2.$$ 
	In addition,  Lemma \ref{la.1} shows $ \|\frac{1}{T}\bF'\bE\|^2 =O_P (\frac{1}{{TN}})$ and  $\|\frac{1}{T}\bG'\bE\|^2= O_P (\frac{1}{{TN}})$. With the condition  $\frac{1}{T}\|\bG\|^2=O_P(1)$,  we reach 
	$
 \|\frac{1}{T}\bF'\bM_\bG\bE\|^2\leq  O_P (\frac{1}{{TN}})+ \|\bF'\bG(\bG'\bG)^{-1}\|^2\|\frac{1}{T}\bG'\bE\|^2=O_P (\frac{1}{{TN}}).
	$
 Therefore $\|\frac{1}{T}\bF'\bK\bE\|^2= O_P (\frac{1}{{TN}})$ and consequently,
	 	$$
	|X|^2\leq 4TN \eta_v^{-2} \|\bv'\bH\|^2\|\frac{1}{T}\bF' \bK\bE\|^2
	\leq  O_P(1) \eta_v^{-2} \|\bv'\bH \|^2\leq  O_P(1)c^{-1} \|\bv'\bH\|^{-2}\|\bv'\bH \|^2=O_P(1).
	$$

 	(ii)  Write $\bar \bH:= \bH(\frac{1}{T}\bF'\bK\bF)^{1/2}$ and $\bS:=\frac{N}{T}\E\bE'\bE=\frac{1}{N}\bW'\bSigma_u\bW$. Then
 	\begin{equation}\label{eqa.1adddd}\frac{1}{T}\widehat\bF'\bK\widehat\bF=\bar \bH \bar\bH'
 	+\frac{1}{N}\bS+ \frac{1}{T}\bH\bF'\bK\bE+\frac{1}{T}\bE'\bK\bF\bH' +\bDelta_2
 	\end{equation}
 	where  we proved in (i) that $\|\bDelta_2\|=\|\frac{1}{T}(\bE'\bE-\E\bE'\bE)+\bDelta_1 \|=O_P(\frac{1}{N\sqrt{T }}).$ 
 	Also  all eigenvalues of $\bS$ are bounded away from both zero and infinity.
In addition, $\bar\bH$ is a $R\times r$ matrix with $R\geq r$, whose Moore-Penrose generalized inverse  is $\bar \bH^+=(\frac{1}{T}\bF'\bK\bF)^{-1/2} \bH^+ $.  Also,  $\bar \bH$ is of rank $r$. Let $$\bar \bH'=\bU_{\bar H}(\bD_{\bar H},0)\bE_{\bar H}'$$ be the singular value decomposition (SVD) of $\bar \bH'$, where 0 is present when $R>r$.   Since  $\lambda_{\min}(\frac{1}{T}\bF'\bK\bF)>c>0$, we have $\lambda_{\min}(\bD_{\bar H})\geq c\nu_{\min}$ where $\nu_{\min}:=\nu_{\min}(\bH)$.


 	The proof is divided into several steps.

 	Step 1. Show $\|\bar\bH'(\bar \bH\bar\bH'+\frac{a}{N}\bI)^{-j}\bar\bH\|
 	=O_P(\nu_{\min}^{-(2j-2)})$    for any fixed $a>0$ and $j=1,2.$
 	 
 	Because $\lambda_{\min}(\bD_{\bar H})\geq c\nu_{\min}$,  for $j=1,2$,  
	 	$$
 	\|\bar\bH'(\bar \bH\bar\bH'+\frac{a}{N}\bI)^{-j}\bar\bH\|
 	=\|\bU_{\bar H}(\bD_{\bar H}^2(\bD_{\bar H}^2+\frac{a}{N}\bI)^{-j},0)\bU_{\bar H}'\|
 	=\|\bD_{\bar H}^2(\bD_{\bar H}^2+\frac{a}{N}\bI)^{-j}\|\leq \|\bD_{\bar H}^{-2j+2}\|.
 	$$
 	
 	Step 2. Show  $\|\bar\bH'(\bar \bH\bar\bH'+\frac{1}{N}\bS)^{-1}\bar\bH\|=O_P(1)$.
 	
 	Let $0<a<  \lambda_{\min}(\bS)$ be a constant. Then
 	$ (\bar \bH\bar\bH'+\frac{a}{N}\bI)^{-1}-(\bar \bH\bar\bH'+\frac{1}{N}\bS)^{-1}$ is positive definite.  (This is because, if both $\bA_1$ and $\bA_2-\bA_1$ are positive definite, then so is $\bA_1^{-1}-\bA_2^{-1}$.) Let $\bv$ be a unit vector so that
 	$\bv'\bar\bH'(\bar \bH\bar\bH'+\frac{1}{N}\bS)^{-1}\bar\bH\bv=\|\bar\bH'(\bar \bH\bar\bH'+\frac{1}{N}\bS)^{-1}\bar\bH\|$. Then
 	$$
 	\|\bar\bH'(\bar \bH\bar\bH'+\frac{1}{N}\bS)^{-1}\bar\bH\|\leq
 	\bv'\bar\bH'(\bar \bH\bar\bH'+\frac{a}{N}\bI)^{-1}\bar\bH\bv
 	\leq \|\bar\bH'(\bar \bH\bar\bH'+\frac{a}{N}\bI)^{-1}\bar\bH\|.
 	$$
 	The right hand side is $O_P(1)$ due to step 1.
 	
 	Step 3. Show  $\|\bar\bH'(\bar \bH\bar\bH'+\frac{1}{N}\bS)^{-1} \|=O_P(\nu_{\min}^{-1}).$

 	Fix any $a>0$.
 	Let $\bM=\bar\bH'(\bar \bH\bar\bH'+\frac{a}{N}\bI)^{-1}.$ By step 1, $\|\bM\|
 	= \|\bar\bH'(\bar \bH\bar\bH'+\frac{a}{N}\bI)^{-2}\bar\bH\|^{1/2}=O_P(\nu_{\min}^{-1})$. So 
 	\begin{eqnarray*}
 		\|\bar\bH'(\bar \bH\bar\bH'+\frac{1}{N}\bS)^{-1} \|
 		&\leq&\|\bM\|+ \|\bar\bH'(\bar \bH\bar\bH'+\frac{1}{N}\bS)^{-1}-\bM \|\cr
 		&=^{(1)}&\|\bM\|+\|\bar\bH'(\bar \bH\bar\bH'+\frac{a}{N}\bI)^{-1} (\frac{1}{N}\bS-\frac{a}{N}\bI)   (\bar \bH\bar\bH'+\frac{1}{N}\bS)^{-1}  \|\cr
 		&\leq&\|\bM\|+\frac{C}{N}\|\bM\|  \|  (\bar \bH\bar\bH'+\frac{1}{N}\bS)^{-1}  \|\cr
 		&\leq^{(2)}& \|\bM\|(1+O_P(1))=O_P(\nu_{\min}^{-1}).
 	\end{eqnarray*}
 	   (1)   used $\bA_1^{-1}-\bA_2^{-1}=\bA_1^{-1}(\bA_2-\bA_1)\bA_2^{-1}$; (2) is from:  $\|  (\bar \bH\bar\bH'+\frac{1}{N}\bS)^{-1}  \|\leq\lambda_{\min}^{-1}(\frac{1}{N}\bS)=O_P(N)$.

 	Step 4.  Show $\|\bH'(\frac{1}{T}\widehat\bF'\bK\widehat\bF)^{-1} \|
 	=O_P(\nu_{\min}^{-1}+\sqrt{\frac{N}{T}})$.
 	
 	Let $\bA:=\bar \bH\bar\bH'+\frac{1}{N}\bS$.
By steps 2,3   $\|\bar\bH\bA^{-1}\|=O_P(\nu_{\min}^{-1})$ and $ \|\bar\bH\bA^{-1}\bar \bH\|=O_P(1). $ Now
 	\begin{eqnarray*}
 		&& \|\bar\bH'(\frac{1}{T}\widehat\bF'\bK\widehat\bF)^{-1}-\bar\bH'\bA^{-1}\|
 		= \|\bar\bH' \bA^{-1}  (\frac{1}{T}\widehat\bF'\bK\widehat\bF-\bA) (\frac{1}{T}\widehat\bF'\bK\widehat\bF)^{-1}\|\cr
 		&\leq^{(3)}& O_P(\frac{\nu_{\max}(\bH)}{\nu_{\min}(\bH)\sqrt{TN}})\| (\frac{1}{T}\widehat\bF'\bK\widehat\bF)^{-1}\|
 		=^{(4)}O_P(\frac{N}{\sqrt{NT}})=O_P(\sqrt{\frac{N}{T}}).
 	\end{eqnarray*}
 	In (3) we used $\frac{1}{T}\widehat\bF'\bK\widehat\bF-\bA=O_P(\frac{1}{N\sqrt{T}}+\|\frac{1}{T}\bH\bF'\bK\bE\|)=O_P(\frac{1}{N\sqrt{T}}+\frac{\nu_{\max}}{\sqrt{TN}})=O_P(\frac{\nu_{\max}}{\sqrt{TN}})$; in (4) we used $(\frac{1}{T}\widehat\bF'\bK\widehat\bF)^{-1}=O_P(N)$ by part (i) and $\nu_{\max}\leq C\nu_{\min}$.
 	Hence $$ \|\bar\bH'(\frac{1}{T}\widehat\bF'\bK\widehat\bF)^{-1}\|
	\leq O_P(\sqrt{\frac{N}{T}})+\|\bar\bH\bA^{-1}\|=
	O_P(\nu_{\min}^{-1}+\sqrt{\frac{N}{T}}).$$
	Thus $\|\bH'(\frac{1}{T}\widehat\bF'\bK\widehat\bF)^{-1} \|
 	\leq \|(\frac{1}{T}\bF'\bK\bF)^{-1/2}\| \|\bar\bH'(\frac{1}{T}\widehat\bF'\bK\widehat\bF)^{-1}\|,
 	$  which leads to the result for $\|\bH'(\frac{1}{T}\widehat\bF'\bK\widehat\bF)^{-1} \|
 	=O_P(\nu_{\min}^{-1}+\sqrt{\frac{N}{T}})$.
 	
 	Step 5.  show $ \bH'(\frac{1}{T}\widehat\bF'\bK\widehat\bF)^{-1}  \bH=\bH'(\frac{1}{T}\bH\bF'\bK\bF\bH'+\frac{1}{N}\bS)^{-1} \bH + O_P(\frac{1}{\nu_{\min}\sqrt{NT}}+\frac{1}{T}) .$

 	Because $\|\bar\bH\bA^{-1}\|=O_P(\nu_{\min}^{-1})$  and $ \|\bar\bH\bA^{-1}\bar \bH\|=O_P(1) $ by step 3, (\ref{eqa.1adddd}) implies
 	\begin{eqnarray*}
 		&&\|\bar \bH'(\frac{1}{T}\widehat\bF'\bK\widehat\bF)^{-1} \bar \bH- \bar\bH'\bA^{-1}\bar \bH \|
 		=\|\bar \bH'(\frac{1}{T}\widehat\bF'\bK\widehat\bF)^{-1} (\frac{1}{T}\widehat\bF'\bK\widehat\bF-\bA)   \bA^{-1}\bar \bH \|\cr
 		&\leq& \|\bar\bH '\bA^{-1}     \bar  \bH   (\frac{1}{T}\bF'\bK\bF)^{-1/2} \frac{1}{T} \bF'\bK\bE     (\frac{1}{T}\widehat\bF'\bK\widehat\bF)^{-1}\bar\bH  \|
 		+\|\bar\bH '\bA^{-1}     \frac{1}{T}\bE'\bK\bF(\frac{1}{T}\bF'\bK\bF)^{-1/2} \bar\bH'     (\frac{1}{T}\widehat\bF'\bK\widehat\bF)^{-1}\bar\bH \|\cr
		&&
 		+\|\bar\bH '\bA^{-1}     \bDelta_1    (\frac{1}{T}\widehat\bF'\bK\widehat\bF)^{-1}\bar\bH \|\cr
 		&\leq& O_P(\nu_{\min}^{-1} \frac{1}{N\sqrt{T }}+ \frac{1}{\sqrt{NT}}) \|  (\frac{1}{T}\widehat\bF'\bK\widehat\bF)^{-1}\bar\bH \|
		=^{(5)}
 		O_P(\frac{1}{\sqrt{NT}}) O_P(\nu_{\min}^{-1}+\sqrt{\frac{N}{T}}) = O_P(\frac{1}{\nu_{\min}\sqrt{NT}}+\frac{1}{T}).
 	\end{eqnarray*}
 	(5) follows from step 4 and $\nu_{\min}\gg N^{-1/2}$. Then due to $\|(\frac{1}{T}\bF'\bK\bF)^{-1/2}\|=O_P(1)$,
	$$\bH'(\frac{1}{T}\widehat\bF'\bK\widehat\bF)^{-1}  \bH=\bH'(\frac{1}{T}\bH\bF'\bK\bF\bH'+\frac{1}{N}\bS)^{-1} \bH + O_P(\frac{1}{\nu_{\min}\sqrt{NT}}+\frac{1}{T}) .$$  
 	In addition, step 3 implies $\|\bH'(\frac{1}{T}\bH\bF'\bK\bF\bH'+\frac{1}{N}\bS)^{-1} \bH\|\leq O_P(\nu_{\min}^{-1}\nu_{\max})=O_P(1)$, so
 	  $$ \|\bH'(\frac{1}{T}\widehat\bF'\bK\widehat\bF)^{-1} \bH\|=O_P(1+\frac{1}{\nu_{\min}\sqrt{NT}}+\frac{1}{T})=O_P(1).$$

 	(iii) The proof still consists of several steps.

 	Step 1.   $ \bH'(\frac{1}{T}\widehat\bF'\bK\widehat\bF)^{-1}  \bH=\bH'(\frac{1}{T}\bH\bF'\bK\bF\bH'+\frac{1}{N}\bS)^{-1} \bH + O_P(\frac{1}{\nu_{\min}\sqrt{NT}}+\frac{1}{T}).$
 	
 	It follows from  step 5 of part (ii).

 	Step 2. show $\bar\bH'( \bar\bH\bar\bH'+ \frac{1}{N}\bS)^{-1}\bar\bH
 	=\bar\bH'( \bar\bH\bar\bH')^{+}\bar\bH
 	+O_P(\frac{1}{N\nu^2_{\min}})$ where $\bar\bH=\bH(\frac{1}{T}\bF'\bK\bF)^{1/2}$.
 	Write  $\bT=\bar\bH'( \bar\bH\bar\bH'+ \frac{1}{N}\bS)^{-1}\bar\bH
 	-\bar\bH'( \bar\bH\bar\bH')^{+}\bar\bH  $. The goal is to show $\|\bT\|=O_P(\frac{1}{N\nu^2_{\min}})$. Let $\bv$ be the unit vector so that $|\bv'\bT\bv|=\|\bT\|.$
 	Define a function, for $d>0$,
 	$$
 	g(d):= \bv'\bar\bH'( \bar\bH\bar\bH'+ \frac{d}{N}\bI)^{-1}\bar\bH \bv.
 	$$ Note that there are constants $c, C>0$ so that $ \frac{c}{N}<\lambda_{\min}(\frac{1}{N}\bS)\leq \lambda_{\max}(\frac{1}{N}\bS)<  \frac{C}{N}.$
 	Then we have $g(C)< \bv'\bar\bH'( \bar\bH\bar\bH'+ \frac{1}{N}\bS)^{-1}\bar\bH \bv<g(c)$.
 	Hence
 	$$
 	|\bv'\bT\bv|\leq |g(c)-\bv'\bar\bH'( \bar\bH\bar\bH')^{+}\bar\bH \bv| +|g(C)-\bv'\bar\bH'( \bar\bH\bar\bH')^{+}\bar\bH \bv| .
 	$$
 	Recall $\bar \bH'=\bU_{\bar H}(\bD_{\bar H},0)\bE_{\bar H}'$ is the SVD of $\bar \bH'$ and $N^{-1}\lambda^{-1}_{\min}(\bD_{\bar H}^2)=o_P(1)$. Then for any  $d\in\{c, C\}$,  as $N\to\infty$,
 	$
 	g(d)=\bv'\bU_{\bar H}\bD_{\bar H}^2(\bD_{\bar H}^2+\frac{d}{N}\bI)^{-1}\bU_{\bar H}'\bv
 	\overset{P}{\longrightarrow} \bv'\bv=\bv'\bar\bH'( \bar\bH\bar\bH')^{+}\bar\bH \bv,
 	$
 	where we used
 	$\bar\bH'( \bar\bH\bar\bH')^{+}\bar\bH=\bI$, easy to see from its SVD.
 	The rate of convergence is
 	$$
 	\|\bD_{\bar H}^2(\bD_{\bar H}^2+\frac{d}{N}\bI)^{-1} -\bI\|
 	\leq  \|\bD_{\bar H}^2(\bD_{\bar H}^2+\frac{d}{N}\bI)^{-1} \frac{d}{N} \bD_{\bar H}^{-2}\|=O_P(\frac{1}{N\nu^2_{\min}}).
 	$$
 	Hence $|\bv'\bT\bv|=O_P(\frac{1}{N\nu^2_{\min}})$.
 	
 	Step 3. show $\|\bH'(\frac{1}{T}\widehat\bF'\bK\widehat\bF)^{-1}\bH- \bH'(\bH\frac{1}{T} \bF' \bK\bF\bH')^{+}\bH\|=O_P( \frac{1}{N\nu^2_{\min}} + \frac{1}{T})$.
 	By steps 1 and 2,
 	\begin{eqnarray*}
 		\bH'(\frac{1}{T}\widehat\bF'\bK\widehat\bF)^{-1}\bH&=&\bH'( \bar \bH\bar \bH
 		'+ \frac{1}{N}\bS)^{-1}\bH +O_P(\frac{1}{\nu_{\min}\sqrt{NT}}+\frac{1}{T})\cr
 		&=&(\frac{1}{T}\bF'\bK\bF)^{-1/2}\bar\bH'( \bar \bH\bar \bH
 		'+ \frac{1}{N}\bS)^{-1}\bar\bH(\frac{1}{T}\bF'\bK\bF)^{-1/2} + O_P(\frac{1}{\nu_{\min}\sqrt{NT}}+\frac{1}{T})\cr
 		&=^{(6)}&(\frac{1}{T}\bF'\bK\bF)^{-1/2}\bar\bH'( \bar\bH\bar\bH')^{+}\bar\bH (\frac{1}{T}\bF'\bK\bF)^{-1/2} +O_P( \frac{1}{N\nu^2_{\min}} + \frac{1}{\nu_{\min}\sqrt{NT}}+\frac{1}{T})\cr
 		&= & \bH'( \bar\bH\bar\bH')^{+} \bH  +O_P(\frac{1}{N\nu^2_{\min}} + \frac{1}{T}).
 	\end{eqnarray*}
 	where (6) is due to $\lambda_{\min}(\frac{1}{T}\bF'\bK\bF)>c$ and step  2.
 	
 	Step 4. show $\frac{1}{T}\bG'\bP_{\widehat\bF}\bG =\frac{1}{T}\bG'\bP_{\bF\bH'}\bG+ O_P(\frac{1}{N\nu_{\min}^2}+\frac{1}{T})$.
	
	  By  part (ii) $\|\bH'(\frac{1}{T}\widehat\bF'\bK\widehat\bF)^{-1} \|=O_P(\nu_{\min}^{-1}+\sqrt{\frac{N}{T}})$,    and that $\frac{1}{T}\bG'\bE=O_P(\frac{1}{\sqrt{NT}})$, 
 	\begin{eqnarray*}
 		\frac{1}{T}\bG'\bP_{\widehat\bF}\bG
 		&=&\frac{1}{T}\bG'\bF\bH'(\widehat\bF'\widehat\bF)^{-1}\bH\bF'\bG+
 		 \frac{1}{T}\bG'\bE  (\widehat\bF'\widehat\bF)^{-1}\bE '\bG
 		+\frac{1}{T}\bG'\bE  (\widehat\bF'\widehat\bF)^{-1} \bH\bF '\bG
 		\cr
		&&+\frac{1}{T}\bG'\bF \bH'  (\widehat\bF'\widehat\bF)^{-1}\bE '\bG\cr
 		&=&
		\frac{1}{T}\bG'\bF\bH'(\widehat\bF'\widehat\bF)^{-1}\bH\bF'\bG+O_P(\frac{1}{T}+\frac{1}{\nu_{\min}\sqrt{NT}})\cr
	&=&	
		\frac{1}{T}\bG'\bF\bH'(\bH\bF'\bF\bH')^+\bH\bF'\bG+O_P(\frac{1}{N\nu_{\min}^2}+\frac{1}{T}),
 	\end{eqnarray*}
 	where the last equality follows from step 3.
 \end{proof}

The proof of Lemma \ref{la.1} below does \text{not} rely on Proposition \ref{la.2}, as it does not involve $\bH$ or $\widehat\bF$. Also, let  $\bE= (\be_1,\cdots,\be_T)'=\frac{1}{N}\bU'\bW$. In addition, we shall use the following inequality $\tr(\bW'\bSigma\bW)\leq R\|\bW\|^2\|\bSigma\|$ for any semipositive definite matrix $\bSigma$, whose simple proof is as follows: let $\bv_i$ be the $i$ th  eigenvector of $\bW'\bSigma\bW$. Then
$$
\tr(\bW'\bSigma\bW)=\sum_{i=1}^R\bv_i'\bW'\bSigma\bW\bv_i\leq \|\bSigma\|\sum_{i=1}^R\|\bW\bv_i\|^2\leq \|\bSigma\|\|\bW\|^2R.
$$
  \begin{lem}
  \label{la.1} For any    $R\geq 1$, ($R$ can be either smaller, equal to or larger than $r$),\\
  (i) $\|\frac{1}{T}\E \bE'\bE\|\leq \frac{C}{N}$ and $\|\bE\|=O_P(\sqrt{\frac{T}{N}})$. \\
  (ii) $\E\|\frac{1}{T}\bF'\bE\|^2\leq O (\frac{1}{{TN}})$, $\E\|\frac{1}{T}\bG'\bE\|^2\leq O (\frac{1}{{TN}})$,  here $\bG$ is defined as in Section \ref{sec:3.1}
  \\
  (iii) $\|\frac{1}{T}(\bE'\bE-\E\bE'\bE) \|\leq O_P(\frac{1}{N\sqrt{T}})$, $ \|\frac{1}{T}\bE'\bP_\bG\bE\|=O_P(\frac{1}{NT}). $\\
  (iv) $\|\frac{1}{N} \bU'\bW\| \leq O_P(\sqrt{\frac{T}{N}}).$ 
  \end{lem}

  \begin{proof}
  
  (i)  By the assumption  $\|\frac{1}{T}\E\bU\bU'\| = \| \E\bu_t\bu_t'\|\leq  \E \| \E(\bu_t\bu_t'|\bF)\|<C$. Thus
  $$
  \|\frac{1}{T}\E \bE'\bE\|
  =\frac{1}{N^2}\|\bW'\frac{1}{T}\E\bU\bU'\bW\|
 \leq  \frac{1}{N^2}\|\bW\|^2 \leq \frac{C}{N}.
  $$
  Also, $\E\|\bE\|^2\leq \tr\E \bE'\bE \leq R\|\E\bE'\bE\| \leq \frac{CT}{N}$.

  (ii) Let $f_{k,t}$ be the $k$ th entry of $\bff_t$. 
  By the assumption  $ \frac{1}{T } \sum_{s=1}^T \sum_{t=1}^T\E\| \bff_{t}  \| \|\bff_{s} \|    \|\E(\bu_t \bu_s' |\bF)\|<C $, 
  \begin{eqnarray*}
 \E  \|\frac{1}{T}\bF'\bE\|^2&=& 
  \frac{1}{T^2N^2}\E \| \sum_{t=1}^T\bW'\bu_t\bff_t'  \|^2
  \leq   \sum_{k=1}^r\frac{1}{T^2N^2} \sum_{s=1}^T \sum_{t=1}^T\E f_{k,t} f_{k,s}  \E(  \bu_s' \bW  \bW'\bu_t  |  \bF) \cr
  &\leq& \sum_{k=1}^r\frac{1}{T^2N^2} \sum_{s=1}^T \sum_{t=1}^T\E f_{k,t} f_{k,s}    \tr \bW'  \E(\bu_t \bu_s' |\bF)\bW   \cr
    &\leq& \sum_{k=1}^r\frac{1}{T^2N^2} \sum_{s=1}^T \sum_{t=1}^T\E| f_{k,t} f_{k,s} |  \| \bW\|_F ^2 \|\E(\bu_t \bu_s' |\bF)\| \cr
     &\leq&  \frac{C}{T^2N } \sum_{s=1}^T \sum_{t=1}^T\E\| \bff_{t}  \| \|\bff_{s} \|    \|\E(\bu_t \bu_s' |\bF)\| \cr
&\leq& \frac{C}{TN}.
  \end{eqnarray*}
  Similarly, $\E\|\frac{1}{T}\bG'\bE\|^2\leq O (\frac{1}{{TN}})$.

  (iii) By the assumption that $\frac{1}{T N^2} \sum_{t,s\leq T}  \sum_{i,j,m,n\leq N}    |\Cov (  u_{it}u_{jt},    u_{ms}u_{ns} )|<C$,
    \begin{eqnarray*}
&&\E\|\frac{1}{T}(\bE'\bE-\E\bE'\bE) \|^2
\leq \sum_{k,q\leq R}\E (\frac{1}{TN^2}\sum_{t=1}^T \sum_{i,j\leq N} w_{k,i}w_{q,j} (u_{it}u_{jt}-\E u_{it}u_{jt}))^2\cr
&\leq& \frac{C}{T N^2}\frac{1}{T N^2} \sum_{t,s\leq T}  \sum_{i,j,m,n\leq N}    |\Cov (  u_{it}u_{jt},    u_{ms}u_{ns} )|\leq \frac{C}{T N^2}.
  \end{eqnarray*}
   Next, by part (ii)
    \begin{eqnarray*}
  \|\frac{1}{T}\bE'\bP_\bG\bE\|&\leq &  \|\frac{1}{T}\bE' \bG\|^2\|(\frac{1}{T}\bG'\bG)^{-1}\|
  \leq O_P(\frac{1}{TN}).
    \end{eqnarray*}

   (iv) $\E\|\frac{1}{N} \bU'\bW\|^2\leq \frac{1}{N^2}\tr\E\bW'\bU\bU'\bW\leq \frac{CT}{N^2}\|\bW\|_F^2\leq\frac{CT}{N}$, where we used the assumption that $\|\E\bu_t\bu_t'\|<C.$

  \end{proof}

\subsection{Proof of Theorem \ref{t2.1}}
\begin{proof}
	We shall first show the convergence of $\bP_{\widehat\bF\bM}-\bP_\bF$, and then the convergence of $\bP_{\widehat\bF}\bP_\bF-\bP_\bF$.
	
	First,  from the SVD $\bH'=\bU_{H}(\bD_{ H},0)\bE_{ H}'$, it is straightforward to verify that $\bM'=\bU_{H}(\bD_{ H}^{-1},0)\bE_{ H}'$. Then from Proposition \ref{la.2}, $\lambda_{\min}(\frac{1}{T}\bM'\widehat\bF'\widehat\bF\bM)\geq c_0 N^{-1}\lambda_{\min}(\bD_H^{-2})$ with large probability. Hence
$\bP_{\widehat\bF\bM}$ is well defined. 

Next, it is easy to see $\bH'  ( \bH\bH')^{+}\bH=\bI$ when $R\geq r.$ Then $\widehat\bF= \bF\bH'+\bE$ implies $\widehat\bF\bM-\bF =\bE( \bH\bH')^{+}\bH$ with $\bM=( \bH\bH')^{+}\bH.$
	Since $\|( \bH\bH')^{+}\bH\|=O_P(\nu_{\min}^{-1})$, we have 
	$$ \frac{1}{\sqrt{T}}\|\widehat\bF\bM-\bF\|= O_P(\frac{1}{\sqrt{N}}\nu_{\min}^{-1}),\quad  \frac{1}{T}\|\bF'(\widehat\bF\bM-\bF)\|= O_P(\frac{1}{\sqrt{NT}}\nu_{\min}^{-1})
	$$
	where the second statement uses Lemma \ref{la.1}. Then
  $\|\frac{1}{T}\bM'\widehat\bF'\widehat\bF\bM-\frac{1}{T}\bF'\bF\|=O_P(\frac{1}{\sqrt{NT}}\nu_{\min}^{-1}+\frac{1}{N}\nu_{\min}^{-2})$.
	Thus $(\frac{1}{T}\bM'\widehat\bF'\widehat\bF\bM)^{-1}=O_P(1)$ and 
	\begin{equation}\label{eqa.2new}
	\|(\frac{1}{T}\bM'\widehat\bF'\widehat\bF\bM)^{-1}-(\frac{1}{T}\bF'\bF)^{-1}\|=O_P(\frac{1}{\sqrt{NT}}\nu_{\min}^{-1}+\frac{1}{N}\nu_{\min}^{-2}).
	\end{equation}
	 The triangular inequality then implies
	$
\|	 \bP_{\widehat\bF\bM}-\bP_\bF\|
\leq O_P(\frac{1}{\sqrt{N}}\nu_{\min}^{-1}).
	$
	
	Finally,  $\bP_{\widehat \bF} \bP_{\widehat\bF\bM}=\bP_{\widehat\bF\bM}$ gives
$$
\|\bP_{\widehat\bF}\bP_\bF-\bP_\bF\|
\leq \|\bP_{\widehat\bF}(\bP_\bF-\bP_{\widehat\bF\bM})\|
+ \|\bP_{\widehat\bF\bM}-\bP_\bF\|\leq O_P(\frac{1}{\sqrt{N}}\nu_{\min}^{-1}).
$$

\end{proof}

 \subsection{Proof  of Theorem \ref{t3.1} }

\begin{proof}
	Here we assume $R\geq r$.
 	We let $\bz_t=(\bff_t'\bH',\bg_t')'$ and $\bdelta=(\balpha'\bH^+,\bbeta')'.$ Then $\bdelta'\bz_t=y_{t+h|t}$.
	First, we have the following expansion
$$
\widehat\bdelta'\widehat\bz_T-\bdelta'\bz_T= (\widehat\bdelta-\bdelta)' \widehat\bz_T +\balpha'\bH^+(\widehat\bff_T-\bH\bff_T).
$$	
Now $\widehat\bdelta=(\widehat\bZ'\widehat\bZ)^{-1}\widehat\bZ'\bY$, where $\bY$ is the $(T-h)\times 1$ vector of $y_{t+h}$, and $\widehat\bZ$ is the $(T-h)\times\dim(\bdelta)$ matrix of $\widehat\bz_t$, $t=1,\cdots,T-h$. Also recall that $\be_t=\widehat\bff_t-\bH\bff_t=\frac{1}{N}\bW'\bu_t$.  Then
\begin{eqnarray*}
	\widehat\bz_T'(\widehat\bdelta-\bdelta)&=&\widehat\bz_T'  (\frac{1}{T}\widehat\bZ'\widehat\bZ)^{-1}\sum_{d=1}^4a_d,\text{ where }\cr
	a_1&=&(\frac{1}{T}\sum_t\varepsilon_t\be_t', 0)' ,\quad
		a_2= \frac{1}{T}\sum_t\bz_t\varepsilon_t  \cr
	a_3&=&(-\balpha'\bH^+\frac{1}{T}\sum_t \be_t\be_t',0)'   ,\quad 	a_4=-\frac{1}{T}\sum_t \bz_t\be_t'\bH^{+'}\balpha .
\end{eqnarray*}
On the other hand, let $\bG$ be the $(T-h)\times \dim(\bg_t)$ matrix of $\{\bg_t: g\leq T-h\}$. We have, by the matrix block inverse formula, for the operator $\bM_\bA:= \bI-\bP_{\bA},$
$$
( \frac{1}{T}\widehat\bZ'\widehat\bZ)^{-1}=
\begin{pmatrix}
\bA_1& \bA_2\\
\bA_2'& \bA_3
\end{pmatrix},\quad \text{ where }
\begin{pmatrix}
\bA_1\\
\bA_2\\
\bA_3
\end{pmatrix}=\begin{pmatrix}
(\frac{1}{T}\widehat\bF'\bM_\bG\widehat\bF)^{-1}\\
-  	\bA_1 \widehat\bF'\bG(\bG'\bG)^{-1} \\
(\frac{1}{T} \bG'\bM_{\widehat\bF}\bG)^{-1}
\end{pmatrix}.
$$
Then $\widehat\bz_T'  (\frac{1}{T}\widehat\bZ'\widehat\bZ)^{-1}= (\widehat\bff_T'\bA_1+\bg_T'\bA_2',\widehat\bff_T'\bA_2+\bg_T'\bA_3)$. This implies
\begin{eqnarray*}
	\widehat\bz_T'(\widehat\bdelta-\bdelta)&=&  (\widehat\bff_T'\bA_1+\bg_T'\bA_2')\frac{1}{T}\sum_t[\be_t\varepsilon_t -\be_t\be_t'\bH^{+'}\balpha  ]\cr
	&&+(\widehat\bff_T'\bA_1\bH+\bg_T'\bA_2'\bH)\frac{1}{T}\sum_t[\bff_t\varepsilon_t -\bff_t\be_t'\bH^{+'}\balpha]
  \cr
	&&+ (\widehat\bff_T'\bA_2+\bg_T'\bA_3)\frac{1}{T}\sum_t[\bg_t\varepsilon_t -\bg_t\be_t'\bH^{+'}\balpha ] .
\end{eqnarray*}

It is easy to show  $\|\frac{1}{T}\sum_t\bff_t\varepsilon_t\|+ \|\frac{1}{T}\sum_t\bg_t\varepsilon_t\|=O_P(\frac{1}{\sqrt{T}})   $ and
$\|\frac{1}{T}\sum_t\be_t\varepsilon_t\|=O_P(\frac{1}{\sqrt{TN}})$.   Also Lemma \ref{la.1} gives $\frac{1}{T}\sum_t\be_t\be_t'=\frac{1}{T}\bE'\bE=O_P(\frac{1}{N})$, $\frac{1}{T}\sum_t\bff_t\be_t=\frac{1}{T}\bF'\bE=O_P(\frac{1}{\sqrt{TN}})$, and $\frac{1}{T}\sum_t\bg_t\be_t=\frac{1}{T}\bF'\bE=O_P(\frac{1}{\sqrt{TN}})$. Together with Lemma \ref{lb.1},
\begin{eqnarray*}
	\widehat\bz_T'(\widehat\bdelta-\bdelta)
&=&\|\widehat\bff_T'\bA_1 +\bg_T'\bA_2\|O_P(\frac{1}{\sqrt{TN}}+\frac{1}{N\nu_{\min}})
 \cr
 &&
 +\|\widehat\bff_T'\bA_1\bH+\bg_T'\bA_2'\bH\|  O_P(\frac{1}{\sqrt{T}})
 + \|\widehat\bff_T'\bA_2+\bg_T'\bA_3\|  O_P(\frac{1}{\sqrt{T}})
 \cr
&=&O_P(\frac{1}{\sqrt{T}}+\frac{1}{\sqrt{N}\nu_{\min}}).
\end{eqnarray*}

Finally, as $\|\bH^+\|=O_P(\nu_{\min}^{-1})$,
$\balpha'\bH^+(\widehat\bff_T-\bH\bff_T)=O_P(\nu_{\min}^{-1})\|\be_T\|=O_P(\nu_{\min}^{-1}N^{-1/2}).$


 \end{proof}

\begin{lem}\label{lb.1} For all $R\geq r$, (i)
$\|\bA_1\widehat\bff_T\|+\|\bA_2\|=O_P(\sqrt{N}),  $ and \\ $\|\bH'\bA_1\widehat\bff_T\|+\|\bH'\bA_2\|+\|
\bA_2'\widehat\bff_T\|+\|\bA_3\|=O_P(1) $.
\end{lem}
 \begin{proof} 
  First, by Proposition \ref{la.2},   $\|\bA_1\|=O_P(N)$ and $\|\bA_1\bH\|=O_P(\nu_{\min}^{-1}+\sqrt{\frac{N}{T}})$, and $\frac{1}{T}\bE'\bG=O_P(\frac{1}{\sqrt{NT}})$
 \begin{eqnarray*}
 \bA_1\widehat\bff_T&=&(\frac{1}{T}\widehat\bF'\bM_\bG\widehat\bF)^{-1}\be_T
 +(\frac{1}{T}\widehat\bF'\bM_\bG\widehat\bF)^{-1}\bH\bff_T
 =O_P(\sqrt{N})\cr
\bH'\bA_1\widehat\bff_T &=&\bH'(\frac{1}{T}\widehat\bF'\bM_\bG\widehat\bF)^{-1}\be_T
+\bH'(\frac{1}{T}\widehat\bF'\bM_\bG\widehat\bF)^{-1}\bH\bff_T=O_P(1) \cr
-\bA_2&=&\bA_1 \widehat\bF'\bG(\bG'\bG)^{-1}
=\bA_1\bE'\bG(\bG'\bG)^{-1}
+\bA_1 \bH \bF'\bG(\bG'\bG)^{-1} =O_P(\sqrt{\frac{N}{T}}+\nu_{\min}^{-1})\cr
-\bH'\bA_2&=&
\bH'\bA_1\bE'\bG(\bG'\bG)^{-1}
+\bH'\bA_1 \bH \bF'\bG(\bG'\bG)^{-1} = O_P(1) \cr
\bA_2'\widehat\bff_T&=&\bA_2'\bH\bff_T+\bA_2'\be_T=O_P(1).
 \end{eqnarray*}
Finally,   it follows from Proposition \ref{la.2} that
$
\frac{1}{T}\bG'(\bP_{\widehat\bF}-\bP_{\bF\bH'})\bG= O_P( \frac{1}{T}+ \frac{1}{N\nu^2_{\min}} )
$. Hence   $\|\bA_3\|=O_P(1)$ since $\lambda_{\min}(\frac{1}{T}\bG'\bM_{\bF\bH'}\bG)>c$.


\end{proof}

 \subsection{Proof  of Theorem \ref{t3.2} }

Let
$\widehat\bvarepsilon_g$,  $\widehat\bvarepsilon_y$,$ \bvarepsilon_g$,  $ \bvarepsilon_y$, $\bY$, $\bG$  and $\bfeta$ be $T\times 1$ vectors of $\widehat\bvarepsilon_{g,t}$,  $\widehat\bvarepsilon_{y,t}$, $ \bvarepsilon_{g,t}$,  $ \bvarepsilon_{y,t}$, $y_t$, $\bg_t$ and  $\eta_t$.
 Let $\widehat J$ denote the index set of components in $\widehat \bu_t$ that are selected by \textit{either} $\widehat\bgamma$ \textit{or} $\widehat\btheta$. Let $\widehat \bU_{\widehat J}$ denote the $N\times |J|_0$ matrix of  rows of $\widehat\bU$ selected by $J$. Then
 $$
 \widehat \bvarepsilon_y=\bM_{\widehat\bU_{\widehat J}} \bM_{\widehat\bF} \bY,\quad
  \widehat \bvarepsilon_g=\bM_{\widehat\bU_{\widehat J}} \bM_{\widehat\bF} \bG.
 $$
 \subsubsection{The case $r\geq 1$.}\label{sec:thereisfactor}

 \begin{proof}
From Lemma \ref{lc.4}  
\begin{eqnarray}
\sqrt{T}(\widehat\bbeta- \bbeta)&=&
\sqrt{T}[(\widehat\bvarepsilon_g'\widehat\bvarepsilon_g)^{-1}
 \widehat\bvarepsilon_g'( \widehat\bvarepsilon_y- \bvarepsilon_y)
+(\widehat\bvarepsilon_g'\widehat\bvarepsilon_g)^{-1} \widehat\bvarepsilon_g'\bfeta
+(\widehat\bvarepsilon_g'\widehat\bvarepsilon_g)^{-1}
 \widehat\bvarepsilon_g'(  \bvarepsilon_g- \widehat\bvarepsilon_g)\bbeta]
\cr
&=&O_P(1)\frac{1}{\sqrt{T}} \widehat\bvarepsilon_g'( \widehat\bvarepsilon_y- \bvarepsilon_y)
+ O_P(1)\frac{1}{\sqrt{T}} \widehat\bvarepsilon_g'(  \bvarepsilon_g- \widehat\bvarepsilon_g)
+O_P(1)\frac{1}{\sqrt{T}} \bfeta'(  \widehat\bvarepsilon_g-\bvarepsilon_g)
\cr
&&+(\frac{1}{T} \bvarepsilon_g' \bvarepsilon_g)^{-1}  \frac{1}{\sqrt{T}}\bvarepsilon_g'\bfeta\cr
&=&\sigma_g^{-2}  \frac{1}{\sqrt{T}}\bvarepsilon_g'\bfeta+o_P(1)\overset{d}{\longrightarrow} \mathcal N(0,\sigma_g^{-4} \sigma_{\eta g}^2 ).
 \end{eqnarray}
 In the above, we used the condition that  
 $ 
 |J|_0^4+|J|^2_0\log^2 N =o(T)
 $ 
,  $ 
T|J|_0^4=o(N^2\min\{1,  \nu^{4}_{\min}|J|_0^4 \} )
 $ 
 and  $ 
 \sqrt{\log N}  |J|_0^2=o(N\nu^2_{\min}) $, whose sufficient conditions are  
$ 
T|J|_0^4=o(N^2\min\{1,  \nu^{4}_{\min}|J|_0^4 \} )
 $ 
 and $ 
 |J|_0^4 \log^2 N =o(T)
 $.

 In addition,
 $
\widehat\sigma_{\eta, g}^{-1}\widehat\sigma_g^{2}  {\sqrt{T}(\widehat\bbeta-\bbeta)}\overset{d}{\longrightarrow}\mathcal N(0,1),
 $ follows from  $\widehat\sigma_g^2:=\frac{1}{T}\widehat\bvarepsilon_g'\widehat\bvarepsilon_g\overset{P}{\longrightarrow} \sigma_g^2$.

 \end{proof}

  \begin{prop}\label{proc.1} Suppose 
 $ T=O(  \nu_{\min}^4N^2\log N )$, 
   $|J|_0^2T=O( \nu_{\min}^2N^2\log N ) $, 
$  |J|_0^2=O(N \nu_{\min}^2\log N  )$ and $|J|_0^2\log N=O(T)$, $|J|_0^2=o(N)$ For all $R\geq r$, \\
 (i) $\frac{1}{T}\|  \widehat \bU'\btheta- \widehat  \bU'\widetilde\btheta\|^2=O_P(|J|_0\frac{\log N}{T})$ and $\|\widetilde\btheta-\btheta\|_1=O_P(|J|_0\sqrt{\frac{\log N}{T}})$. \\
 (ii) $|\widehat J|_0=O_P(|J|_0)$.
 \end{prop}

 \begin{proof}
 (i) Let $
 L(\btheta):=  \frac{1}{T}\sum_{t=1}^T(\bg_t-\widehat\balpha_g'\widehat\bff_t-\btheta'\widehat\bu_t)^2+  {\tau}\|\btheta\|_1,
 $
 $$
d_t=  \balpha_g' \bff_t -\widehat\balpha_g'\widehat\bff_t +(\bu_t-\widehat\bu_t)'\btheta,\quad \bDelta=\btheta-\widetilde\btheta.
 $$
 Then $  \bg_t= \balpha_g'\bff_t +\btheta'\bu_t+\varepsilon_{g,t},$ and $ L(\widetilde \btheta)\leq  L(\btheta)$ imply
 $$
   \frac{1}{T}\sum_{t=1}^T[(\widehat\bu_t'\bDelta)^2+2(\varepsilon_{g,t}+d_t)\widehat\bu_t'\bDelta ] +  {\tau}\|\widetilde\btheta\|_1
  \leq   {\tau}\|\btheta\|_1.
 $$  It follows from Lemma \ref{lc.2} that $ \|\frac{1}{T}\widehat\bU \bvarepsilon_{g}\|_\infty\leq O_P( \sqrt{\frac{\log N}{T}}) $. Also Lemma \ref{lc.0} implies  that
 \begin{eqnarray*}
 \| \frac{1}{T}\sum_{t=1}^Td_t\widehat\bu_t\|_\infty
 &\leq& \|\frac{1}{T}\widehat\bU\bE\bH^{+'}\balpha\|_\infty+ \|\frac{1}{T}\widehat\bU\bE (\bH^{+'}\balpha_g-\widehat\balpha_g) \|_\infty
 +\|\frac{1}{T}\widehat\bU \bF\bH'(\bH^{+'}\balpha_g-\widehat\balpha_g)\|_\infty\cr
 &&+\|\frac{1}{T}\btheta'(\widehat\bU-\bU) \widehat \bU'\|_\infty\cr
 &\leq&
  O_P(|J|_0\sqrt{\frac{\log N}{TN}} +|J|_0\frac{\log N}{T} +\frac{1}{N\nu^2_{\min}} +\nu^{-1}_{\min}\sqrt{\frac{\log N}{TN}}   
  + \frac{ |J|_0}{N\nu_{\min}}+    \frac{|J|_0}{\nu_{\min}\sqrt{NT}} ).    \end{eqnarray*}   
  Thus the ``score" satisfies $\|\frac{1}{T}\sum_{t=1}^T 2(\bvarepsilon_{g,t}+d_t)\widehat\bu_t'\|_\infty\leq \tau/2 $ for sufficiently large $C>0$ in $\tau=C\sigma\sqrt{\frac{\log N}{T}}$ with probability arbitrarily close to one, given $ T=O(  \nu_{\min}^4N^2\log N )$, 
   $|J|_0^2T=O( \nu_{\min}^2N^2\log N ) $, 
  $ |J|_0^2=O(N \nu_{\min}^2\log N  )$ and $|J|_0^2\log N=O(T)$.
Then by the standard argument in the lasso literature,
$$
\frac{1}{T}\sum_{t=1}^T(\widehat\bu_t'\bDelta)^2+\frac{\tau}{2}\|\bDelta_{J^c}\|_1\leq\frac{3\tau}{2}\|\bDelta_J\|_1.
$$
Meanwhile, by the restricted eigenvalue condition and Lemma \ref{lc.0},
$$\frac{1}{T}\sum_{t=1}^T(\widehat\bu_t'\bDelta)^2\geq \frac{1}{T}\sum_{t=1}^T( \bu_t'\bDelta)^2-\|\bDelta\|_1^2\|\frac{1}{T}\widehat\bU\widehat\bU'-\bU\bU'\|_\infty
\geq \|\bDelta\|_2^2 (\phi_{\min}-o_P(1))
$$
where the last inequality follows from $|J|_0 O_P(\nu_{\min}^{-2}\frac{1}{N}+\frac{\log N}{T} )=o_P(1)$ (Lemma \ref{la.1new}).
From here, the desired convergence results follow  from the standard argument in the lasso literature, we omit details for brevity, and refer to, e.g., \cite{hansen2018fac}.

(ii) The proof of $|\widehat J|_0=O_P(|J|_0)$ also follows from the  standard argument in the lasso literature, we omit details but refer to the proof of Proposition D.1 of \cite{hansen2018fac} and \cite{belloni2014inference}.

 \end{proof}

\begin{lem}\label{la.1new}
	(i) $ \|\frac{1}{T}      \bE'        \bU'\|_\infty=O_P(\sqrt{\frac{\log N}{TN}}+\frac{1}{N}) $
	\\
	(ii)   $\|\frac{1}{T}            \bE' \bP_{\widehat\bF}               \bE       \|= O_P(\frac{1}{N}   )$, $\|\frac{1}{T}        \bE'   \bP_{\widehat\bF}      \bU'\|_\infty  =O_P(\sqrt{\frac{\log N}{TN}}+\frac{1}{N}) $ ,  \\
	(iii)
	$ \|\frac{1}{T}(\widehat\bU-\bU)(\widehat\bU-\bU)'\|_\infty
	+ 2 \|\frac{1}{T}(\widehat\bU-\bU)\bU'\|_\infty= O_P(\nu_{\min}^{-2}\frac{1}{N}+\frac{\log N}{T} ).$\\
	(iv) $ \|\frac{1}{T}\widehat\bU\widehat\bU'- \frac{1}{T} \bU \bU'\|_\infty=  O_P(\nu_{\min}^{-2}\frac{1}{N}+\frac{\log N}{T} )$.
\end{lem}

\begin{proof} Let $\widehat\bF=(\widehat\bff_1,\cdots,\widehat\bff_T)'$. In addition,  $\widehat\bB-\bB\bH^+=  -\bB\bH^+ \bE' \widehat\bF(\widehat\bF'\widehat\bF)^{-1}
	+ \bU\bE(\widehat\bF'\widehat\bF)^{-1}
	+ \bU \bF\bH'(\widehat\bF'\widehat\bF)^{-1}.$
	Therefore,
	\begin{eqnarray}\label{eqa.2add}
	\bU-\widehat\bU&=&\widehat\bB\widehat\bF'-\bB\bF'
	=(\widehat\bB-\bB\bH^+)\widehat\bF'
	+\bB\bH^+\bE'\cr
	&=&-\bB\bH^+ \bE' \widehat\bF(\widehat\bF'\widehat\bF)^{-1}\widehat\bF'
	+ \bU\bE(\widehat\bF'\widehat\bF)^{-1}\widehat\bF'
	+ \bU \bF\bH'(\widehat\bF'\widehat\bF)^{-1}\widehat\bF'+\bB\bH^+\bE'.
	\end{eqnarray} 
	(i) We have \begin{eqnarray*}
		\| \frac{1}{T} \bU\bE\|_\infty& \leq&  \sum_{k\leq r}\max_{i\leq N}|\frac{1}{TN}\sum_t \sum_j(u_{it}u_{jt}-\E u_{it}u_{jt} )w_{k,j}|
		+O(\frac{1}{N}) =O_P(\sqrt{\frac{\log N}{TN}}+\frac{1}{N})\cr
	\end{eqnarray*}
	
	(ii) By Proposition  \ref{la.2} ,  Lemma \ref{la.1} ,  $\nu_{\min}\gg N^{-1/2}$, and $\|\frac{1}{T}\bF'\bU'       \|_\infty=O_P(\sqrt{\frac{\log N}{T}})$
	\begin{eqnarray*}
		\|\frac{1}{T}            \bE' \bP_{\widehat\bF}               \bE       \|&\leq&
		\|\frac{1}{T}            \bE' \bE(\widehat\bF'\widehat\bF)^{-1}\bE  '            \bE       \|
		+  \|\frac{2}{T}            \bE' \bE(\widehat\bF'\widehat\bF)^{-1}\bH \bF  '            \bE       \|
		+ \|\frac{1}{T}            \bE'  \bF\bH'(\widehat\bF'\widehat\bF)^{-1}\bH\bF  '            \bE       \| \cr
		&\leq& O_P(\frac{1}{N}   ) \cr
		\|\frac{1}{T}        \bE'   \bP_{\widehat\bF}      \bU'\|_\infty &\leq &
		\|\frac{1}{T}        \bE'  \bE( \widehat\bF' \widehat\bF)^{-1} \bE'     \bU'\|_\infty
		+  \|\frac{1}{T}        \bE'  \bE( \widehat\bF' \widehat\bF)^{-1} \bH\bF'      \bU'\|_\infty
		\cr
		&&+ \|\frac{1}{T}        \bE'  \bF\bH'( \widehat\bF' \widehat\bF)^{-1} \bE'     \bU'\|_\infty
		+  \|\frac{1}{T}        \bE'  \bF\bH'( \widehat\bF' \widehat\bF)^{-1} \bH\bF'      \bU'\|_\infty\cr
		&\leq&O_P(\sqrt{\frac{\log N}{TN}}+\frac{1}{N}    ).
	\end{eqnarray*}

	(iii) We have $\|\bH^+\|=O(\nu_{\min}^{-1})$. Also, $\|\widehat\bF(\widehat\bF'\widehat\bF)^{-1}\widehat\bF'\|\leq 1$. In addition, by Lemma
	\ref{la.2},
	$\|(\widehat\bF'\widehat\bF)^{-1}\widehat\bF'\|^2
	=\|  (\widehat\bF'\widehat\bF)^{-1}\|\leq  O_P(\frac{N}{T})
	$ and that
	$\|\bH'(\widehat\bF'\widehat\bF)^{-1}\widehat\bF'\|^2=\|\bH'(\widehat\bF'\widehat\bF)^{-1}\bH\| =O_P(\frac{1}{T}).
	$
	Next, by Lemma \ref{la.1},  $\|\bE\|=O_P(\sqrt{\frac{T}{N}})$, and $\max_i\|\bb_i\|<C.$ 
	Substitute the expansion (\ref{eqa.2add}),  and by Proposition  \ref{la.2},  \begin{eqnarray*}
		&&
		\|\frac{1}{T}(\widehat\bU-\bU)(\widehat\bU-\bU)'\|_\infty
		+ 2 \|\frac{1}{T}(\widehat\bU-\bU)\bU'\|_\infty \cr
		&\leq&   \|\frac{2}{T}      \bB\bH^+\bE'        \bU'\|_\infty+
		\|\frac{1}{T}         \bB\bH^+\bE'              \bE\bH^{+'}\bB'   \|_\infty  +      \|\frac{3}{T}       \bU\bE  (\widehat\bF'\widehat\bF)^{-1}   \bE' \bU'      \|_\infty   \cr
		&&+\|\frac{4}{T}           \bB\bH^+ \bE' \bE(\widehat\bF'\widehat\bF)^{-1}   \bE' \bU'         \|_\infty     +\|\frac{4}{T}           \bB\bH^+ \bE' \bE (\widehat\bF'\widehat\bF)^{-1} \bH \bF'\bU'         \|_\infty  \cr
		&&+ \|(\frac{6}{T}       \bU\bE +\frac{3}{T}         \bU \bF\bH') (\widehat\bF'\widehat\bF)^{-1} \bH \bF'\bU'       \|_\infty
		+\|\frac{4}{T}           \bB\bH^+ \bE'  \bF \bH'(\widehat\bF'\widehat\bF)^{-1} (\bH \bF'\bU'   + \bE' \bU'  )     \|_\infty  \cr
		&& + \|\frac{2}{T}       \bB\bH^+ \bE'   \bP_{\widehat\bF}      \bU'\|_\infty    +\|\frac{3}{T}           \bB\bH^+ \bE' \bP_{\widehat\bF}               \bE \bH^{+'} \bB'       \|_\infty  \cr
		&\leq&   \|\frac{C}{T}      \bE'        \bU'\|_\infty O_P(\nu_{\min}^{-1})+
		\|\frac{C}{T}         \bE'              \bE   \| O_P(\nu_{\min}^{-2})  +   N   \|\frac{C }{T}       \bU\bE        \|_\infty   ^2
		+N\|\frac{C}{T}         \bE' \bE \|  \|\frac{1}{T}   \bE' \bU'         \|_\infty O_P(\nu_{\min}^{-1}) \cr
		&&   + O_P(\nu_{\min}^{-1})\|\frac{C}{T}            \bE' \bE \| \| (\widehat\bF'\widehat\bF)^{-1} \bH\| \| \bF'\bU'         \|_\infty 
		+ \|\frac{6}{T}       \bU\bE  \|_{\infty}\| (\widehat\bF'\widehat\bF)^{-1} \bH \|\|\bF'\bU'       \|_\infty  \cr
		&& + \| \frac{3}{T}         \bU \bF\|_\infty\|\bH'  (\widehat\bF'\widehat\bF)^{-1} \bH \| \|\bF'\bU'       \|_\infty
		+O_P(\nu_{\min}^{-1})\|\frac{4}{T}           \bE'  \bF\| \| \bH'(\widehat\bF'\widehat\bF)^{-1}  \bH\| \|\bF'\bU'       \|_\infty
		\cr
		&&      +O_P(\nu_{\min}^{-1})\|\frac{4}{T}          \bE'  \bF \| \|\bH'(\widehat\bF'\widehat\bF)^{-1} \| \|  \bE' \bU'       \|_\infty 
		+  O_P(\nu_{\min}^{-1}) \|\frac{C}{T}      \bE'   \bP_{\widehat\bF}      \bU'\|_\infty    + O_P(\nu_{\min}^{-2})\|\frac{C}{T}            \bE' \bP_{\widehat\bF}               \bE       \|         \cr
		&=& O_P(\nu_{\min}^{-2}\frac{1}{N}+\frac{\log N}{T} ).
	\end{eqnarray*}
	

	Also,
	$
	\|\frac{1}{T}\widehat\bU\widehat\bU'- \frac{1}{T} \bU \bU'\|_\infty \leq
	\|\frac{1}{T}(\widehat\bU-\bU)(\widehat\bU-\bU)'\|_\infty
	+ 2 \|\frac{1}{T}(\widehat\bU-\bU)\bU'\|_\infty \leq    O_P(\nu_{\min}^{-2}\frac{1}{N}+\frac{\log N}{T} ).
	$
	
\end{proof}

 \begin{lem}\label{lc.0}
 For all $R\geq r$, \\
(i)  $ \|\frac{1}{T}\btheta'(\widehat\bU-\bU) \widehat \bU'\|_\infty\leq O_P(\frac{\log N}{T}+\frac{1}{N\nu_{\min}^2})|J|_0.
$ \\
(ii) $\|\frac{1}{T}\bE'   \bP_{\widehat\bF}  \bF\|= O_P(\frac{1}{N\nu_{\min}}   +\frac{1}{\sqrt{NT}})
 $,  $   \|\frac{1}{T}    \bU\bP_{\widehat \bF} \bF\|_\infty= O_P(\sqrt{\frac{\log N}{T}}+\frac{1}{N\nu_{\min}}) $. \\
(iii)
    $\|\frac{1}{T}      \bE'       \widehat \bU'\|_\infty\leq O_P(\sqrt{\frac{\log N}{TN}}+\frac{1}{N\nu_{\min}})$,   $\|\frac{1}{T}      \bF'       \widehat \bU'\|_\infty\leq O_P(\sqrt{\frac{\log N}{T}}+\frac{1}{N\nu_{\min}^2}  )$,  \\
    (iv) $\|\frac{1}{T}\btheta'\bU\bE\|=  |J|_0O_P(\frac{1}{N}+\frac{1}{\sqrt{NT}})$,
$\|\frac{1}{T}\btheta'\bU\bF\|=O_P(\sqrt{\frac{|J|_0}{T}})$, \\
(v) $\widehat\balpha_g-\bH^{+'}\balpha_g=|J|_0O_P(1+\sqrt{\frac{N}{T}})+O_P(\nu_{\min}^{-1})$, $\bH'(\widehat\balpha_g-\bH^{+'}\balpha_g)=O_P( \nu_{\min}^{-1}\frac{|J|_0}{N} +\sqrt{\frac{|J|_0}{T}}+ \nu_{\min} ^{-2}\frac{1}{N})$.

 \end{lem}

 \begin{proof}

        (i) By  Lemma \ref{la.1new}
        $
        \|\frac{1}{T}\btheta'(\widehat\bU-\bU) \widehat \bU'\|_\infty\leq \|\btheta\|_1 \|\frac{1}{T} (\widehat\bU-\bU) \widehat \bU'\|_\infty\leq O_P(\frac{\log N}{T}+\frac{1}{N\nu_{\min}^2})|J|_0.
        $
\\
        (ii) Note $\bH'\bH^{+'}=\bI$,  Lemma \ref{la.1new} shows  $\|\frac{1}{T}            \bE' \bP_{\widehat\bF}               \bE       \|= O_P(\frac{1}{N}   )$, $\|\frac{1}{T}        \bE'   \bP_{\widehat\bF}      \bU'\|_\infty  =O_P(\sqrt{\frac{\log N}{TN}}+\frac{1}{N}) $ ,
        \begin{eqnarray*}
        \|\frac{1}{T}\bE'   \bP_{\widehat\bF}  \bF\|&\leq &
          \|\frac{1}{T}\bE'   \bP_{\widehat\bF}  \bE\bH^{+'}\|
          +     \|\frac{1}{T}\bE'    \bE\bH^{+'}\|+ \|\frac{1}{T}\bE'    \bF \|= O_P(\frac{1}{N\nu_{\min}}   +\frac{1}{\sqrt{NT}}) \cr
           \|\frac{1}{T}    \bU\bP_{\widehat \bF} \bF\|_\infty&\leq&
        \|\frac{1}{T}    \bU\bP_{\widehat \bF} \bE\bH^{+'}\|_\infty
        +        \|\frac{1}{T}    \bU  \bE\bH^{+'}\|_\infty
               +    \|\frac{1}{T}    \bU  \bF \|_\infty\cr
               &\leq& O_P(\sqrt{\frac{\log N}{T}}+\frac{1}{N\nu_{\min}}) .
 \end{eqnarray*}

(iii) By Lemma \ref{la.1new} $ \|\frac{1}{T}      \bE'        \bU'\|_\infty=O_P(\sqrt{\frac{\log N}{TN}}+\frac{1}{N}) $ 
 and (ii)
\begin{eqnarray*}
        \|\frac{1}{T}     \widehat \bU\bE\|_\infty
        &\leq&   \|\frac{1}{T}     \bU\bE\|_\infty+ \|\frac{1}{T}    ( \widehat \bU-\bU)\bE\|_\infty\cr
        &\leq& \|\frac{1}{T}     \bU\bE\|_\infty+
          \|\frac{1}{T}      \bB\bH^+ \bE'   \bP_{\widehat\bF}     \bE\|_\infty +      \|\frac{1}{T}      \bU\bP_{\widehat\bF}   \bE\|_\infty  +          \|\frac{1}{T}    \bB\bH^+\bE'     \bE\|_\infty\cr
                   &\leq& O_P(\sqrt{\frac{\log N}{TN}}+\frac{1}{N\nu_{\min}})\cr
                     \|\frac{1}{T}     \widehat \bU\bF\|_\infty&\leq&
                        \|\frac{1}{T}       \bU\bF\|_\infty
                        +   \|\frac{1}{T}   (  \widehat \bU-\bU)\bF\|_\infty\cr
                        &\leq &    \|\frac{1}{T}       \bU\bF\|_\infty
                        + \|\frac{1}{T}      \bB\bH^+ \bE'   \bP_{\widehat\bF}  \bF\|_\infty +            \|\frac{1}{T}    \bU\bP_{\widehat \bF} \bF\|_\infty
                  + \|\frac{1}{T}    \bB\bH^+\bE'  \bF\|_\infty\cr
                  &\leq& O_P(\sqrt{\frac{\log N}{T}}+\frac{1}{N\nu^2_{\min}}  ).
\end{eqnarray*}

        (iv)
        $\frac{1}{T}\btheta'\bU\bE= \frac{1}{NT}\btheta' (\bU\bU'-\E\bU\bU')\bW
 + \frac{1}{NT}\btheta' \E\bU\bU'\bW
 $. So
 \begin{eqnarray*}
&&\E\| \frac{1}{NT}\btheta' (\bU\bU'-\E\bU\bU')\bW \|^2
=\sum_{k=1}^R\frac{1}{N^2T^2}\Var(\sum_{t=1}^T\btheta'\bu_t\bu_t'\bw_k  )
\cr
&\leq &\frac{C}{N^2T^2}\|\btheta\|_1^2 \max_{j,i\leq N}\sum_{ q,v\leq N}\sum_{t,s\leq T}|\Cov(u_{it}u_{qt},u_{js}u_{vs} )|\leq \frac{C|J|_0^2}{NT}.
  \end{eqnarray*}
  Also,
  $\| \frac{1}{NT}\btheta' \E\bU\bU'\bW\|
  \leq \max_{j\leq N}\sum_k|w_{k,j}| \|\btheta\|_1\| \frac{1}{TN} \E\bU\bU'\|_1
  \leq O(\frac{|J|_0}{N}).
  $ Also,
  \begin{eqnarray*}
&&\E\|\frac{1}{T}\btheta'\bU\bF \|^2
= \frac{1}{T^2} \tr\E \bF'\E(\bU'\btheta\btheta'\bU|\bF)\bF\leq \frac{C}{T}
 \|\E(\bU'\btheta\btheta'\bU|\bF)\|_1\cr
 &\leq&\frac{C}{T}\max_t\sum_{s=1}^T|\E( \btheta'\bu_t\bu_s'\btheta|\bF)|
 \leq \frac{C}{T}\max_t\sum_{s=1}^T\|\E(\bu_t\bu_s'|\bF)\|_1 \|\btheta\|_1\|\btheta\|_\infty\leq\frac{C|J|_0}{T}.
  \end{eqnarray*}

         (v) Since $ \widehat\balpha_g=(\widehat\bF'\widehat\bF)^{-1}\widehat\bF'\bG$,  simple calculations using Proposition  \ref{la.2} yield
\begin{eqnarray*}
 \widehat\balpha_g-\bH^{+'}\balpha_g &=&(\widehat\bF'\widehat\bF)^{-1}\widehat\bF'\bG-\bH^{+'}\balpha_g \cr
&=&(\widehat\bF'\widehat\bF)^{-1}\bE'\bvarepsilon_{g}-(\widehat\bF'\widehat\bF)^{-1}\bE'\bE\bH^{+'}\balpha_g+ (\widehat\bF'\widehat\bF)^{-1}\bE'\bU'\btheta
+O_P(\sqrt{\frac{|J|_0}{T}}) \cr
&=&  |J|_0O_P(1+\sqrt{\frac{N}{T}})+O_P(\nu_{\min}^{-1})\cr
\bH' ( \widehat\balpha_g-\bH^{+'}\balpha_g )&=&
\bH' (\widehat\bF'\widehat\bF)^{-1}\bE'\bvarepsilon_{g}-\bH' (\widehat\bF'\widehat\bF)^{-1}\bE'\bE\bH^{+'}\balpha_g+ \bH' (\widehat\bF'\widehat\bF)^{-1}\bE'\bU'\btheta
+O_P(\sqrt{\frac{|J|_0}{T}})\cr
&=&  O_P(\nu_{\min}^{-1}\frac{|J|_0}{N} +\sqrt{\frac{|J|_0}{T}}+ \nu_{\min} ^{-2}\frac{1}{N}).
\end{eqnarray*}

 \end{proof}

\begin{lem}\label{lc.2} Suppose $|J|_0=o(N\nu_{\min}^2)$. For any $R\geq r$
\\
(i) $\frac{1}{T}\| \bP_{\widehat\bF} \bU'\btheta\|^2=O_P(\frac{|J|_0^2}{N}+\frac{|J|_0^2}{T} +\frac{|J|_0^{3/2}}{\nu_{\min}N\sqrt{T}}) $,
$\frac{1}{T}\|\bP_{\widehat\bF} \bvarepsilon_g\|^2=O_P(\frac{1}{T})$,\\
(ii)   $\| \frac{1}{T}(\widehat\bU-\bU)\bvarepsilon_g\|_\infty=O_P(\frac{\nu^{-1}_{\min}}{\sqrt{NT}}+\frac{\sqrt{\log N}}{T} ),
$
and $\| \frac{1}{T} \widehat\bU \bvarepsilon_g\|_\infty=O_P( \sqrt{\frac{\log N}{T}})
=\|\frac{1}{T}     \widehat  \bU\bvarepsilon_y\|_\infty$
\\
(iii)  $\lambda_{\min}(\frac{1}{T}\widehat\bU_{\widehat J}\widehat\bU_{\widehat J}')>c_0$ with probability approaching one. $\frac{1}{T}\|\bP_{\widehat\bU_{\widehat J}}   \bvarepsilon_g\|^2=O_P(\frac{|J|_0\log N}{T})=\frac{1}{T}\|\bP_{\widehat\bU_{\widehat J}}   \bvarepsilon_y\|^2$. \\
 (iv)  $\frac{1}{T}\|  ( \widehat\bU- \bU)'\btheta\|^2=O_P(\frac{|J|_0^2+\nu^{-2}_{\min}}{N}+\frac{|J|_0^2}{T} +\frac{\nu_{\min}^{-1}|J|_0^{3/2}}{N\sqrt{T}})$,  $\frac{1}{T}   \bE'   \bP_{\widehat\bF}   \bvarepsilon_y=O_P(\frac{1}{\sqrt{NT}})$
, \\ $\frac{1}{T}  \btheta'  \bU\bP_{\widehat \bF}   \bvarepsilon_y=   O_P(\frac{|J|_0}{T}+\frac{|J|_0}{\sqrt{NT}} + \frac{ \nu_{\min}^{-1/2}|J|_0^{3/4}}{\sqrt{N}T^{3/4}})$.
\end{lem}

 \begin{proof}

 (i) By Lemma \ref{lc.0} (vi) and Proposition \ref{la.2},
 \begin{eqnarray*}
 \frac{1}{T}\| \bP_{\widehat\bF} \bU'\btheta\|^2
& =&\frac{1}{T}\btheta'\bU  \bE(\widehat\bF'\widehat\bF)^{-1}\bE'\bU'\btheta
+\frac{2}{T}\btheta'\bU  \bE(\widehat\bF'\widehat\bF)^{-1}\bH\bF'\bU'\btheta
\cr
&&
+\frac{1}{T}\btheta'\bU  \bF\bH'(\widehat\bF'\widehat\bF)^{-1}\bH\bF'\bU'\btheta\cr
&\leq & O_P(\frac{|J|_0^2}{N}+\frac{|J|_0^2}{T} +\frac{|J|_0^{3/2}}{\nu_{\min}N\sqrt{T}}) ,\cr
\frac{1}{T}\|\bP_{\widehat\bF} \bvarepsilon_g\|^2&=&
\frac{1}{T}\bvarepsilon_g' \bE(\widehat\bF'\widehat\bF)^{-1}\bE' \bvarepsilon_g
+\frac{2}{T}\bvarepsilon_g' \bE(\widehat\bF'\widehat\bF)^{-1}\bH\bF' \bvarepsilon_g
+\frac{1}{T}\bvarepsilon_g'  \bF\bH(\widehat\bF'\widehat\bF)^{-1}\bH\bF' \bvarepsilon_g\cr
&\leq& O_P(\frac{N}{NT})
+O_P(\frac{1}{\sqrt{NT}})\frac{\nu^{-1}_{\min}}{\sqrt{T}}+O_P(\frac{1}{T})=O_P(\frac{1}{T}).
 \end{eqnarray*}

 (ii)
By (\ref{eqa.2add})
\begin{eqnarray*}
 \frac{1}{T}  (\bU-\widehat\bU) \bvarepsilon_g
&=&- \frac{1}{T} \bB\bH^+ \bE' \bE(\widehat\bF'\widehat\bF)^{-1}\bE'\bvarepsilon_g
- \frac{1}{T} \bB\bH^+ \bE'  \bF\bH'(\widehat\bF'\widehat\bF)^{-1}\bE'\bvarepsilon_g+  \frac{1}{T} \bU\bE(\widehat\bF'\widehat\bF)^{-1}\bE'\bvarepsilon_g
\cr
&&- \frac{1}{T} \bB\bH^+ \bE' \bE(\widehat\bF'\widehat\bF)^{-1}\bH\bF'\bvarepsilon_g
- \frac{1}{T} \bB\bH^+ \bE'  \bF\bH'(\widehat\bF'\widehat\bF)^{-1}\bH\bF'\bvarepsilon_g
+\frac{1}{T} \bU\bE(\widehat\bF'\widehat\bF)^{-1}\bH\bF'\bvarepsilon_g \cr
&&
+ \frac{1}{T} \bU \bF\bH'(\widehat\bF'\widehat\bF)^{-1}\bE'\bvarepsilon_g
+\frac{1}{T} \bU \bF\bH'(\widehat\bF'\widehat\bF)^{-1}\bH\bF'\bvarepsilon_g
+ \frac{1}{T} \bB\bH^+\bE'\bvarepsilon_g.
 \end{eqnarray*}
 So  by Lemmas \ref{la.1}  and $ \| \frac{1}{T} \bU\bE\|_\infty=O_P(\sqrt{\frac{\log N}{TN}}+\frac{1}{N})$,
$
\| \frac{1}{T}(\widehat\bU-\bU)\bvarepsilon_g\|_\infty=O_P(\frac{\nu^{-1}_{\min}}{\sqrt{NT}}+\frac{\sqrt{\log N}}{T} ).
 $

Also,  with  $\| \frac{1}{T}  \bU \bvarepsilon_g\|_\infty=O_P( \sqrt{\frac{\log N}{T}})
 $  we have  $\| \frac{1}{T} \widehat \bU \bvarepsilon_g\|_\infty=O_P( \sqrt{\frac{\log N}{T}}) .
 $    The proof for  $\|\frac{1}{T}     \widehat  \bU\bvarepsilon_y\|_\infty$ is the same.

 (iii)   First, it follows from Lemma \ref{lc.0} that
 $
  \|\frac{1}{T}\widehat\bU\widehat\bU'- \frac{1}{T} \bU \bU'\|_\infty \leq  O_P(\frac{\log N}{T}+\frac{\nu^{-2}_{\min}}{N}).
$

   Also     by Proposition \ref{proc.1},  $|\widehat J|_0=O_P(|J|_0)$.  Then  with probability approaching one,
\begin{eqnarray*}
 \lambda_{\min}(\frac{1}{T}\widehat\bU_{\widehat J}\widehat\bU_{\widehat J}')&\geq&
 \lambda_{\min}(\frac{1}{T} \bU_{\widehat J} \bU_{\widehat J}')
 - \|\frac{1}{T}\widehat\bU\widehat\bU'- \frac{1}{T} \bU \bU'\|_\infty|\widehat J|_0\cr
 &\geq& \phi_{\min} -O_P(\frac{\log N}{T}+\frac{\nu^{-2}_{\min}}{N})|J|_0\geq c\cr
        \frac{1}{T}\| \bP_{\widehat\bU_{\widehat J}}   \bvarepsilon_g
\|^2&=& \frac{1}{T}\bvarepsilon_g' \widehat\bU_{\widehat J}'(\widehat\bU_{\widehat J}\widehat\bU_{\widehat J}')^{-1}\widehat\bU_{\widehat J}\bvarepsilon_g
\leq \|\frac{1}{T}\bvarepsilon_g' \widehat\bU_{\widehat J}'\|^2\lambda^{-1}_{\min}(\frac{1}{T}\widehat\bU_{\widehat J}\widehat\bU_{\widehat J}')\cr
&\leq&c\|\frac{1}{T}\bvarepsilon_g' \widehat\bU '\|^2_\infty|\widehat J|_0
\leq O_P(\frac{|J|_0\log N}{T}).
        \end{eqnarray*}
 $\frac{1}{T}\|\bP_{\widehat\bU_{\widehat J}}   \bvarepsilon_y\|^2$ follows from the same proof.

 (iv) Recall that $\|\balpha_g'\|=\|\btheta'\bB\|<C$.  By part (i) and Lemma \ref{lc.0},
\begin{eqnarray*}
        \frac{1}{T}\| \btheta' ( \widehat\bU- \bU)\|^2
      & \leq&        \frac{1}{T}\| \btheta'       \bB\bH^+ \bE'   \bP_{\widehat\bF}   \|^2
       +   \frac{1}{T}\| \btheta'  \bU\bP_{\widehat \bF}  \|^2
         +   \frac{1}{T}\| \btheta' \bB\bH^+\bE'   \|^2
       \cr
       &  \leq&  O_P(\frac{|J|_0^2+\nu^{-2}_{\min}}{N}+\frac{|J|_0^2}{T} +\frac{\nu_{\min}^{-1}|J|_0^{3/2}}{N\sqrt{T}}).
 \cr
    \|  \frac{1}{T}   \bE'   \bP_{\widehat\bF}   \bvarepsilon_y\|
      &\leq &     \|  \frac{1}{T}   \bE'   \bP_{\widehat\bF} \| \|  \bP_{\widehat\bF}  \bvarepsilon_y\|=O_P(\frac{1}{\sqrt{NT}})\cr
      \frac{1}{T}  \btheta'  \bU\bP_{\widehat \bF}   \bvarepsilon_y
      &\leq&
    \frac{1}{T} \| \btheta'  \bU\bP_{\widehat \bF} \|\bP_{\widehat \bF}  \bvarepsilon_y\|
    =O_P(\frac{|J|_0}{T}+\frac{|J|_0}{\sqrt{NT}} + \frac{ \nu_{\min}^{-1/2}|J|_0^{3/4}}{\sqrt{N}T^{3/4}}).
          \end{eqnarray*}

 \end{proof}

\begin{lem}\label{lc.3} For any $R\geq r$\\
(i)  $\frac{1}{T}\| \bM_{\widehat\bU_{\widehat J}} \widehat  \bU'\btheta\|^2=O_P(|J|_0\frac{\log N}{T})$,
$\frac{1}{T}\| \bM_{\widehat\bU_{\widehat J}}   \bU'\btheta\|^2= O_P(\frac{|J|_0\log N}{T}+ \frac{|J|_0^2+\nu^{-2}_{\min}}{N}+\frac{|J|_0^2}{T} ).$\\
(ii)
 $\frac{1}{T}  \bvarepsilon_y'  \bP_{\widehat\bU_{\widehat J}}   (\widehat \bU-\bU)'\btheta= |J|_0^2\sqrt{\frac{\log N}{T}} O_P(\frac{\log N}{T}+\frac{1}{N\nu^2_{\min}}) $,\\
 $\frac{1}{T}  \bvarepsilon_y'  \bM_{\widehat\bU_{\widehat J}}   \bU'\btheta\leq  O_P(\frac{|J|_0\log N}{T}+\frac{|J|_0+\nu^{-1}_{\min}}{\sqrt{NT}}  + \frac{ \nu_{\min}^{-1/2}|J|_0^{3/4}}{\sqrt{N}T^{3/4}} + \sqrt{\frac{\log N}{T}} \frac{|J|_0^2}{N\nu^2_{\min}})$,\\
 (iii)
$\|\bP_{\widehat\bU_{\widehat J}}   \bE\| =O_P(\sqrt{\frac{|J|_0\log N}{N}}+ \frac{\sqrt{T|J|_0}}{N\nu_{\min}})$,
 $\frac{1}{T}  \bvarepsilon_y'\bP_{\widehat\bU_{\widehat J}}  \bE=O_P(\frac{|J|_0\log N}{T\sqrt{N}}+\frac{|J|_0\sqrt{\log N}}{N\nu_{\min}\sqrt{T}}).$
 \end{lem}

\begin{proof}
(i) 
First note that $\bP_{\widehat \bU_{\widehat J}}\widehat\bU'\btheta= \widehat\bU'\widehat\bm$, where
$$
\widehat\bm=(\widehat m_1,\cdots,\widehat m_N)'=\arg\min_{\bm}\|\widehat \bU'(\btheta-\bm)\|:\quad m_j=0,\text{ for } j\notin \widehat J.
$$
Thus by the definition of $\widehat\bm$, Proposition \ref{proc.1} and Lemma \ref{lc.2},
\begin{eqnarray*}
\frac{1}{T}\| \bM_{\widehat\bU_{\widehat J}}  \widehat \bU'\btheta\|^2
&=&\frac{1}{T}\|  \widehat \bU'\btheta- \widehat  \bU'\widehat\bm\|^2
\leq \frac{1}{T}\|  \widehat \bU'\btheta- \widehat  \bU'\widetilde\btheta\|^2\leq O_P(|J|_0\frac{\log N}{T})\cr
\frac{1}{T}\| \bM_{\widehat\bU_{\widehat J}}   \bU'\btheta\|^2
&\leq& O_P(\frac{|J|_0\log N}{T})
+\frac{1}{T}\|  ( \widehat\bU- \bU)'\btheta\|^2
=  O_P(\frac{|J|_0\log N+|J|_0^2}{T} +\frac{|J|_0^2+\nu^{-2}_{\min}}{N})
\end{eqnarray*}
where we used  $\frac{\nu_{\min}^{-1}|J|_0^{3/2}}{N\sqrt{T}}=O_P(\frac{|J|_0\log N}{T})$ by our assumption. 

(ii) Let  $\bDelta=\btheta-\widehat\bm$. Then $\dim(\bDelta)=O_P(|J|_0)$. Also, by Lemma \ref{lc.0},
$$
\bDelta'\frac{1}{T}(\widehat\bU\widehat\bU'-\bU\bU')\bDelta
\leq \|\bDelta\|_1^2\|\frac{1}{T}(\widehat\bU\widehat\bU'-\bU\bU')\|_\infty\leq O_P(\frac{\log N}{T}+\frac{1}{N\nu^2_{\min}})\|\bDelta\|^2|J|_0.
$$
Also, $\|\bDelta\|^2
\leq  \frac{C}{T}\|\bU'\bDelta\|^2$ due to the   spare eigenvalue condition on $\frac{1}{T}\bU\bU'$. Then $\widetilde\btheta_j=0$ for $j\notin\widehat J$ implies $\|\widehat \bU'\bDelta\|\leq \|\widehat \bU'(\btheta-\widetilde\btheta)\|$ and Proposition \ref{proc.1} implies  
\begin{eqnarray*}
\|\btheta-\widehat\bm\|_1^2
&\leq& |J|_0\|\bDelta\|^2
\leq |J|_0\frac{1}{T}\|\bU'\bDelta\|^2
\leq  |J|_0\frac{1}{T}\|\widehat \bU'\bDelta\|^2+O_P(\frac{\log N}{T}+\frac{1}{N\nu^2_{\min}})\|\bDelta\|^2|J|_0\cr
&\leq& |J|_0\frac{1}{T}\|  \widehat \bU'\btheta- \widehat  \bU'\widetilde\btheta\|^2+O_P(\frac{\log N}{T}+\frac{1}{N\nu^2_{\min}})\|\bDelta\|^2|J|_0\cr
&\leq& \frac{|J|_0^2\log N}{T}+O_P(\frac{|J|_0\log N}{T}+\frac{|J|_0}{N\nu^2_{\min}})\|\bDelta\|^2.
\end{eqnarray*}
The above implies
$\|\btheta-\widehat\bm\|_1^2\leq O_P(|J|_0^2\frac{\log N}{T})$.
Hence by Lemma \ref{lc.2},
\begin{eqnarray*}
  \frac{1}{T}  \bvarepsilon_y'  \bP_{\widehat\bU_{\widehat J}}   (\widehat \bU-\bU)'\btheta &\leq&
    \|\frac{1}{\sqrt{T}}  \bvarepsilon_y'  \bP_{\widehat\bU_{\widehat J}} \|  \|\widehat\bU  (\widehat \bU-\bU)'\btheta \|_\infty\frac{\sqrt{|J|_0}}{T} \lambda_{\min}^{-1/2}(\frac{1}{T}\widehat\bU_{\widehat J}\widehat\bU_{\widehat J}')\cr
    &\leq& |J|_0^2\sqrt{\frac{\log N}{T}} O_P(\frac{\log N}{T}+\frac{1}{N\nu^2_{\min}})  .\cr
\frac{1}{T}  \bvarepsilon_y'  \bM_{\widehat\bU_{\widehat J}}   \widehat \bU'\btheta
&=&\frac{1}{T}  \bvarepsilon_y'    \widehat  \bU'(\btheta-\widehat\bm)
\leq \|\frac{1}{T}  \bvarepsilon_y'    \widehat  \bU'\|_\infty\|\btheta-\widehat\bm\|_1
\leq O_P(\frac{|J|_0\log N}{T}).\cr
\frac{1}{T}  \bvarepsilon_y'  \bM_{\widehat\bU_{\widehat J}}     \bU'\btheta
&\leq& \frac{1}{T}  \bvarepsilon_y'  \bM_{\widehat\bU_{\widehat J}}   \widehat \bU'\btheta
+\frac{1}{T}  \bvarepsilon_y'     (\widehat \bU-\bU)'\btheta
-\frac{1}{T}  \bvarepsilon_y'  \bP_{\widehat\bU_{\widehat J}}   (\widehat \bU-\bU)'\btheta
\cr
&\leq&O_P(\frac{|J|_0\log N}{T})+
\frac{1}{T}  \btheta'        \bB\bH^+ \bE'   \bP_{\widehat\bF}   \bvarepsilon_y
+\frac{1}{T}  \btheta'  \bU\bP_{\widehat \bF}   \bvarepsilon_y
+\frac{1}{T}  \btheta'    \bB\bH^+\bE'    \bvarepsilon_y
\cr
&&-\frac{1}{T}  \bvarepsilon_y'  \bP_{\widehat\bU_{\widehat J}}   (\widehat \bU-\bU)'\btheta\cr
&\leq& O_P(\frac{|J|_0\log N}{T}+\frac{|J|_0+\nu^{-1}_{\min}}{\sqrt{NT}}  + \frac{ \nu_{\min}^{-1/2}|J|_0^{3/4}}{\sqrt{N}T^{3/4}} + \sqrt{\frac{\log N}{T}} \frac{|J|_0^2}{N\nu^2_{\min}}).
\end{eqnarray*}
(iii) By Lemma \ref{lc.0},
\begin{eqnarray*}
\|\bP_{\widehat\bU_{\widehat J}}   \bE\| &\leq&\| \widehat\bU_{\widehat J}'(  \frac{1}{T}\widehat\bU_{\widehat J}\widehat\bU_{\widehat J}')^{-1}\| \frac{1}{T} \|\widehat\bU \bE \|_\infty\sqrt{|J|_0}
\leq O_P(\sqrt{\frac{|J|_0\log N}{N}}+ \frac{\sqrt{T|J|_0}}{N\nu_{\min}})
\cr
\| \frac{1}{T}  \bvarepsilon_y'\bP_{\widehat\bU_{\widehat J}}  \bE\|
 &\leq&\| \frac{1}{T}  \bvarepsilon_y'\bP_{\widehat\bU_{\widehat J}} \| \|\bP_{\widehat\bU_{\widehat J}}   \bE\| =O_P(\frac{|J|_0\log N}{T\sqrt{N}}+\frac{|J|_0\sqrt{\log N}}{N\nu_{\min}\sqrt{T}})
 \end{eqnarray*}

\end{proof}

 \begin{lem}\label{lc.4}
For any $R\geq r$, \\
(i) $\frac{1}{T}\|\widehat\bvarepsilon_g-\bvarepsilon_g\|^2= O_P(\frac{|J|_0^2+|J|_0\log N}{T}+ \frac{|J|_0^2+\nu^{-2}_{\min}}{N} +\frac{|J|_0^{3/2}}{\nu_{\min}N\sqrt{T}} )
=\frac{1}{T}\|\widehat\bvarepsilon_y-\bvarepsilon_y\|^2.
$\\
(ii)
$
\frac{1}{T}  \bvarepsilon_y'( \widehat\bvarepsilon_g- \bvarepsilon_g)
=O_P(\frac{|J|_0\log N}{T}+\frac{|J|_0+\nu^{-1}_{\min}}{\sqrt{NT}}  + \frac{ \nu_{\min}^{-1/2}|J|_0^{3/4}}{\sqrt{N}T^{3/4}} + \sqrt{\frac{\log N}{T}} \frac{|J|_0^2}{N\nu^2_{\min}}).
$
 The same rate applies to
$
 \frac{1}{T}  \bvarepsilon_g'(  \widehat \bvarepsilon_g-  \bvarepsilon_g)
 $
, $
 \frac{1}{T} \bfeta'(  \widehat\bvarepsilon_g-\bvarepsilon_g)
$,
$
\frac{1}{T}  \bvarepsilon_g'( \widehat\bvarepsilon_y- \bvarepsilon_y)
$
and
$
 \frac{1}{T}  \bvarepsilon_y'(  \widehat \bvarepsilon_y-  \bvarepsilon_y).
 $\\
 (iii) $\frac{1}{T}\widehat\bvarepsilon_g'\widehat\bvarepsilon_g= \frac{1}{T} \bvarepsilon_g'\bvarepsilon_g+o_P(1)$.

 \end{lem}

\begin{proof}
Note that $  \widehat \bvarepsilon_g=\bM_{\widehat\bU_{\widehat J}} \bM_{\widehat\bF} \bG$
and $\bG=\bF\balpha_g+\bU'\btheta+\bvarepsilon_g$. Also,  $\widehat\bU=\bX\bM_{\widehat\bF} $ implies
$$
 \bP_{\widehat\bU_{\widehat J}} \bP_{\widehat\bF} =0, \text{ and } \bM_{\widehat\bU_{\widehat J}} \bM_{\widehat\bF} =
 \bM_{\widehat\bF} -\bP_{\widehat\bU_{\widehat J}}  .
$$
Recall that $\bH^+\bH=\bI$ and $\widehat\bF=\bF\bH'+\bE$,
hence straightforward calculations yield
\begin{eqnarray}\label{ec.1}
\widehat \bvarepsilon_g-\bvarepsilon_g&=&
 \bM_{\widehat\bU_{\widehat J}}   \bU'\btheta
 -  \bP_{\widehat\bF} \bU'\btheta
+
\bM_{\widehat\bU_{\widehat J}} \bM_{\widehat\bF} \bF\balpha_g
 -\bP_{\widehat\bU_{\widehat J}}   \bvarepsilon_g
  -\bP_{\widehat\bF} \bvarepsilon_g\cr
  &=&
   \bM_{\widehat\bU_{\widehat J}}   \bU'\btheta
 -  \bP_{\widehat\bF} \bU'\btheta
 -\bP_{\widehat\bU_{\widehat J}}   \bvarepsilon_g
  -\bP_{\widehat\bF} \bvarepsilon_g
  -
 (  \bI-  \bP_{\widehat\bF} -\bP_{\widehat\bU_{\widehat J}}  )\bE\bH^{+'}\balpha_g.
\end{eqnarray}
It  follows from Lemmas \ref{lc.2}, \ref{lc.3}  that $\frac{1}{T}\|\widehat\bvarepsilon_g-\bvarepsilon_g\|^2= O_P(\frac{|J|_0^2+|J|_0\log N}{T}+ \frac{|J|_0^2+\nu^{-2}_{\min}}{N} +\frac{|J|_0^{3/2}}{\nu_{\min}N\sqrt{T}} ).$ The proof for $\frac{1}{T}\|\widehat\bvarepsilon_g-\bvarepsilon_g\|^2$ follows similarly.

(ii) It follows from (\ref{ec.1}) and Lemmas \ref{lc.2}  \ref{lc.3} that
\begin{eqnarray*}
\frac{1}{T}  \bvarepsilon_y'( \widehat\bvarepsilon_g- \bvarepsilon_g)
&=&  \frac{1}{T}  \bvarepsilon_y'  \bM_{\widehat\bU_{\widehat J}}   \bU'\btheta
 - \frac{1}{T}  \bvarepsilon_y' \bP_{\widehat\bF} \bU'\btheta
 -\frac{1}{T}  \bvarepsilon_y'\bP_{\widehat\bU_{\widehat J}}   \bvarepsilon_g
  -\frac{1}{T}  \bvarepsilon_y'\bP_{\widehat\bF} \bvarepsilon_g\cr
  && -   \frac{1}{T}  \bvarepsilon_y'\bE\bH^{+'}\balpha_g
  - \frac{1}{T}  \bvarepsilon_y'\bP_{\widehat\bF}\bE\bH^{+'}\balpha_g
  -\frac{1}{T}  \bvarepsilon_y'\bP_{\widehat\bU_{\widehat J}}  \bE\bH^{+'}\balpha_g\cr
  &\leq&  O_P(\frac{|J|_0\log N}{T}+\frac{|J|_0+\nu^{-1}_{\min}}{\sqrt{NT}}  + \frac{ \nu_{\min}^{-1/2}|J|_0^{3/4}}{\sqrt{N}T^{3/4}} + \sqrt{\frac{\log N}{T}} \frac{|J|_0^2}{N\nu^2_{\min}}).
\end{eqnarray*}
The same proof applies to other terms as well.

(iii) It follows from parts (i)  that all these terms are $o_P(1)$, given that
$|J|_0^2=o(\min\{T,N\})$, $|J|_0\log N=o(T)$.

\end{proof}

\subsubsection{The case $r=0$: there are no factors.}
\begin{proof}
In this case
 $
\bx_t=\bu_t.
 $
 And we have
 $$
 \widehat\bF= \frac{1}{N}\bX'\bW=\frac{1}{N}\bU'\bW:=\bE.
 $$
 Then $
\lambda_{\min}( \frac{1}{T}\widehat\bF'\widehat\bF)
=\lambda_{\min}( \frac{1}{T}\bE'\bE)\geq \frac{c}{N}
 $
 with probability approaching one, still by Lemma \ref{la.1}. Hence $  \frac{1}{T}\widehat\bF'\widehat\bF$ is still invertible.  In addition, $\widehat\bU=\bX\bM_{\widehat\bF}$ implies
 $
\bU- \widehat\bU=\bU\bP_{\bE} .
 $
 Also,
 \begin{eqnarray*}
 y_t&=&\bgamma'\bu_t+\bvarepsilon_{y,t}\cr
 \bg_t&=&\btheta'\bu_t+\bvarepsilon_{g,t}\cr
 \bvarepsilon_{y,t}&=&\bbeta'\bvarepsilon_{g,t}+  \eta_t
 \end{eqnarray*}
 Hence $\balpha_g=\balpha_y=0$. Then 
 $
\frac{1}{T} \widehat\bF'\widehat\bF = \frac{1}{T} \bE'\bE=\frac{1}{N^2}\bW'\Cov(\bu_t)\bW +O_P(\frac{1}{N\sqrt{T}}).
 $
 Hence with probability approaching one $\lambda_{\min}(\frac{1}{T} \widehat\bF'\widehat\bF)\geq cN^{-1} $. 
 In addition, $\widehat\balpha_y=(\bE'\bE)^{-1}\bE'\bU'\bgamma + (\bE'\bE)^{-1}\bE'\bepsilon_y$ implies
 $\frac{1}{T}\sum_{t=1}^T(\widehat\balpha_y'\widehat\bff_t)^2=O_P(\frac{ |J|_0^2}{N}+\frac{ |J|_0^2}{ { T}}).$
 
 As for the ``score" $\max_i|\frac{1}{T}\sum_t(\varepsilon_{g,t}+d_t)\widehat u_{it} |$ in the proof of Proposition \ref{proc.1}, note that 
 \begin{eqnarray}\label{eqa.6}
 \max_{i,j\leq N}|\frac{1}{T}\sum_t(\widehat u_{it}\widehat u_{jt}- u_{it}u_{jt})|&\leq& \frac{3}{T}\|\bU\bP_\bE\bU'\|_\infty=O_P(\frac{1}{N} + \frac{ \log N}{T})\cr
\max_{i\leq N}| \frac{1}{T}\sum_t \widehat\balpha_y'\widehat\bff_t\widehat u_{it}|&=&O_P(\frac{|J|_0}{N} + \frac{|J|_0\log N}{T})
\cr
 \max_{i\leq N}| \frac{1}{T}\sum_t \widehat u_{it}(\bu_t-\widehat \bu_t)'\btheta |&=& \frac{1}{T}\|\bU\bP_\bE\bU'\|_\infty O_P(|J|_0)=O_P(\frac{|J|_0}{N} + \frac{|J|_0\log N}{T}) \cr
 \max_{i\leq N}| \frac{1}{T}\sum_t \widehat u_{it}\varepsilon_{g,t}|&=& O_P(\frac{\sqrt{\log N}}{T}+\frac{1}{\sqrt{TN}}).
 \end{eqnarray}
 As for the residual, note that $  \widehat \bvarepsilon_g=\bM_{\widehat\bU_{\widehat J}} \bM_{\bE} \bG$
and $\bG= \bU'\btheta+\bvarepsilon_g$.  Then 
 $$
\widehat \bvarepsilon_g-\bvarepsilon_g=
    \bM_{\widehat\bU_{\widehat J}}   \bU'\btheta
 -  \bP_{\bE} \bU'\btheta
 -\bP_{\widehat\bU_{\widehat J}}   \bvarepsilon_g
  -\bP_{\bE} \bvarepsilon_g.
 $$
All the proofs in Section \ref{sec:thereisfactor} carry over. In fact, all terms involving $\balpha_g, \bH$ and $\bH^+$ can be  set to zero.

 In addition, in the case $R=r=0$, the setting/estimators are  the same as in \cite{belloni2014inference}.
 \end{proof}

\subsubsection{Proof of Corollary \ref{cor3.1}.}

\begin{proof}
The corollary immediately follows from Theorem \ref{t3.2}. If there exist a pair $(r, R)$ that violate the conclusion of the corollary, then it also violates the conclusion of Theorem \ref{t3.2}.  This finishes the proof.
\end{proof}

\subsection{Proof of Theorem \ref{t2.2}}
\begin{proof}
	In the proof of Theorem \ref{t2.2} we assume $R\geq r$.
	
	(i)  When $r>0$,
	by Lemma \ref{la.1new}, 
	\begin{eqnarray*}
		\max_{i,j\leq N}|\frac{1}{T}\sum_t(\widehat u_{it}\widehat u_{jt}- u_{it}u_{jt})|\leq \|\frac{1}{T}\widehat\bU\widehat\bU'- \frac{1}{T} \bU \bU'\|_\infty\leq O_P(\frac{\log N}{T}+\frac{1}{N\nu_{\min}^2})  .
	\end{eqnarray*}
	When  $r=0$ and $R>0$, by (\ref{eqa.6}), $\max_{i,j\leq N}|\frac{1}{T}\sum_t(\widehat u_{it}\widehat u_{jt}- u_{it}u_{jt})|\leq   O_P(\frac{\log N}{T}+\frac{1}{N\nu_{\min}^2})  .
$

In both cases, part (i) implies, for $\nu_{\min}^2\gg \frac{1}{\sqrt{N}}$ or $\nu_{\min}^2\gg\frac{1}{N}\sqrt{\frac{T}{\log N}}$,
	\begin{eqnarray*}
		\max_{i,j\leq N} |s_{u,ij}-\E u_{it}u_{jt}|
		& \leq&   \max_{i,j\leq N} |\frac{1}{T}\sum_t\widehat u_{it}\widehat u_{jt}-  u_{it}u_{jt}|
		+\max_{i,j\leq N} |\frac{1}{T}\sum_t  u_{it}u_{jt}- \E u_{it}u_{jt}   | \cr
		&\leq&  O_P(\sqrt{\frac{\log N}{T}}+\frac{1}{N\nu_{\min}^2})
		= O_P(\sqrt{\frac{\log N}{T}}+\frac{1}{\sqrt{N}}).
	\end{eqnarray*}
	where $\max_{i,j\leq N} |\frac{1}{T}\sum_t  u_{it}u_{jt}- \E u_{it}u_{jt}   | =O_P(\sqrt{\frac{\log N}{T}})$.

	Given this convergence, the convergence of $\widehat\bSigma_u$ and $\widehat\bSigma_u^{-1}$ in (ii)(iii) then follows from the same proof of Theorem A.1 of \cite{POET}. We thus omit it for brevity.
	Finally, the case $r=R=0$ is the usual case of sparse thresholding as in \cite{Bickel08a}.
\end{proof}

 \subsection{Proof  of Theorem \ref{t3.3} }

 \begin{proof}
 First note that when $R=r$, by (\ref{eqa.2new})
$$
 \|(\frac{1}{T}\widehat\bF'\widehat\bF)^{-1}-(\frac{1}{T}\bH\bF'\bF\bH')^{-1}\| \leq O_P (\frac{1}{N}+\frac{\nu_{\max}(\bH)}{\sqrt{TN}})\frac{1}{\nu^4_{\min}(\bH)} .
 $$
 Also by the proof of Theorem \ref{t2.1} for  $
 \|(\frac{1}{T}\widehat\bF'\widehat\bF)^{-1}\|
 +\|(\frac{1}{T}\bH\bF'\bF\bH')^{-1}\|\leq \frac{c}{\nu^2_{\min}(\bH)} $.
 Because $\bP_{\widehat\bF}-\bP_\bG
 =\bE(\widehat\bF'\widehat\bF)^{-1}\bH\bF'
 +\bF\bH'[(\widehat\bF'\widehat\bF)^{-1}-(\bH\bF'\bF\bH')^{-1}]\bH\bF'
 +\widehat\bF(\widehat\bF'\widehat\bF)^{-1}\bE'$, we have
  \begin{eqnarray*}
 	\|\bP_{\widehat\bF}-\bP_\bG\|_F^2
 	&=&\tr  (\widehat\bF'\widehat\bF)^{-1}\bH\bF'\bF\bH'(\widehat\bF'\widehat\bF)^{-1}\bE'   \bE
 		+  \tr (\widehat\bF'\widehat\bF)^{-1}\bE'   \bE   \cr
 	&&+ 2 \tr  (\widehat\bF'\widehat\bF)^{-1}\bH\bF'  \bF\bH'[(\widehat\bF'\widehat\bF)^{-1}-(\bH\bF'\bF\bH')^{-1}]\bH\bF'  \bE \cr
 		&&+  \tr [(\widehat\bF'\widehat\bF)^{-1}-(\bH\bF'\bF\bH')^{-1}]\bH\bF'\bF\bH'[(\widehat\bF'\widehat\bF)^{-1}-(\bH\bF'\bF\bH')^{-1}]\bH\bF'   \bF\bH'  \cr
 					&&+  2\tr  \bF\bH'[(\widehat\bF'\widehat\bF)^{-1}-(\bH\bF'\bF\bH')^{-1}]\bH\bF'   \bE(\widehat\bF'\widehat\bF)^{-1}\widehat\bF'  \cr
 				&&+ 2 \tr  (\widehat\bF'\widehat\bF)^{-1}\bH\bF' \bE(\widehat\bF'\widehat\bF)^{-1}\bE' \bE  \cr
 					&&+ 2 \tr (\widehat\bF'\widehat\bF)^{-1}\bH\bF' \bE(\widehat\bF'\widehat\bF)^{-1}\bH\bF' \bE   \cr
 					&=& 2 \tr \bH^{'-1}( \bF' \bF)^{-1}\bH^{-1}\bE'   \bE  +O_P(\frac{1}{TN\nu^2_{\min}}+\frac{1}{N^2\nu_{\min}^4}+\frac{1}{N\sqrt{NT}\nu^3_{\min}}). 
 \end{eqnarray*}

 Write $X:= 2 \tr \bH^{'-1}( \bF' \bF)^{-1}\bH^{-1}\bE'   \bE  =\tr(\bA\frac{1}{T}\bE'\bE) $ and $\bA:=2\bH^{'-1}(\frac{1}{T} \bF' \bF)^{-1}\bH^{-1}$.
 Now
 $$
 \MEAN= \E( X|\bF,\bW)=\tr\bA \frac{1}{N^2}\bW'(\E\bu_t\bu_t'|\bF)\bW= \tr\bA \frac{1}{N^2}\bW' \bSigma_u\bW.
 $$
 We note that $
 \Var(X|\bF)=\frac{1}{TN^2}\sigma^2
 $ and that   $ N\sqrt{T}\frac{(X- \MEAN)}{  \sigma}\overset{d}{\longrightarrow}\mathcal N(0,1)$ due to the serial indepence of $\bu_t\bu_t'$ conditionally on $\bF$ and that $\E\|\frac{1}{\sqrt{N}}\bW'\bu_t\|^4<C$.  In addition,
 Lemma \ref{lc.5} below shows that with $\widehat{\MEAN}= \tr\widehat\bA \frac{1}{N^2}\bW'\widehat \bSigma_u\bW,$ and $\widehat \bA=2(\frac{1}{T}\widehat\bF'\widehat\bF)^{-1}$, we have   
 $$
(\widehat\MEAN-\MEAN)N\sqrt{T}=o_P(1).
 $$
  Also, the same lemma shows $\widehat\sigma^2\overset{P}{\longrightarrow}\sigma^2.$ As a result
 $$
 \frac{\|\bP_{\widehat\bF}-\bP_\bG\|_F^2-\widehat\MEAN}{\frac{1}{N\sqrt{T}}\widehat\sigma} = \frac{X- \MEAN}{\frac{1}{N\sqrt{T}} \sigma} +o_P(1)\overset{d}{\longrightarrow}\mathcal N(0,1). $$
 given that $\sigma>0$, $\sqrt{T}=o(N)$.

 \end{proof}

 \begin{lem}
 \label{lc.5}     Suppose $R=r$. Let $g_{NT}:= \nu_{\min}^{-2}\frac{1}{N}+\frac{\log N}{T} $.\\
(i) $\widehat\MEAN-\MEAN=O_P(\frac{g_{NT}^2}{N^2\nu^2_{\min}}) \sum_{\sigma_{u,ij}\neq 0} 1+O_P(\frac{1}{N^2\nu_{\min}^4}+\frac{1}{N\sqrt{NT}\nu^3_{\min}}).$\\
(ii) $\widehat\sigma^2\overset{P}{\longrightarrow}\sigma^2$.
 \end{lem}

 \begin{proof}

   By lemma \ref{la.1new},  $$
\max_{ij}|\frac{1}{T}\sum_t  u_{it}(\widehat u_{jt}- u_{jt})|\leq  O_P(g_{NT}).
$$

 (i)    Recall $\bA:=2\bH^{'-1}(\frac{1}{T} \bF' \bF)^{-1}\bH^{-1}.$ Note that $\|\bA\|=O_P( \frac{1}{\nu^2_{\min}(\bH)} )$. We now bound $\frac{1}{N}\bW'(\widehat\bSigma_u-\bSigma_u)\bW$.  For   simplicity we focus on the   case $r=R=1$ and hard-thresholding estimator. The proof of SCAD thresholding follows from the same argument. We have
 \begin{eqnarray*}
 \frac{1}{N}\bW'(\widehat\bSigma_u-\bSigma_u)\bW
 =\frac{1}{N}\sum_{\sigma_{u,ij}=0} w_iw_j\widehat\sigma_{u,ij}+\frac{1}{N}\sum_{\sigma_{u,ij}\neq 0} w_iw_j(\widehat\sigma_{u,ij}-\sigma_{u,ij}):=a_1+a_2.
 \end{eqnarray*}
 Term
 $a_1$ satisfies: for any $\epsilon>0$, when $C$ in the threshold is large enough,
 $$
\mathbb P(a_1>(NT)^{-2})\leq \mathbb P(\max_{\sigma_{u,ij}=0}|\widehat\sigma_{u,ij}|\neq0)
\leq \mathbb P( |s_{u,ij}| > \tau_{ij}  ,\text{ for some } \sigma_{u,ij}=0)   <\epsilon   .
 $$ Thus $a_1=O_P((NT)^{-2}) $. The main task is to bound $a_2=\frac{1}{N}\sum_{\sigma_{u,ij}\neq 0} w_iw_j(\widehat\sigma_{u,ij}-\sigma_{u,ij}).$
 \begin{eqnarray*}
 a_2&=&a_{21} +a_{22},\cr
 a_{21}&=&\frac{1}{N}\sum_{\sigma_{u,ij}\neq 0} w_iw_j\frac{1}{T}\sum_t (\widehat u_{it}\widehat u_{jt}-u_{it}u_{jt})\cr
  a_{22}&=&\frac{1}{N}\sum_{\sigma_{u,ij}\neq 0} w_iw_j\frac{1}{T}\sum_t ( u_{it}u_{jt}-\E u_{it}u_{jt}).
  \end{eqnarray*}
Now for $\omega_{NT}:=\sqrt{\frac{\log N}{T}}+\frac{1}{\sqrt{N}}$, by part (i),
  \begin{eqnarray*}
 a_{21}&=&\frac{1}{N}\sum_{\sigma_{u,ij}\neq 0} w_iw_j\frac{1}{T}\sum_t (\widehat u_{it}-u_{it})(\widehat u_{jt} -u_{jt})+
 \frac{2}{N}\sum_{\sigma_{u,ij}\neq 0} w_iw_j\frac{1}{T}\sum_t  u_{it}(\widehat u_{jt}- u_{jt})\cr
 &\leq& [\max_i\frac{1}{T}\sum_t(\widehat u_{it}-u_{it})^2+\max_{ij}|\frac{1}{T}\sum_t  u_{it}(\widehat u_{jt}- u_{jt})|  ] \frac{1}{N}\sum_{\sigma_{u,ij}\neq 0} 1
  \cr
  &\leq& O_P(g_{NT}^2)\frac{1}{N}\sum_{\sigma_{u,ij}\neq 0} 1.
  \end{eqnarray*}
  As for $a_{22}, $ due to    $ \frac{1}{N}\sum_{\sigma_{u,mn}\neq 0} \sum_{\sigma_{u,ij}\neq 0}|\Cov( u_{it}u_{jt} ,    u_{mt}u_{nt} )|<C$ and serial independence,
  \begin{eqnarray*}
  \Var(a_{22})&\leq&\frac{1}{N^2T^2}\sum_{s,t\leq T}  \sum_{\sigma_{u,mn}\neq 0} \sum_{\sigma_{u,ij}\neq 0}|\Cov( u_{it}u_{jt} ,
u_{ms}u_{ns} )|\cr
&\leq&\frac{1}{N^2T}  \sum_{\sigma_{u,mn}\neq 0} \sum_{\sigma_{u,ij}\neq 0}|\Cov( u_{it}u_{jt} ,    u_{mt}u_{nt} )|\leq O(\frac{1}{NT}) .
  \end{eqnarray*}
  Together $a_{2}=O_P(g_{NT}^2)\frac{1}{N}\sum_{\sigma_{u,ij}\neq 0} 1+O_P(\frac{1}{\sqrt{NT}}) $. Therefore
  $$
  \frac{1}{N}\bW'(\widehat\bSigma_u-\bSigma_u)\bW
  =O_P(g_{NT}^2)\frac{1}{N}\sum_{\sigma_{u,ij}\neq 0} 1+O_P(\frac{1}{\sqrt{NT}}).
  $$
  This implies
  \begin{eqnarray*} |\widehat\MEAN-\MEAN|&\leq&
\frac{C}{N}  \|\bA\| \|\frac{1}{N}\bW'(\bSigma_u-\widehat\bSigma_u)\bW\|
+O_P(\frac{1}{N})\|\bA-2(\frac{1}{T}\widehat\bF'\widehat\bF)^{-1}\|
\cr
&\leq& O_P(\frac{g_{NT}^2}{N^2\nu^2_{\min}}) \sum_{\sigma_{u,ij}\neq 0} 1+O_P(\frac{1}{N^2\nu_{\min}^4}+\frac{1}{N\sqrt{NT}\nu^3_{\min}}).
  \end{eqnarray*}

  (ii)  First, note that $|\sigma^2-f(\bA,\bV)|\to 0$ by the assumption. In addition, it is easy to show that $\|\widehat\bA- \bA\|=o_P(1)$ and
  $
  \|\widehat\bV-\bV\|\leq\frac{1}{N}\|\bW\|^2\|\widehat\bSigma_u-\bSigma_u\|=o_P(1).
  $
  Since $f(\bA,\bV)$ is continuous in $(\bA,\bV)$ due to the property of the normality of $\bZ_t$, we have
  $|f(\bA,\bV)-f(\widehat\bA,\widehat\bV)|=o_P(1)$. Hence $|f(\widehat\bA,\widehat\bV)-\sigma^2|=o_P(1)$. This finishes the proof since $\widehat\sigma^2:=f(\widehat\bA,\widehat\bV)$.

 \end{proof}

\small

\bibliographystyle{ims}

\end{document}